\documentclass[secthm,seceqn,amsthm,ussrhead,10pt]{amsart}
\usepackage[utf8]{inputenc}
\usepackage[english]{babel}

\usepackage{amssymb,amsmath,amsthm,amsfonts,xcolor,enumerate,hyperref,comment,longtable,cleveref}

\usepackage{times}
\usepackage{cite}
\usepackage{pdflscape}
\usepackage{ulem}
\usepackage[mathcal]{euscript}
\usepackage{tikz}
\usepackage{hyperref}
\usepackage{cancel}
\usepackage{stmaryrd}

\usetikzlibrary{arrows}

\newtheorem{theorem}{Theorem}
\newtheorem*{theoremA}{Theorem A}
\newtheorem*{theoremB}{Theorem B}

\theoremstyle{definition}
\newtheorem*{definition}{Definition}

\theoremstyle{remark}

\usepackage{stmaryrd}
\usepackage{xcolor}

\newcommand{\aut}[1]{\operatorname{\mathrm{Aut}}{(#1)}}

\newcommand{\la}{\langle}
\newcommand{\ra}{\rangle}

\newcommand{\Dl}[2]{[\Delta_{#1#2}]}
\newcommand{\nb}[1]{\nabla_{#1}}
\newcommand{\as}[1]{\alpha^\star_{#1}}

\newcommand{\0}{\theta}
\newcommand{\af}{\alpha}
\newcommand{\bt}{\beta}

\newcommand{\orb}{\mathrm {Orb}}

\begin{document}
\noindent{\Large  
The algebraic and geometric classification  of nilpotent anticommutative algebras}\footnote{
The first part of this work is supported by the Russian Science Foundation under grant 18-71-10007. 
The second part of this work was supported by CNPq   404649/2018-1;
FAPESP 	18/09299-2, 18/15712-0;
by CMUP (UID/MAT/00144/2019), which is funded by FCT with national (MCTES) and European structural funds through the programs FEDER, under the partnership agreement PT2020;
by the Funda\c{c}\~ao para a Ci\^encia e a Tecnologia (Portuguese Foundation for Science and Technology) through the project PTDC/MAT-PUR/31174/2017.} 

   \

   {\bf   Ivan Kaygorodov$^{a,b}$, Mykola Khrypchenko$^{c,d}$ \& Samuel A.\ Lopes$^{e}$   \\

    \medskip
 
    \medskip
}

{\tiny

$^{a}$ CMCC, Universidade Federal do ABC, Santo Andr\'e, Brazil

$^{b}$ Siberian Federal University, Krasnoyarsk, Russia

$^{c}$ Departamento de Matem\'atica, Universidade Federal de Santa Catarina, Florian\'opolis, Brazil

$^{d}$ Centro de Matem\'atica e Aplica\c{c}\~oes, Faculdade de Ci\^{e}ncias e Tecnologia, Universidade Nova de Lisboa, Caparica, Portugal

$^{e}$ CMUP, Faculdade de Ci\^encias, Universidade do Porto, Rua do Campo Alegre 687, 4169-007 Porto, Portugal

\

\smallskip

  \medskip

   E-mail addresses:

\smallskip
    
    Ivan Kaygorodov (kaygorodov.ivan@gmail.com) 

 \smallskip
    Mykola Khrypchenko (nskhripchenko@gmail.com)  
    
 \smallskip    
      Samuel A.\ Lopes (slopes@fc.up.pt)

}

\ 

\ 

  \medskip

\ 

\noindent {\bf Abstract.}
{\it We give algebraic and geometric classifications of $6$-dimensional complex nilpotent anticommutative algebras. Specifically, we find that, up to isomorphism, there are $14$ one-parameter families of $6$-dimensional nilpotent anticommutative algebras, complemented by $130$ additional isomorphism classes. The corresponding geometric variety is irreducible and determined by the Zariski closure of a one-parameter family of algebras. In particular, there are no rigid $6$-dimensional complex nilpotent anticommutative algebras.}

\ 

\noindent {\bf Keywords}: {\it Nilpotent algebra, anticommutative algebra, Tortkara  algebra, Malcev algebra,
Lie algebra, algebraic classification, geometric classification, central extension, degeneration.}

\ 

\noindent {\bf MSC2010}: 17D10, 17D30.

\section*{Introduction}

Anticommutativity is a frequent phenomenon in the study of geometry and physics, as it arises naturally in the context of symmetry. Two of the most prominent examples of anticommutative algebras are the exterior algebra (which is also associative) and Lie algebras (which, in general, are non-associative). We view an $n$-dimensional complex algebra as a point in $\mathbb{C}^{n^3}$, given by its structure constants relative to some fixed basis. Since anticommutativity and nilpotency are closed conditions (with respect to the Zariski topology), the subset of $n$-dimensional nilpotent anticommutative algebras forms an affine variety. Moreover, the general linear group $\mathrm{GL}_n(\mathbb{C})$ acts on $\mathbb{C}^{n^3}$ by changing the basis, and the orbits parametrize the isomorphism classes. 

In this paper, our goal is to obtain a complete algebraic and geometric description of the variety of all $6$-dimensional nilpotent anticommutative algebras over the complex field. To do so, we first determine all such $6$-dimensional algebra structures, up to isomorphism (what we call the algebraic classification), and then proceed to determine the geometric properties of the corresponding variety, namely its dimension and description of the irreducible components (the geometric classification). 

Our main results are summarized below. 

\begin{theoremA}
Up to isomorphism, the variety of $6$-dimensional complex nilpotent anticommutative algebras has infinitely many isomorphism classes, described explicitly in Appendix~\ref{appb} in terms of $14$ one-parameter families and $130$ additional isomorphism classes.
\end{theoremA}

From the geometric point of view, in many cases the irreducible components of the variety are determined by the rigid algebras, i.e., algebras whose orbit closure is an irreducible component. It is worth mentioning that this is not always the case and already in \cite{fF68} Flanigan had shown that the variety of $3$-dimensional nilpotent associative algebras has an irreducible component which does not contain any rigid algebras --- it is instead defined by the closure of a union of a one-parameter family of algebras. Here, we encounter a similar situation. Informally, although Theorem~B shows that there is no single {\it generic} $6$-dimensional nilpotent anticommutative algebra, one can see the family ${\mathbb A}_{82}(\af)$ given below as the {\it generic family} in the variety. 

\begin{theoremB}
The variety of $6$-dimensional complex nilpotent anticommutative algebras is irreducible of dimension $34$. It contains no rigid algebras and can be described as the closure of the union of $\mathrm{GL}_6(\mathbb{C})$-orbits of the following one-parameter family of algebras ($\alpha\in\mathbb{C}$):
\begin{equation*}
{\mathbb A}_{82}(\af) :\qquad   
e_1e_2=e_3, \quad e_1e_3=e_4, \quad e_2e_5=\af e_6, \quad e_3e_4=e_5, \quad e_3e_5=e_6, \quad e_4e_5=e_6 .
\end{equation*} 
\end{theoremB}

Let us mention a few similar results. In \cite{maz80}, Mazzola showed that the variety of $n$-dimensional nilpotent commutative associative algebras is irreducible for $n\leq 6$, but not so for $n=7$. For nilpotent Lie algebras, a similar phenomenon takes place: that variety is irreducible for algebras of dimension up to $6$, by \cite{GRH} and \cite{S90}, but becomes reducible in dimensions $7$ and $8$, by \cite{GAB92}.

\medskip

\paragraph{\bf Motivation and contextualization} Given algebras ${\bf A}$ and ${\bf B}$ in the same variety, we write ${\bf A}\to {\bf B}$ and say that ${\bf A}$ {\it degenerates} to ${\bf B}$, or that ${\bf A}$ is a {\it deformation} of ${\bf B}$, if ${\bf B}$ is in the Zariski closure of the orbit of ${\bf A}$ (under the aforementioned base-change action of the general linear group). The study of degenerations of algebras is very rich and closely related to deformation theory, in the sense of Gerstenhaber \cite{mG64} (see also \cite{fF68}). It offers an insightful geometric perspective on the subject and has been the object of a lot of research.
In particular, there are many results concerning degenerations of algebras of small dimensions in a  variety defined by a set of identities.
One of the main problems in this direction is a description of the irreducible components of the variety. In the case of finitely-many orbits (i.e., isomorphism classes), the irreducible components are determined by the rigid algebras --- algebras whose orbit closure is an irreducible component of the variety under consideration. For this reason, they are seen of as {\it generic algebras} in that variety.

For example, rigid algebras have been classified in the following varieties: 
$4$-dimensional Leibniz algebras \cite{ikv17}, 
$4$-dimensional nilpotent Novikov algebras \cite{kkk18},
$4$-dimensional nilpotent bicommutative algebras \cite{kpv19},
$6$-dimensional nilpotent binary Lie algebras \cite{ack},
$6$-dimensional nilpotent Tortkara  algebras \cite{gkk19}.
There are fewer works in which the full information about degenerations has been given for some variety of algebras.
This problem was solved 
for $2$-dimensional pre-Lie algebras \cite{bb09},  
for $2$-dimensional terminal algebras \cite{cfk19},
for $3$-dimensional Novikov algebras \cite{bb14},  
for $3$-dimensional Jordan algebras \cite{gkp},  
for $3$-dimensional Jordan superalgebras \cite{maria},
for $3$-dimensional Leibniz algebras \cite{ikv18}, 
for $3$-dimensional anticommutative algebras \cite{ikv18},
for $3$-dimensional nilpotent algebras \cite{fkkv},
for $4$-dimensional Lie algebras \cite{BC99},
for $4$-dimensional Lie superalgebras \cite{aleis},
for $4$-dimensional Zinbiel  algebras \cite{kppv},
for $4$-dimensional nilpotent Leibniz algebras \cite{kppv},
for $4$-dimensional nilpotent commutative algebras \cite{fkkv},
for $5$-dimensional nilpotent Tortkara algebras \cite{gkks},
for $5$-dimensional nilpotent anticommutative algebras \cite{fkkv},
for $6$-dimensional nilpotent Lie algebras \cite{S90,GRH}, 
for $6$-dimensional nilpotent Malcev algebras \cite{kpv}, 
for $2$-step nilpotent $7$-dimensional Lie algebras \cite{ale2}, 
and for all $2$-dimensional algebras \cite{kv16}.
There are many results related to the algebraic and geometric 
classification
of low-dimensional algebras in the varieties of Jordan, Lie, Leibniz and 
Zinbiel algebras;
for algebraic classifications  see, for example, \cite{ack, cfk19, degr3, usefi1, degr2, degr1, gkks, gkk,  ikm19,   ikv18, hac18, kkk18, kpv19, kv16};
for geometric classifications and descriptions of degenerations see, for example, 
\cite{ack, ale, ikm19, ale2, aleis, maria, bb14, BC99, cfk19, gkks, gkk19, gkp,  GRH, GRH2, ikv17, ikv18, kkk18, kpv19, kppv, kpv, kv16,kv17, S90,  gorb93, gorb98, khud15, khud13, maz79, maz80, klpp}.

\medskip 

\paragraph{\bf Methods}
Our first and most comprehensive step is the classification, up to isomorphism, of all $6$-dimensional nilpotent anticommutative algebras. 
This will be achieved via the method of central extensions stemming from \cite{ss78}, \cite{ha16} and \cite{hac16}. Every nilpotent algebra can be constructed as a central extension of an algebra $\mathbf{A}$ of smaller dimension, and the isomorphism classes of the extensions are controlled by a suitable action of the automorphism group of $\mathbf{A}$ on the Grassmannian space based on the second cohomology of $\mathbf{A}$ with trivial coefficients.

Skjelbred and Sund \cite{ss78} used central extensions of Lie algebras to classify nilpotent Lie algebras. In later works, using the same method, all non-Lie central extensions of $4$-dimensional Malcev algebras \cite{hac16},
all non-associative central extensions of $3$-dimensional Jordan algebras \cite{ha17},
all anticommutative central extensions of $3$-dimensional anticommutative algebras \cite{cfk182} and all central extensions of $2$-dimensional algebras \cite{cfk18} were described, to mention but a few. Related work on central extensions can be found, for example, in \cite{zusmanovich,kkl18,omirov}.

The class of anticommutative algebras includes all Malcev (in particular, all Lie) and all Tortkara algebras. Concerning the latter, the algebraic and geometric classifications of $6$-dimensional nilpotent Tortkara algebras have been completed in \cite{gkk} and \cite{gkk19}, respectively. We will rely on this work; in particular, we will be able to proceed in our algebraic and geometric classifications modulo the class of Tortkara algebras. We will also rely on the classification of 
$4$-dimensional nilpotent anticommutative algebras in \cite{cfk182},
$5$-dimensional nilpotent anticommutative algebras in \cite{fkkv},
$6$-dimensional nilpotent Malcev algebras in \cite{hac18} and
$6$-dimensional nilpotent Tortkara algebras in \cite{gkk,gkk19,gkks}.

\medskip 

\paragraph{\bf Organization of the paper}
We will work over the base field $\mathbb{C}$ of complex numbers. In Section~\ref{S:alg}, we use the action of automorphism groups of algebras of smaller dimension on central extensions to determine the distinct isomorphism classes of $6$-dimensional nilpotent anticommutative algebras, yielding Theorem~A. Then, in Section~\ref{S:geo}, we obtain the corresponding geometric description. Our main result, Theorem B, says that the variety defined by these algebras is irreducible and determined by a one-parameter family of pairwise non-isomorphic algebras.

\section{The algebraic classification of $6$-dimensional  nilpotent anticommutative algebras}\label{S:alg}

\subsection{The algebraic classification of nilpotent anticommutative algebras}\label{algcl}
Let ${\bf A}$ be an anticommutative  algebra, ${\bf V}$ a vector space and ${\rm Z}^{2}\left( {\bf A},{\bf V}\right)\cong {\rm Hom}(\wedge^2{\bf A},\bf V)$ the space of skew-symmetric  bilinear maps $\theta :{\bf A}\times 
{\bf A}\longrightarrow {\bf V}.$ For $f\in{\rm Hom}({\bf A},{\bf V})$, we define $\delta f\in {\rm Z}^{2}\left( {\bf A},{\bf V}\right)$ by the equality $\delta f\left( x,y\right) =f(xy)$ and
set ${\rm B}^{2}\left( {\bf A},{\bf V}\right) =\left\{\delta f \mid f\in {\rm Hom}\left( {\bf A},{\bf V}\right) \right\} $. One
can easily check that ${\rm B}^{2}({\bf A},{\bf V})$ is a linear subspace of ${\rm Z}^{2}\left( {\bf A},{\bf V}\right)$.
Let us define $\rm {H}_{\mathcal A}^{2}\left( {\bf A},{\bf V}\right) $ as the quotient space ${\rm Z}^{2}\left( {\bf A},{\bf V}\right) \big/{\rm B}^{2}\left( {\bf A},{\bf V}\right)$.
The equivalence class of $\theta\in {\rm Z}^{2}\left( {\bf A},{\bf V}\right)$ in $\rm {H}_{\mathcal A}^{2}\left( {\bf A},{\bf V}\right)$ is denoted by $\left[ \theta \right]$. As usual, we call the elements of ${\rm Z}^{2}\left( {\bf A},{\bf V}\right)$ cocycles, those of ${\rm B}^{2}({\bf A},{\bf V})$ coboundaries, and $\rm {H}_{\mathcal A}^{2}\left( {\bf A},{\bf V}\right)$ is the corresponding second cohomology space.

Suppose now that $\dim{\bf A}=m<n$ and $\dim{\bf V}=n-m$. For any skew-symmetric
bilinear map $\theta :{\bf A}\times {\bf A}\longrightarrow {\bf V%
}$, one can define on the space ${\bf A}_{\theta }:={\bf A}\oplus 
{\bf V}$ the anticommutative bilinear product  $\left[ -,-\right] _{%
{\bf A}_{\theta }}$ by the equality $\left[ x+x^{\prime },y+y^{\prime }\right] _{%
{\bf A}_{\theta }}= xy  +\theta \left( x,y\right) $ for  
$x,y\in {\bf A},x^{\prime },y^{\prime }\in {\bf V}$. The algebra ${\bf A}_{\theta }$ is called an $(n-m)$-{\it %
dimensional central extension} of ${\bf A}$ by ${\bf V}$.
It is clear that ${\bf A}_{\theta }$ is nilpotent if and only if ${\bf A}$ is nilpotent and also that ${\bf A}_{\theta }$ is anticommutative if
and only if ${\bf A}$ is anticommutative and $\theta$ is skew-symmetric.

For a skew-symmetric bilinear form $\theta :{\bf A}\times {\bf A}\longrightarrow {\bf V}$, the space $\theta ^{\bot }=\left\{ x\in {\bf A}\mid \theta \left(
{\bf A},x\right) =0\right\} $ is called the {\it annihilator} of $\theta$.
For an anticommutative algebra ${\bf A}$, the ideal 
${\rm Ann}\left( {\bf A}\right) =\left\{ x\in {\bf A}\mid {\bf A}x =0\right\}$ is called the {\it annihilator} of ${\bf A}$.
One has
\begin{equation*}
{\rm Ann}\left( {\bf A}_{\theta }\right) =\left( \theta ^{\bot }\cap {\rm Ann}\left( 
{\bf A}\right) \right) \oplus {\bf V}.
\end{equation*}
Any $n$-dimensional  anticommutative algebra with non-trivial annihilator can be represented in
the form ${\bf A}_{\theta }$ for some $m$-dimensional  anticommutative algebra ${\bf A}$, some $(n-m)$-dimensional vector space ${\bf V}$ and some $\theta \in {\rm Z}^{2}\left( {\bf A},{\bf V}\right)$, where $m<n$ (see \cite[Lemma 5]{hac16}).
Moreover, there is a unique such representation with $m=n-\dim{\rm Ann}({\bf A})$. Note that the latter equality is equivalent to the condition  $\theta ^{\bot }\cap {\rm Ann}\left( 
{\bf A}\right)=0$. 

Let us pick some $\phi\in {\rm Aut}\left( {\bf A}\right)$, where ${\rm Aut}\left( {\bf A}\right)$ is the automorphism group of  ${\bf A}$.
For $\theta\in {\rm Z}^{2}\left( {\bf A},{\bf V}\right)$, let us define $(\phi \theta) \left( x,y\right) =\theta \left( \phi \left( x\right)
,\phi \left( y\right) \right) $. Then we get an action of ${\rm Aut}\left( {\bf A}\right) $ on ${\rm Z}^{2}\left( {\bf A},{\bf V}\right)$ which induces an action of that group on $\rm {H}_{\mathcal{A}}^{2}\left( {\bf A},{\bf V}\right)$.

\begin{definition}
Let ${\bf A}$ be an algebra and $I$ be a subspace of ${\rm Ann}({\bf A})$. If ${\bf A}={\bf A}_0 \oplus I$ for some subalgebra ${\bf A}_0$ of ${\bf A}$
then $I$ is called an {\it annihilator component} of ${\bf A}$. We say that an algebra is  {\it split} if it has a nontrivial annihilator component.
\end{definition}

For a linear space $\bf U$, the {\it Grassmannian} $G_{s}\left( {\bf U}\right) $ is
the set of all $s$-dimensional linear subspaces of ${\bf U}$. For any $0\leq s\leq \dim \rm {H}_{\mathcal A}^{2}\left( {\bf A},{\mathbb C}\right)$, the action of ${\rm Aut}\left( {\bf A}\right)$ on $\rm {H}_{\mathcal A}^{2}\left( {\bf A},\mathbb{C}\right)$ induces 
an action of that group on $G_{s}\left( \rm {H}_{\mathcal A}^{2}\left( {\bf A},\mathbb{C}\right) \right)$.
Let us define
$$
{\bf T}_{s}\left( {\bf A}\right) =\left\{ {\bf W}\in G_{s}\left( \rm {H}_{\mathcal A}^{2}\left( {\bf A},\mathbb{C}\right) \right)\left|\underset{[\theta]\in W}{\cap }\theta^{\bot }\cap {\rm Ann}\left( {\bf A}\right) =0\right.\right\}.
$$
Note that, by \cite[Lemmas 15 and 16]{hac16}, ${\bf T}_{s}\left( {\bf A}\right)$ is well defined and stable under the action of ${\rm Aut}\left( {\bf A}\right) $.

Let us fix a basis $e_{1},\ldots
,e_{s} $ of ${\bf V}$, and $\theta \in {\rm Z}^{2}\left( {\bf A},{\bf V}\right) $. Then there are unique $\theta _{i}\in {\rm Z}^{2}\left( {\bf A},\mathbb{C}\right)$ ($1\le i\le s$) such that $\theta \left( x,y\right) =\underset{i=1}{\overset{s}{%
\sum }}\theta _{i}\left( x,y\right) e_{i}$ for all $x,y\in{\bf A}$. Note that $\theta ^{\bot
}=\theta^{\bot} _{1}\cap \theta^{\bot} _{2}\cap\cdots \cap \theta^{\bot} _{s}$ in this case.
If   $\theta ^{\bot
}\cap {\rm Ann}\left( {\bf A}\right) =0$, then by \cite[Lemma 13]{hac16} the algebra ${\bf A}_{\theta }$ is split if and only if $\left[ \theta _{1}\right] ,\left[
\theta _{2}\right] ,\ldots ,\left[ \theta _{s}\right] $ are linearly
dependent in $\rm {H}_{\mathcal A}^{2}\left( {\bf A},\mathbb{C}\right)$. Thus, if $\theta ^{\bot
}\cap {\rm Ann}\left( {\bf A}\right) =0$ and ${\bf A}_{\theta }$ is non-split, then $\left\langle \left[ \theta _{1}\right] , \ldots,%
\left[ \theta _{s}\right] \right\rangle$ is an element of ${\bf T}_{s}\left( {\bf A}\right)$.
Now, if $\vartheta\in {\rm Z}^{2}\left( {\bf A},\bf{V}\right)$ is such that $\vartheta ^{\bot
}\cap {\rm Ann}\left( {\bf A}\right) =0$ and ${\bf A}_{\vartheta }$ is non-split, then by \cite[Lemma 17]{hac16} one has ${\bf A}_{\vartheta }\cong{\bf A}_{\theta }$ if and only if
$\left\langle \left[ \theta _{1}\right] ,\left[ \theta _{2}%
\right] ,\ldots ,\left[ \theta _{s}\right] \right\rangle,
\left\langle \left[ \vartheta _{1}\right] ,\left[ \vartheta _{2}\right] ,\ldots,%
\left[ \vartheta _{s}\right] \right\rangle\in {\bf T}_{s}\left( {\bf A}\right)$ belong to the same orbit under the action of ${\rm Aut}\left( {\bf A}\right) $, where $%
\vartheta \left( x,y\right) =\underset{i=1}{\overset{s}{\sum }}\vartheta
_{i}\left( x,y\right) e_{i}$.

Hence, there is a one-to-one correspondence between the set of $
{\rm Aut}\left({\bf A}\right)$-orbits on ${\bf T}_{s}\left( {\bf A}
\right) $ and the set of isomorphism classes of non-split central extensions of $\bf{A}$ by $\bf{V}$ with $s$-dimensional annihilator.
Consequently to construct all non-split $n$-dimensional central extensions with $s$-dimensional annihilator
of a given $(n-s)$-dimensional algebra ${\bf A}$ one has to describe ${\bf T}_{s}({\bf A})$, ${\rm Aut}({\bf A})$ and the action of ${\rm Aut}({\bf A})$ on ${\bf T}_{s}({\bf A})$ and then
for each orbit under the action of ${\rm Aut}({\bf A})$ on ${\bf T}_{s}({\bf A})$ pick a representative and construct the algebra corresponding to it.

We will use the following auxiliary notation during the construction of central extensions.
Let ${\bf A}$ be an anticommutative  algebra with basis $e_{1},e_{2},\ldots,e_{n}$. For $1\leq i\neq j\leq n$, 
$\Delta_{ij}:{\bf A}\times {\bf A}\longrightarrow \mathbb{C}$ denotes the skew-symmetric bilinear form  defined by the equalities $\Delta _{ij}\left( e_{i},e_{j}\right)=-\Delta _{ij}\left( e_{j},e_{i}\right)=1$
and $\Delta _{ij}\left( e_{l},e_{m}\right) =0$ for $\left\{ l,m\right\} \neq \left\{ i,j\right\}$. In this case, the $\Delta_{ij}$ with $1\leq i < j\leq n $ form a basis of the space ${\rm Z}^{2}\left( {\bf A},{\mathbb C}\right)$ of skew-symmetric bilinear forms on $\bf{A}$.

We will often use the symbols $\nabla_i$ to represent a basis of $\rm {H}_{\mathcal{A}}^{2}\left( {\bf A}, \mathbb C\right)$. Given an element $\0=\sum_{i}\af_i\nb i\in {\rm H}^2_{{\mathcal A}}({\bf A}, \mathbb C)$, with $\af_i\in\mathbb{C}$, and an automorphism $\phi\in {\rm Aut}\left( {\bf A}\right)$, we will write $\phi \theta=\sum_{i}\af_i^{*}\nb i$. The coefficients $\af_i^{*}$ are easy to compute: identifying $\phi$ and $\theta$ with their corresponding matrix representations with respect to some fixed basis of ${\bf A}$, we deduce that the matrix representation of $\phi \theta$ in that basis of ${\bf A}$ is just $\phi^{T}\theta\phi$, from which the $\af_i^{*}$ are readily determined. Below, we omit the details of those computations. 

The description of the multiplication of a given $n$-dimensional anticommutative  algebra ${\bf A}$ is given in terms of the distinguished basis $e_{1},e_{2},\ldots,e_{n}$. 
We omit the products of basis elements which either are zero, or can be deduced from the anti\-com\-mu\-ta\-ti\-vi\-ty of the algebra. Unless otherwise stated, all matrices involving ${\bf A}$ are taken with respect to this distinguished basis and the automorphism group of ${\bf A}$ is described in terms of the matrices of its elements. If no additional conditions are mentioned, the variables in these descriptions may take arbitrary complex values, subject only to the restriction that the corresponding determinant is nonzero. The details of the computations of the automorphism groups are omitted.

We will make use of previous work on the algebraic and geometric classification of certain classes of anticommutative algebras. An important such class is that of Tortkara algebras. These are anticommutative algebras satisfying the identity
\begin{equation*}
(ab)(cb) = J (a, b, c)b, \quad \mbox{where $J (a, b, c) = (ab)c + (bc)a + (ca)b$.} 
\end{equation*}
The algebraic classification of all $6$-dimensional nilpotent Tortkara algebras was completed in \cite{gkk} and their geometric classification was obtained in \cite{gkk19}. Therefore, our classification of anticommutative algebras will be carried out modulo the class of Tortkara algebras. Malcev algebras form another important class of anticommutative algebras which includes all Lie algebras. We will thence use the following notation:

$$\begin{array}{lll}
{\bf A}_{j}& \mbox{the }j\mbox{th }4\mbox{-dimensional nilpotent anticommutative algebra}, \\
{\mathcal A}_{j}& \mbox{the }j\mbox{th }5\mbox{-dimensional nilpotent anticommutative algebra}, \\
{\mathbb A}_{j}& \mbox{the }j\mbox{th }6\mbox{-dimensional nilpotent anticommutative non-Tortkara algebra}, \\
{\mathbb M}_{j}& \mbox{the }j\mbox{th }6\mbox{-dimensional nilpotent Malcev-Tortkara   algebra}, \\
{\mathbb T}_{j}& \mbox{the }j\mbox{th }6\mbox{-dimensional nilpotent Tortkara non-Malcev  algebra}. 
\end{array}$$

The subspace ${\rm Z}_{\mathbb T}^{2}\left( {\bf A},{\bf V}\right)$ of cocycles $\theta\in{\rm Z}^{2}\left( {\bf A},{\bf V}\right)$ such that ${\bf A}_{\theta }$ is Tortkara determines the second cohomology space $\rm {H}_{\mathbb T}^{2}\left( {\bf A},{\bf V}\right)$, which we view as a subspace of $\rm {H}_{\mathcal A}^{2}\left( {\bf A},{\bf V}\right)$. In case ${\bf V}=\mathbb C$ we simply write these spaces as $\rm {H}_{\mathbb T}^{2}\left( {\bf A}\right)$ and $\rm {H}_{\mathcal A}^{2}\left( {\bf A}\right)$.

\subsection{The algebraic classification of $6$-dimensional split nilpotent anticommutative algebras}
Thanks to \cite{fkkv}, we have only one (non-Tortkara) algebra of this type:

$$\begin{array}{lclll l}
{\mathbb A}_{00}:  & e_1e_2=e_3,& e_1e_3=e_4,& e_3e_4=e_5. 
\end{array}$$

\medskip

It is easy to see that any $6$-dimensional nilpotent anticommutative algebra ${\bf A}$ such that $\dim {\rm Ann}\left( {\bf A}\right)\geq 3$ is necessarily split. Thus, it remains to consider non-split algebras having an annihilator of dimension $1$ or $2$.

\subsection{The algebraic classification of $6$-dimensional non-split nilpotent anticommutative algebras with $2$-dimensional annihilator}

%

%


Thanks to \cite{cfk182}, we have the classification of all nontrivial $4$-dimensional nilpotent anticommutative algebras.
\begin{equation*}
\begin{array}{|l|l|l|l|} 
\hline
\mbox{$\bf A$}  & \mbox{ Multiplication table} 
& \mbox{${\rm H}_{\mathbb T}^2({\bf A})$}
& \mbox{${\rm H}_{\mathcal A}^2({\bf A})$}  \\ 

\hline
\hline

{\bf A}_{01}& e_1e_2 = e_3 &
\la[\Delta_{13}], [\Delta_{14}], [\Delta_{23}], [\Delta_{24}], [\Delta_{34}]  \ra &
\mbox{${\rm H}_{\mathbb T}^2({\bf A}_{01})$}
  \\
\hline
{\bf A}_{02}& e_1e_2=e_3, e_1e_3=e_4 &
\la[\Delta_{14}], [\Delta_{23}], [\Delta_{24}] \ra &
\mbox{${\rm H}_{\mathbb T}^2({\bf A}_{02})$} \oplus 
\la [\Delta_{34}] \ra \\

\hline

\end{array}
\end{equation*}

In view of \cite{gkks},
all anticommutative central extensions of ${\bf A}_{01}$ and of the $4$-dimensional trivial algebra are Tortkara algebras, so we need only consider central extensions of ${\bf A}_{02}$.

\subsubsection{$2$-dimensional central extensions of ${\bf A}_{02}$}
	Let us use the notation
	$$ 
	\begin{array}{rclrclrclrclrcl}
	\nb 1& = &\Dl 14, & \nb 2& = &\Dl 23, &\nb 3& = &\Dl 24, &\nb 4& = &\Dl 34.    
	\end{array}
	$$
	Take $\0=\sum_{i=1}^4\af_i\nb i\in {\rm H}^2_{{\mathcal A}}({\bf A}_{02})$.
	If 
$$
	\phi=
\begin{pmatrix} 
x& 0 & 0 & 0\\
y & z & 0 & 0 \\
u & v & xz  & 0\\
h & g & xv & x^2z
\end{pmatrix}
\in\aut{{\bf A}_{02}},
	$$
	then
$\phi\theta=\sum_{i=1}^4 \alpha_i^{*} \nabla_i$,
where
$$\begin{array}{lcl}
\alpha_1^{*}&=&x^2z(\alpha_1x+\alpha_3y+\alpha_4u), \\
\alpha_2^{*}&=&xz(\alpha_2z-\alpha_4g)+vx(\alpha_3z+\alpha_4v), \\
\alpha_3^{*}&=&x^2z(\alpha_3z+\alpha_4v),\\
\alpha_4^{*}&=&\alpha_4x^3z^2.
\end{array}$$

Consider the vector space generated by the following two cocycles 
$$\begin{array}{lcl}
\theta_1&=&\alpha_1  \nabla_1+\alpha_2  \nabla_2+\alpha_3 \nabla_3+ \nabla_4,\\
\theta_2&=&\beta_1 \nabla_1+\beta_2 \nabla_2+\beta_3 \nabla_3.\\
\end{array}$$
Choosing $u=-(\alpha_1x+\alpha_3y), 
g={\alpha_2 z}, v=-\alpha_3z$, we get $\phi\langle \theta_1\rangle=\langle\nb 4\rangle$. Furthermore, 
\begin{enumerate}
    \item if $\beta_3\neq0,$ 
    then we can suppose that $\beta_2\beta_3^{-1}=\alpha_3$, which gives $\bt_2^*=0$. Choosing now  $y=-\beta_1\beta_3^{-1}x$, we get the representative $\langle \nabla_3, \nabla_4 \rangle;$
    \item if $\beta_3=0$ and $\beta_2\neq 0, \beta_1 \neq 0,$ then choosing $z=\frac{\beta_1x^2}{\beta_2}$ we get the representative $\langle \nabla_1+\nabla_2, \nabla_4 \rangle;$
    \item if $\beta_3=0$ and $\beta_2\neq 0, \beta_1 = 0,$ then we get the representative $\langle \nabla_2, \nabla_4 \rangle;$
    \item if $\beta_3=0$ and $\beta_2 = 0, \beta_1 \neq 0,$ then  we get the representative $\langle \nabla_1, \nabla_4 \rangle.$
\end{enumerate}
It is easy to see that the $4$ subspaces above are elements of ${\bf T}_{2}({\bf A}_{02})$ and that they determine distinct orbits under ${\rm Aut}({\bf A}_{02})$. Thus, we have the following algebras:
\begin{longtable}{lllll lll}
${\mathbb A}_{01}$ &:& $e_1e_2=e_3$,& $e_1e_3=e_4$, & $e_2e_4=e_5$, & $e_3e_4=e_6$;\\
${\mathbb A}_{02}$ &:& $e_1e_2=e_3$,& $e_1e_3=e_4$, & $e_1e_4=e_5$,& $e_2e_3=e_5$, & $e_3e_4=e_6$;\\
${\mathbb A}_{03}$ &:& $e_1e_2=e_3$,& $e_1e_3=e_4$, &  $e_2e_3=e_5$, & $e_3e_4=e_6$;\\
${\mathbb A}_{04}$ &:& $e_1e_2=e_3$,& $e_1e_3=e_4$, & $e_1e_4=e_5$,&  $e_3e_4=e_6$.
\end{longtable}

\subsection{The algebraic classification of $6$-dimensional non-split nilpotent anticommutative algebras with $1$-dimensional annihilator}

Thanks to \cite{fkkv} we have the algebraic classification of all nontrivial $5$-dimensional nilpotent anticommutative algebras.

\begin{longtable}{|l|l|l|l|} 
\hline
$\bf A$  & multiplication table & ${\rm H_{\mathbb T}^2}({\bf A})$ & ${\rm H_{\mathcal A}^2}({\bf A})$ \\ 
\hline
\hline
${\mathcal A}_{01}$& 
$\begin{array}{l}
e_1e_2 = e_3 
\end{array}$&   

 $\Big\langle
\begin{array}{l} 
 [\Delta_{13}] ,  [\Delta_{14}] , [\Delta_{15}] ,  [\Delta_{23}],    [\Delta_{24}], \\ \relax
 [ \Delta_{25} ] ,  [\Delta_{34}], [\Delta_{35}] ,  [\Delta_{45}]
\end{array} 
\Big\rangle$

&${\rm H_{\mathbb T}^2}({\mathcal A}_{01})$   \\

\hline
${\mathcal A}_{02}$& 
$\begin{array}{l}
e_1e_2=e_3, e_1e_3=e_4 
\end{array}$& \relax
$\Big\langle \relax \begin{array}{l} 
[\Delta_{14}], [\Delta_{15}], [\Delta_{23}], [\Delta_{24}], \\ \relax
[\Delta_{25}],  [\Delta_{35}], [\Delta_{45}] \end{array}
\Big\rangle$
&${\rm H_{\mathbb T}^2}({\mathcal A}_{02}) \oplus \langle  [\Delta_{34}]\rangle$  
\\
\hline

${\mathcal A}_{03}$ & 
$\begin{array}{l}
e_1e_2=e_4, e_1e_3=e_5
\end{array}$& 
$\Big\langle 
\begin{array}{l}
[\Delta_{14}], [\Delta_{15}], [\Delta_{23}], [\Delta_{24}], \\ \relax
[\Delta_{25}], [\Delta_{34}], [\Delta_{35}]
\end{array}
\Big\rangle$ & ${\rm H_{\mathbb T}^2}({\mathcal A}_{03})  \oplus \langle  [\Delta_{45}]\rangle$  \\
\hline

${\mathcal A}_{04}$ & 
$\begin{array}{l}
e_1e_2=e_3,  e_1e_3=e_4,\\ e_2e_3=e_5 
\end{array}$& 
$\Big\langle 
\begin{array}{l}[\Delta_{14}],[\Delta_{15}],[\Delta_{24}],[\Delta_{25}]
\end{array}
\Big\rangle$ & ${\rm H_{\mathbb T}^2}({\mathcal A}_{04})\oplus\langle[\Delta_{34}], [\Delta_{35}],[\Delta_{45}]  \rangle$\\
\hline

${\mathcal A}_{05}$& 
$\begin{array}{l}
e_1e_2=e_5, e_3e_4=e_5
\end{array}$&

$\Big\langle 
\begin{array}{l} 
[\Delta_{12}], [\Delta_{13}], [\Delta_{14}], [\Delta_{15}],  [\Delta_{23}],\\ \relax
[\Delta_{24}], [\Delta_{25}], [\Delta_{35}], [\Delta_{45}] 
\end{array} 
\Big\rangle $
&  ${\rm H_{\mathbb T}^2}({\mathcal A}_{05})$\\
\hline

${\mathcal A}_{06}$& 
$\begin{array}{l}
e_1e_2=e_3, e_1e_4=e_5,\\
e_2e_3=e_5
\end{array}$& 
$\Big\langle
\begin{array}{l}[\Delta_{13}],[\Delta_{14}],[\Delta_{15}],[\Delta_{24}],\\ \relax
[\Delta_{25}],[\Delta_{34}], [\Delta_{45}] \end{array}
\Big\rangle$ & ${\rm H_{\mathbb T}^2}({\mathcal A}_{06})\oplus\langle [\Delta_{35}] \rangle$\\
\hline

${\mathcal A}_{07}$& 
$\begin{array}{l}
e_1e_2=e_3,  e_3e_4=e_5 
\end{array}$&
$\Big\langle 
\begin{array}{l}[\Delta_{13}], [\Delta_{14}], [\Delta_{23}], [\Delta_{24}]\end{array}
\Big\rangle$ &  ${\rm H_{\mathbb T}^2}({\mathcal A}_{07}) \oplus\langle [\Delta_{15}], [\Delta_{25}], [\Delta_{35}], [\Delta_{45}]  \rangle$ \\
\hline

${\mathcal A}_{08}$& 
$\begin{array}{l}
e_1e_2=e_3, e_1e_3=e_4, \\ 
e_1e_4=e_5
\end{array}$& 
$\Big\langle
\begin{array}{l}[\Delta_{15}], [\Delta_{23}], [\Delta_{24}], [\Delta_{25}] \end{array}
\Big\rangle$ & ${\rm H_{\mathbb T}^2}({\mathcal A}_{08})\oplus\langle  [\Delta_{34}], [\Delta_{35}],  [\Delta_{45}]\rangle$\\
\hline

${\mathcal A}_{09}$&
$\begin{array}{l}
e_1e_2=e_3, e_1e_3=e_4,\\ 
e_1e_4=e_5,  e_2e_3=e_5
\end{array}$ & 
$\Big\langle
\begin{array}{l}[\Delta_{14}], [\Delta_{15}], [\Delta_{24}], [\Delta_{25}]\end{array}
\Big\rangle$ & 
${\rm H_{\mathbb T}^2}({\mathcal A}_{09})\oplus\langle [\Delta_{34}], [\Delta_{35}],  [\Delta_{45}] \rangle$\\
\hline

${\mathcal A}_{10}$& 
$\begin{array}{l}
e_1e_2=e_3, e_1e_3=e_4,\\ 
e_2e_4=e_5
\end{array}$& 
$\Big\langle \begin{array}{l} [\Delta_{14}], [\Delta_{23}], [\Delta_{34}]+[\Delta_{15}] \end{array} \Big\rangle$ & 
${\rm H_{\mathbb T}^2}({\mathcal A}_{10})\oplus\langle [\Delta_{15}], [\Delta_{25}],[\Delta_{35}],  [\Delta_{45}] \rangle$\\
\hline

${\mathcal A}_{11}$ &  $\begin{array}{l}
e_1e_2=e_3, e_1e_3=e_4,\\ e_3e_4=e_5
\end{array}$ & 
--- &  
$\langle  [\Delta_{14}], [\Delta_{15}],  [\Delta_{23}], [\Delta_{24}], [\Delta_{25}], [\Delta_{35}], [\Delta_{45}] 
 \rangle$ 

\\ \hline

\end{longtable}

\subsubsection{$1$-dimensional  central extensions of ${\mathcal A}_{02}$}
Let us use the notation 
\[\nabla_1=[\Delta_{14}], \nabla_2=[\Delta_{15}], \nabla_3=[\Delta_{23}], \nabla_4=[\Delta_{24}],  \nabla_5=[\Delta_{25}],   
\nabla_6=[\Delta_{35}],\nabla_7= [\Delta_{45}],\nabla_8= [\Delta_{34}].\]
The automorphism group of ${\mathcal A}_{02}$ 
consists of the invertible  matrices of the form 
$$\phi=
\begin{pmatrix}
x& 0& 0& 0& 0\\
f& y& 0& 0& 0\\
u& v& xy& 0& 0\\
h& r& xv& x^2y& l\\
t& g& 0& 0& z
\end{pmatrix}.
$$
Notice that we must have $\det\phi=x^4y^3z\ne 0$. Let $\0=\sum_{i=1}^8\af_i\nb i\in {\rm H}^2_{{\mathcal A}}({\mathcal A}_{02})$.
Then $\phi\theta=\sum\limits_{i=1}^8 \af_i^*  \nb i$,
where
$$
\begin{array}{rcl} 


\alpha_1^*&=& x^2 y (x\af_1 + f \af_4  - t \af_7 + u \af_8),\\

\alpha_2^*&=& xl\af_1  +xz \af_2 + 
   fl \af_4 + zf\af_5 +z u\af_6 + (zh    - tl) \af_7  + ul \af_8,\\

\alpha_3^*&=& x (y^2  \af_3+ v y\af_4 - yg \af_6   - vg \af_7 + (v^2- yr) \af_8),\\

\alpha_4^*&=& x^2 y (y \af_4 - g \af_7 + v\af_8),\\

\alpha_5^* &=& l y\af_4+ zy \af_5  + zv \af_6+   (zr- lg)\af_7 + vl\af_8, \\

\alpha_6^* &=& x (y z \af_6 + v z \af_7  + l y \af_8),\\

\alpha_7^*&=&x^2 y z \af_7,\\

\af_8^*&=& x^3 y^2 \af_8.
\end{array}
$$

We are interested only in those $\theta$ with $\af_8 \ne 0$, so we can assume that $\alpha_8=1$. We have the following cases:

\begin{enumerate}
    \item $\alpha_7 \neq 0$. Choosing
    $l = -\frac{z(y \af_6 + v \af_7)}{y},$
    $v = -y \af_4 + g \af_7$,
    $u = -x \af_1 - f \af_4 + t \af_7,$
    $r = y \af_3 - g \af_6$
    and
    $h = \frac{x \af_1 \af_6 + f \af_4 \af_6 - t \af_6 \af_7 - x \af_2 -   f \af_5}{\af_7}$  
    we have the representative $\langle\alpha_5^\star \nabla_5+ \alpha_7^* \nabla_7+ \alpha_8^* \nabla_8\rangle$, where $\af_5^\star=(-\af_4 \af_6 + \af_3 \af_7 + \af_5)yz$.
    \begin{enumerate}
        \item If $-\af_4 \af_6 + \af_3 \af_7 + \af_5 \ne 0$, then choosing 
        $z = \frac{xy}{\af_7},$
        $x = \sqrt{\frac{-\af_4 \af_6 + \af_3 \af_7 + \af_5}{\af_7}}$
          we have the representative $\langle\nabla_5+\nabla_7+\nabla_8\rangle.$

        \item If $-\af_4 \af_6 + \af_3 \af_7 + \af_5= 0$, then choosing,   
        $z = \frac{xy}{\af_7},$
      we have the representative $\langle\nabla_7+\nabla_8\rangle.$
\end{enumerate}

    \item $\alpha_7=0$. Choosing
    $l = -z \af_6,$
    $v = -y \af_4,$
    $r = y\af_3-g\af_6,$
    $u = -x \af_1 - f \af_4$ 
    we have the representative
    $\langle\alpha_2^\star \nabla_2+\alpha_5^\star \nabla_5+\alpha_8^*  \nabla_8\rangle$, where $\af_2^\star=((\af_5 -\af_4\af_6)f + (\af_2- \af_1\af_6)x)z$ and $\alpha_5^\star=(\af_5 - \af_4 \af_6)yz$.
    \begin{enumerate}
        \item If $\af_5 - \af_4 \af_6\ne 0$ then choosing 
        $z = \frac{x^3 y}{\af_5-\af_4 \af_6},$
        $f = -\frac{(\af_2- \af_1\af_6)x}{\af_5 -\af_4\af_6},$
        we have the representative $\langle\nabla_5+\nabla_8\rangle.$
        
        \item If $\af_5 - \af_4 \af_6=0$ 
        and $\af_2-\af_1 \af_6\ne 0,$ then choosing 
        $z=\frac{x^2y^2}{\af_2-\af_1 \af_6}$,
        we have the representative $\langle \nabla_2+\nabla_8 \rangle.$ 
        
        \item If $\af_5 - \af_4 \af_6=0$ and $\af_2 -\af_1 \af_6=0,$ then  we have the representative $\langle  \nabla_8 \rangle.$ Note that this space is not in ${\bf T}_1(\mathcal{A}_{02})$ because $e_5\in\nabla_8^{\bot
}\cap {\rm Ann}\left( \mathcal{A}_{02}\right)$. It gives an anticommutative algebra with $2$-dimensional annihilator, which was already found above. 

    \end{enumerate}

    \end{enumerate}
    
Summarizing, we obtain the following representatives:
 $\langle \nabla_2+\nabla_8 \rangle$,
 $\langle\nabla_5+\nabla_7+\nabla_8\rangle$,
 $\langle\nabla_5+\nabla_8\rangle$,
 $\langle\nabla_7+\nabla_8\rangle$.
 All of them belong to distinct orbits. The corresponding algebras are:

\begin{longtable}{llllllll}
${\mathbb A}_{05}$ & : &  
$e_1e_2=e_3$, & $e_1e_3=e_4$, &  $e_1e_5=e_6$, &  $e_3e_4=e_6$;\\ 
${\mathbb A}_{06}$ & : &  
$e_1e_2=e_3$, & $e_1e_3=e_4$, &  $e_2e_5=e_6$, & $e_3e_4=e_6$, & $e_4e_5=e_6$;  \\
${\mathbb A}_{07}$ & : &  
$e_1e_2=e_3$, & $e_1e_3=e_4$, &  $e_2e_5=e_6$, & $e_3e_4=e_6$;  \\
${\mathbb A}_{08}$ & : &  
$e_1e_2=e_3$, & $e_1e_3=e_4$, & $e_3e_4=e_6$, & $e_4e_5=e_6$.
\end{longtable}

\subsubsection{$1$-dimensional  central extensions of ${\mathcal A}_{03}$}

Let us use the notation
\[
\nb 1 = \Dl 14, \nb 2 = \Dl 15, \nb 3 = \Dl 23, \nb 4 = \Dl 24,
\nb 5 = \Dl 25, \nb 6 = \Dl 34, \nb 7 = \Dl 35, \nb 8 = \Dl 45.    
\]

Take $\0=\sum_{i=1}^8\af_i\nb i\in {\rm H}^2_{\mathcal{A}}({\mathcal A}_{03})$. The automorphism group of ${\mathcal A}_{03}$ consists of the invertible matrices of the form
$$
\phi=
\begin{pmatrix}
x & 0 & 0 &     0     &   0\\
u & v & w &     0     &   0\\
p & q & r &     0     &   0\\
h & k & l &    xv     & xw\\
a & b & c &    xq     & xr
\end{pmatrix}.
$$
Notice that we must have $\det\phi=x^3(vr-wq)^2\ne 0$. Then $\phi\theta=\sum\limits_{i=1}^8 \af_i^*  \nb i$,
where
\begin{align*}
\af^*_1 &= vx(\af_1x + \af_4u + \af_6p) + qx(\af_2x  + \af_5u  + \af_7p) + (qxh - vxa)\af_8,\\
\af^*_2 &= wx(\af_1x + \af_4u + \af_6p) + rx(\af_2x  + \af_5u  + \af_7p) + (rxh - wxa)\af_8,\\
\af^*_3 &= l(\af_4v + \af_6q) - k(\af_4w + \af_6r) + c(\af_5v + \af_7q) - b(\af_5w + \af_7r) + (rv - qw)\af_3  + (ck - bl)\af_8,\\
\af^*_4 &= vx(\af_4v + \af_6q) + qx(\af_5v  + \af_7q) + (qxk - vxb)\af_8,\\
\af^*_5 &= wx(\af_4v + \af_6q) + rx(\af_5v  + \af_7q) + (rxk - wxb)\af_8,\\
\af^*_6 &= vx(\af_4w + \af_6r) + qx(\af_5w  + \af_7r) + (qxl - vxc)\af_8,\\
\af^*_7 &= wx(\af_4w + \af_6r) + rx(\af_5w  + \af_7r) + (rxl - wxc)\af_8,\\
\af^*_8 &= x^2(vr - wq)\af_8.
\end{align*}

We are interested only in those $\theta$ with $\af_8 \ne 0$. Consider $\af^*_6=\af^*_7=0$ as a linear system in $l$ and $c$. Its determinant is $\af_8^2x^2(vr - wq)\neq 0$. So, we can find $l,c$ such that $\af^*_6=\af^*_7=0$. For the same reason, we may find $k,b$ such that $\af^*_4=\af^*_5=0$ and $h,a$ such that $\af^*_1=\af^*_2=0$. Thus, we may suppose from the very beginning that $\af_1=\af_2=\af_4=\af_5=\af_6=\af_7=0$, $\af_8\ne 0$. Choosing $a=b=c=h=k=l=0$ we have $\af^*_1=\af^*_2=\af^*_4=\af^*_5=\af^*_6=\af^*_7=0$ and 
\begin{align*}
\af^*_3 &= (vr - wq)\af_3,\\
\af^*_8 &= x^2(vr - wq)\af_8.
\end{align*}
Thus, we have two representatives: $\la\nb 8\ra$ and $\la\nb 3+\nb 8\ra$, depending on whether $\af_3=0$ or not.

 The algebras corresponding to $\la\nb 3+\nb 8\ra$ and $\la\nb 8\ra$ are:

\begin{longtable}{llllllll} 
${\mathbb A}_{09}$ & : &  
$e_1e_2=e_4$, & $e_1e_3=e_5$, & $e_2e_3=e_6$, & $e_4e_5=e_6$;\\
${\mathbb A}_{10}$ & : &  
$e_1e_2=e_4$, & $e_1e_3=e_5$, & $e_4e_5=e_6$.
\end{longtable}

\subsubsection{$1$-dimensional  central extensions of ${\mathcal A}_{04}$}
Let us use the notation 

\[ \nabla_1=[\Delta_{14}],  \ \nabla_2=[\Delta_{15}], 
\ \nabla_3=[\Delta_{24}],\ \nabla_4=[\Delta_{25}],
\ \nabla_5=[\Delta_{34}],\ \nabla_6=[\Delta_{35}],
\ \nabla_7=[\Delta_{45}].\]
The automorphism group of ${\mathcal A}_{04}$ consists of invertible matrices of the form
$$\phi=
\begin{pmatrix}
    x& y& 0& 0& 0\\
    v& z& 0& 0& 0\\
    u& h& xz-yv& 0& 0\\
    l& r& xh-yu& x(xz-yv)& y(xz-yv)\\
    t& g& vh-zu& v(xz-yv)& z(xz-yv)
\end{pmatrix}.$$
Thus, we must have $\det\phi=(xz-yv)^5\ne 0$. For $\0=\sum_{i=1}^7\af_i\nb i\in {\rm H}^2_{{\mathcal A}}({\mathcal A}_{04})$, we get $\phi\theta=\sum\limits_{i=1}^7 \af_i^*  \nb i$,
where
$$\begin{array}{rcl}
\alpha_1^* &=& (xz - yv) (v (x \alpha_2 + v \alpha_4 + u \alpha_6 + l \alpha_7) + x (x \alpha_1 + v \alpha_3 + u \alpha_5 - t \alpha_7)),\\
\alpha_2^* &=& (xz - yv) (z (x \alpha_2 + v \alpha_4 + u \alpha_6 + l \alpha_7) +     y (x \alpha_1 + v \alpha_3 + u \alpha_5 - t \alpha_7)),\\
\alpha_3^* &=& (xz - yv) (x (y \alpha_1 + z \alpha_3 + h \alpha_5 - g \alpha_7) +     v (y \alpha_2 + z \alpha_4 + h \alpha_6 + r \alpha_7)),\\
\alpha_4^* &=& (xz - yv) (y (y \alpha_1 + z \alpha_3 + h \alpha_5 - g \alpha_7) +     z (y \alpha_2 + z \alpha_4 + h \alpha_6 + r \alpha_7)),\\
\alpha_5^* &=& ( x z-  yv)^2 (x \alpha_5  + v \alpha_6 + u \alpha_7), \\
\alpha_6^* &=& ( x z-yv)^2 (y \alpha_5  + z \alpha_6 + h \alpha_7),\\
\alpha_7^* &=& (xz - yv)^3  \alpha_7.
\end{array}
$$

We are only interested in cocycles with $(\alpha_5, \alpha_6, \alpha_7) \neq (0,0,0).$ 
Note that this condition is invariant under automorphisms. We have several cases to consider:
\begin{enumerate}
    \item $\af_7\ne 0$. Then choosing 
    $x=0,$ $z=0,$ $v=1,$ $y=-\alpha_7^{-1/3},$ $u=-\frac{\alpha_6}{\alpha_7}, $
    $h=\frac{\alpha_5}{\alpha_7^{4/3}},$
    $g=\frac{\alpha_5^2 - \alpha_1 \alpha_7}{\alpha_7^{7/3}},$
    $r=\frac{\alpha_2 \alpha_7-\alpha_5 \alpha_6}{\alpha_7^{7/3}},$
    $t=\frac{\alpha_3 \alpha_7-\alpha_5 \alpha_6}{\alpha_7^2},$
    $l=\frac{\alpha_6^2 - \alpha_4 \alpha_7}{\alpha_7^2},$
    we get the representative $\langle \nabla_7 \rangle.$

        \item $\af_7= 0$ and $\af_6\ne 0$. Then taking $z=-\frac{y \af_5}{\af_6}$ we can make $\af^*_7=\af^*_6=0$, so we shall suppose that $\af_7=\af_6=0$ and $\af_5 \neq 0$ from the very beginning.
        
        \begin{enumerate}
            \item If $\af_4 \neq 0$, then choosing 
              $x=\frac{\alpha_4^{2/7}}{\alpha_5^{3/7}},$ 
              $y=0,$ 
              $z=\frac{\alpha_5^{1/7}}{\alpha_4^{3/7}},$
              $v=-\frac{\alpha_2}{\alpha_4^{5/7} \alpha_5^{3/7}},$
              $h=\frac{\alpha_2 - \alpha_3}{\alpha_4^{3/7} \alpha_5^{6/7}},$
              $u= \frac{\alpha_2 \alpha_3 - \alpha_1 \alpha_4}{\alpha_4^{5/ 7}  \alpha_5^{10/ 7}},$ 
              we get the representative  $\langle \nabla_4+\nabla_5 \rangle.$
    
            \item If $\af_4=0$ and $\af_2 \neq 0$, then choosing 
            $x=\frac{\alpha_2}{\alpha_5},$
            $y=0,$ $ v=0, $ $z=1,$
            $u=-\frac{\alpha_1 \alpha_2}{\alpha_5^2},$
            $h=-\frac{\alpha_3}{\alpha_5},$
            we get the representative $\langle \nabla_2+\nabla_5 \rangle.$
            \item If $\af_4=0$ and $\af_2 = 0$, then $e_5\in\theta^{\bot
}\cap {\rm Ann}\left( \mathcal{A}_{04}\right)$ so we get an algebra with a $2$-dimensional annihilator, which we have already listed.        
        \end{enumerate}
\end{enumerate}        

The algebras corresponding to $\langle \nabla_2+\nabla_5 \rangle$, $\langle \nabla_4+\nabla_5 \rangle$ and $\langle \nabla_7 \rangle$ are:

	\begin{longtable}{llllllll} 
	${\mathbb A}_{11}$ & : &  
	$e_1e_2=e_3$, & $e_1e_3=e_4$, & $e_1e_5=e_6$, & $e_2e_3=e_5$, & $e_3e_4=e_6$;\\
	${\mathbb A}_{12}$ & : &  
	$e_1e_2=e_3$, & $e_1e_3=e_4$, & $e_2e_3=e_5$, & $e_2e_5=e_6$, & $e_3e_4=e_6$; \\
	${\mathbb A}_{13}$ & : &  
	$e_1e_2=e_3$, & $e_1e_3=e_4$, & $e_2e_3=e_5$, & $e_4e_5=e_6$.
	\end{longtable}

\subsubsection{$1$-dimensional  central extensions of ${\mathcal A}_{06}$}
Let us use the notation
	\[
	\nb 1 = \Dl 13, \nb 2 = \Dl 14, \nb 3 = \Dl 15, \nb 4 = \Dl 24,
	\nb 5 = \Dl 25, \nb 6 = \Dl 34, \nb 7 = \Dl 45, \nb 8 = \Dl 35.    
	\]

	The automorphism group of ${\mathcal A}_{06}$\ consists of invertible
	matrices of the form 
	$$\phi=
	\begin{pmatrix}
	x& p& 0& 0& 0\\
	0& y& 0& 0& 0\\
	z& t& xy& -py& 0\\
	q& r& 0& y^2& 0\\
	s& h& xr-yz-pq& f& xy^2
	\end{pmatrix}.$$
	So, we must have $\det\phi=x^3y^6\ne 0$. For $\0=\sum_{i=1}^8\af_i\nb i\in {\rm H}^2_{{\mathcal A}}({\mathcal A}_{06})$, we get $\phi\theta=\sum\limits_{i=1}^8 \af_i^*  \nb i$,
	where
	\begin{align*}
		\af^*_1 &= xy(\af_1x - \af_6q - \af_8s) -(pq - rx + yz)(\af_3x + \af_7q + \af_8z),\\
		\af^*_2 &= xy(\af_2y- 2\af_1p) + fx\af_3  + y(pq + rx + yz)\af_6 + (fq-sy^2)\af_7 + (psy + hxy  + fz)\af_8\\
		&\quad+(pq - rx + yz)(\af_3p+\af_5y+\af_7r+\af_8t),\\
		\af^*_3 &= xy^2(\af_3x + \af_7q + \af_8z),\\
		\af^*_4 &= py(\af_2y-\af_1p) + \af_3fp + \af_4y^3 + \af_5fy + y(pr + ty)\af_6 + (fr-hy^2)\af_7 + (hpy + ft)\af_8,\\ 
		\af^*_5 &= xy^2(\af_3p + \af_5y + \af_7r + \af_8t),\\
		\af^*_6 &= y(xy^2\af_6+(pq - rx + yz)(\af_7y - \af_8p) + fx\af_8),\\
		\af^*_7 &= xy^3(\af_7y-\af_8p),\\
		\af^*_8 &= x^2y^3\af_8.
	\end{align*}
	
We are only interested in those $\0$ with $\af_8\ne 0$, so for simplicity we assume that $\af_8=1$. Choosing $p=\af_7y$, $t=-(\af_7r + \af_3\af_7y + \af_5y)$, $z=-(\af_3x + \af_7q)$, $f=-\af_6y^2$, $s=\af_1x-\af_6q$, $h=(\af_3\af_6 + 2\af_1\af_7 - \af_2)y-\af_6r$ we obtain the representative $\la\af^\star_4\nb 4+\af^*_8\nb 8\ra$, where 
	$\af^\star_4=(\af_2\af_7 + \af_4-\af_3\af_6\af_7 - \af_1\af_7^2 - \af_5\af_6)y^3.$
	Hence, we have two representatives $\la\nb 8\ra$ and $\la\nb 4+\nb 8\ra$ depending on whether $\af_2\af_7 + \af_4-\af_3\af_6\af_7 - \af_1\af_7^2 - \af_5\af_6=0$ or not.

The algebras corresponding to $\la\nb 4+\nb 8\ra$ and $\la\nb 8\ra$ are:	
	
	\begin{longtable}{llllllll} 
	${\mathbb A}_{14}$ & : &  
		$e_1e_2=e_3$, & $e_1e_4=e_5$, & $e_2e_3=e_5$, & $e_2e_4=e_6$, & $e_3e_5=e_6$; \\
	${\mathbb A}_{15}$ & : &  
		$e_1e_2=e_3$, & $e_1e_4=e_5$, & $e_2e_3=e_5$, & $e_3e_5=e_6$.
	\end{longtable}

\subsubsection{$1$-dimensional  central extensions of ${\mathcal A}_{07}$}
Let us use the notation

\[ \nabla_1=[\Delta_{13}], \nabla_2=[\Delta_{14}], \nabla_3=[\Delta_{23}], \nabla_4=[\Delta_{24}],
\nabla_5=[\Delta_{15}], \nabla_6=[\Delta_{25}], \nabla_7=[\Delta_{35}], \nabla_8=[\Delta_{45}].  \]

The automorphism group of ${\mathcal A}_{07}$\ consists of the invertible
matrices of the form 
$$\phi=
\begin{pmatrix}
x & y& 0& 0& 0\\ 
z & v& 0& 0& 0\\ 
0 & 0& xv-yz& w& 0\\ 
0& 0& 0& t& 0\\ 
r& q& 0& h& (xv-yz)t \end{pmatrix}.$$
So, we must impose the condition $\det\phi=(xv-yz)^3t^2\ne 0$. For $\0=\sum_{i=1}^8\af_i\nb i\in {\rm H}^2_{{\mathcal A}}({\mathcal A}_{07})$, we get $\phi\theta=\sum\limits_{i=1}^8 \af_i^*  \nb i$,
where
$$\begin{array}{rcl}
\alpha^*_1 &=& (v x - y z) (x \alpha_1 + z \alpha_3 - r \alpha_7),\\
\alpha^*_2 &=& t x \alpha_2 + t z \alpha_4 + h x \alpha_5 + h z \alpha_6 + 
  w (x \alpha_1 + z \alpha_3 - r \alpha_7) - r t \alpha_8,\\
\alpha^*_3 &=& (v x - y z) (y \alpha_1 + v \alpha_3 - q \alpha_7),\\
\alpha^*_4 &=& t y \alpha_2 + t v \alpha_4 + h y \alpha_5 + h v \alpha_6 + 
  w (y \alpha_1 + v \alpha_3 - q \alpha_7) - q t \alpha_8,\\
\alpha^*_5 &=& t (v x - y z) (x \alpha_5 + z \alpha_6),\\
\alpha^*_6 &=& t (v x - y z) (y \alpha_5 + v \alpha_6),\\
\alpha^*_7 &=& t (v x - y z)^2 \alpha_7,\\
\alpha^*_8 &=& t (v x - y z) (w \alpha_7 + t \alpha_8).
\end{array}$$
 We are interested in cocycles with $(\af_5,\af_6,\af_7,\af_8)\ne(0,0,0,0).$ There are several cases to consider:
\begin{enumerate}
		
		\item $\alpha_7 \neq 0.$  
		\begin{enumerate}
			\item $(\af_5,\af_6)\ne (0,0).$ Then we may assume that $\af_5\ne 0$. Choosing 
			$z=0$, $y=-\frac{\af_6}{\af_7}$, $v=\frac{\af_5}{\af_7}$, $r=\frac{\af_1x}{\af_7}$, $q=\frac{\af_3\af_5 - \af_1\af_6}{\af_7^2}$, $w=-\frac{\af_8t}{\af_7}$, $h=\frac{(\af_1\af_8-\af_2\af_7)t}{\af_5\af_7}$ we have the representative $\langle \frac{\alpha^\star}{\af_5^2\af_7x^2} \nabla_4+ \nabla_5+\nabla_7 \rangle$, where
			\begin{align*}
			\af^\star = \af_4\af_5\af_7 - \af_2\af_6\af_7 - \af_3\af_5\af_8 + \af_1\af_6\af_8.
			\end{align*}
			Hence, we have two representatives $\langle \nabla_4+ \nabla_5+ \nabla_7 \rangle$ and $\langle \nabla_5+ \nabla_7 \rangle$ depending on whether $\af^\star=0$  or not.
			
			\item $(\af_5, \af_6) = (0,0).$ Choosing 
			$x=0$, $w=-\frac{\af_8t}{\af_7}$, $r=\frac{\alpha_3z}{\alpha_7}$, $q=\frac{\af_1y+\af_3v}{\af_7}$,
			we have the representative 
			$\langle \alpha_2^\star \nabla_2+ \alpha_4^\star \nabla_4+\alpha_7^\star \nabla_7 \rangle$, where
			\begin{align*}
			\af^\star_2 &= (\af_4\af_7 - \af_3\af_8)\frac{tz}{\af_7},\\
			\af^\star_4 &= ((\af_4\af_7 - \af_3\af_8)v + (\af_2\af_7 - \af_1\af_8)y)\frac t{\af_7},\\
			\af^\star_7 &= ty^2z^2\af_7.
			\end{align*}
			\begin{enumerate}
				\item $\af_4\af_7 - \af_3\af_8\ne 0$. Then choosing $v=\frac{\af_1\af_8-\af_2\af_7}{\af_4\af_7 - \af_3\af_8}$, $z=\frac{\af_4\af_7 - \af_3\af_8}{\af_7^2}$ and $y=1$, we obtain the representative $\langle \nabla_2+\nabla_7 \rangle$.
				\item $\af_4\af_7 - \af_3\af_8=0$. Then we have two representatives $\la\nb 7\ra$ and $\la\nb 4+\nb 7\ra$ depending on whether $\af_2\af_7 - \af_1\af_8=0$ or not.
			\end{enumerate}
			
		\end{enumerate}
		\item $\alpha_7=0.$ 
		
		\begin{enumerate}
			\item $\alpha_8\neq 0$ and $(\af_5, \af_6) \neq (0,0).$ Then we may suppose that $\af_6\neq 0.$
			Choosing any $x,z$ such that $\alpha_5x + \alpha_6z\ne 0$,
			$v = -\frac{y \alpha_5}{\alpha_6}$,
			$t = \frac{x \alpha_5 + z \alpha_6}{\alpha_8}$ and the values of $r$ and $q$ such that $\af^*_2=\af^*_4=0$, we have the representative 
			$\langle \alpha_1^\star \nabla_1+ \af^\star_3\nb 3+\nabla_5+\nabla_8 \rangle$, where
			\begin{align*}
			\af^\star_1 &= \frac{(\af_1x + \af_3z)\af_8}{(\af_5x + \af_6z)^2},\\
			\af^\star_3 &= \frac{(\af_1\af_6-\af_3\af_5)\af_8y}{(\af_5x + \af_6z)^2\af_6}.
			\end{align*}
			
			\begin{enumerate}
				\item $\af_1\af_6-\af_3\af_5\ne 0$. Then we can choose $x, z$ such that $\af_1x + \af_3z=\alpha_5x + \alpha_6z=\alpha_8\ne 0$. For these values of $x$ and $z$ and $y=\frac{\af_6\af_8}{\af_1\af_6-\af_3\af_5}$, we obtain the representative $\langle\nb 1 + \nabla_3+  \nabla_5+ \nabla_8 \rangle.$
				
				\item $\af_1\af_6-\af_3\af_5 = 0$. Then 
				we have two representatives 
				$\langle \nabla_5+ \nabla_8 \rangle$ and $\langle \nabla_1+  \nabla_5+ \nabla_8 \rangle$ depending on whether $\alpha_3=0$  or not.
				
			\end{enumerate}
			
			\item $\alpha_8\neq 0$ and $(\af_5, \af_6) = (0,0).$ Choosing 
			$q = \frac{w y \alpha_1 + t y \alpha_2 + v w \alpha_3 + t v \alpha_4}{t \alpha_8}$
			and 
			$r = \frac{w x \alpha_1 + t x \alpha_2 + w z \alpha_3 + t z \alpha_4}{t \alpha_8},$
			we have the representative 
			$\langle \alpha_1^\star \nabla_1+ \alpha_3^\star \nabla_3+\af^\star_8\nabla_8 \rangle$, where
			\begin{align*}
				\af^\star_1 &= (vx - yz)(\af_1x + \af_3z),\\
				\af^\star_3 &= (vx - yz)(\af_3v + \af_1y),\\
				\af^\star_8 &= (vx - yz)\af_8t^2.
			\end{align*}
			\begin{enumerate}
				\item $(\af_1,\af_3)\ne (0,0)$. Then we may assume that $\af_3\ne 0$. Choosing $x=0$, $y=\af_3$, $v=-\af_1$ and $z=\frac{\af_8 t^2}{\af_3}$, we get the representative $\langle \nabla_1+  \nabla_8 \rangle$.
				\item $(\af_1,\af_3)=(0,0)$. Then we have the representative $\langle \nabla_8 \rangle$.
			\end{enumerate}

			\item $\alpha_8=0$. Then $(\alpha_5, \alpha_6) \neq (0,0).$ We may suppose that $\alpha_6 \neq 0.$
			
			\begin{enumerate}
				\item $\alpha_3 \alpha_5 - \alpha_1 \alpha_6\ne 0$. Then we will  choose $x, z$ such that $\af_1x + \af_3z=\af_5x + \af_6z=\alpha_3 \alpha_5 - \alpha_1 \alpha_6\ne 0$. Take $y=-\af_6$ and $v=\af_5$, so that $\af_6^*=0$. Considering $\af^*_2=\af^*_4=0$ as a linear system in $h$ and $w$, we see that it has a unique solution $h=\frac{t(\alpha_1 \alpha_4- \alpha_2 \alpha_3)}{\alpha_3 \alpha_5-\alpha_1 \alpha_6}$, $w= \frac{t (\alpha_2 \alpha_6-\alpha_4 \alpha_5)}{\alpha_3 \alpha_5-\alpha_1 \alpha_6}$. Finally, taking $t=1$ we have the representative    $\langle \nabla_1+ \nabla_3+  \nabla_5 \rangle$.

				\item $\alpha_3 \alpha_5 = \alpha_1 \alpha_6.$
				Choosing 
				$v= -\frac{y  \alpha_5}{\alpha_6}$ and
				$h= -\frac{x(w \alpha_3 \alpha_5 + t \alpha_2 \alpha_6) + z(w \alpha_3 \alpha_6 +  t \alpha_4 \alpha_6)}{\alpha_6 (x \alpha_5 + z \alpha_6)}$, for any $x, z$ such that $\af_5x + \af_6z\neq 0$, 
				we obtain the representative 
				$\langle \af^\star_1 \nabla_1+ \af^\star_4 \nabla_4 + \af^\star_5 \nabla_5 \rangle$, where
				\begin{align*}
				\af^\star_1 &= -\frac{y \alpha_3 (x \alpha_5 + z \alpha_6)^2}{\alpha_6^2},\\
				\af^\star_4 &= \frac{t y (\alpha_2\alpha_6 - \alpha_4 \alpha_5)}{\alpha_6},\\
				\af^\star_5 &= -\frac{t y (x \alpha_5 + z \alpha_6)^2}{\alpha_6}.
				\end{align*}
				\begin{enumerate}
					\item $\alpha_3\neq 0$ and $\alpha_2\alpha_6= \alpha_4 \alpha_5$. Then we have the representative $\langle \nabla_1 +\nabla_5 \rangle.$
					\item $\alpha_3 = 0 $ and $\alpha_2\alpha_6= \alpha_4 \alpha_5$. Then we have the representative $\langle \nabla_5 \rangle.$
					\item $\alpha_3\neq 0 $ and $\alpha_2\alpha_6\neq \alpha_4 \alpha_5$. Then we have the representative $\langle \nabla_1 +\nabla_4+\nabla_5 \rangle.$ 
					\item $\alpha_3=0$ and $\alpha_2\alpha_6\neq \alpha_4 \alpha_5$. Then we have the representative $\langle \nabla_4+\nabla_5 \rangle.$

				\end{enumerate}
				
			\end{enumerate}

		\end{enumerate}
		
	\end{enumerate}

Analyzing the representatives found above, we see that $\orb\la\nb 2+\nb 7\ra=\orb\la\nb 4+\nb 7\ra$. The rest of the representatives belong to distinct orbits. These are 
	$\langle \nabla_1+ \nabla_3+  \nabla_5 \rangle$,
	$\langle\nb 1 + \nabla_3+  \nabla_5+ \nabla_8 \rangle$,
	$\langle \nabla_1 +\nabla_4+\nabla_5 \rangle$,
	$\langle \nabla_1 +\nabla_5 \rangle$,
	$\langle \nabla_1+  \nabla_5+ \nabla_8 \rangle$,
	$\langle \nabla_1+  \nabla_8 \rangle$,
	$\langle \nabla_2+\nabla_7 \rangle$,
	$\langle \nabla_4+\nabla_5 \rangle$
	$\langle \nabla_4+ \nabla_5+ \nabla_7 \rangle$,
	$\langle \nabla_5 \rangle$,
	$\langle \nabla_5+ \nabla_7 \rangle$,
	$\langle \nabla_5+ \nabla_8 \rangle$,
	$\la\nb 7\ra$,
	$\langle \nabla_8 \rangle$. 
	The corresponding algebras are:
	
	\begin{longtable}{llllllll} 
	${\mathbb A}_{16}$ & : &  
	$e_1e_2=e_3$, & $e_1e_3=e_6$, & $e_1e_5=e_6$, & $e_2e_3=e_6$, & $e_3e_4=e_5$; \\
	${\mathbb A}_{17}$ & : &  
	$e_1e_2=e_3$, & $e_1e_3=e_6$, & $e_1e_5=e_6$, & $e_2e_3=e_6$, & $e_3e_4=e_5$, & $e_4e_5=e_6$; \\
	${\mathbb A}_{18}$ & : &  
	$e_1e_2=e_3$, & $e_1e_3=e_6$, & $e_1e_5=e_6$, & $e_2e_4=e_6$, & $e_3e_4=e_5$; \\
	${\mathbb A}_{19}$ & : &  
	$e_1e_2=e_3$, & $e_1e_3=e_6$, & $e_1e_5=e_6$, & $e_3e_4=e_5$; \\
	${\mathbb A}_{20}$ & : &  
	$e_1e_2=e_3$, & $e_1e_3=e_6$, & $e_1e_5=e_6$, & $e_3e_4=e_5$, & $e_4e_5=e_6$; \\
	${\mathbb A}_{21}$ & : &  
	$e_1e_2=e_3$, & $e_1e_3=e_6$, & $e_3e_4=e_5$, & $e_4e_5=e_6$; \\
	${\mathbb A}_{22}$ & : &  
	$e_1e_2=e_3$, & $e_1e_4=e_6$, & $e_3e_4=e_5$, & $e_3e_5=e_6$; \\
	${\mathbb A}_{23}$ & : &  
	$e_1e_2=e_3$, & $e_1e_5=e_6$, & $e_2e_4=e_6$, & $e_3e_4=e_5$; \\
	${\mathbb A}_{24}$ & : &  
	$e_1e_2=e_3$, & $e_1e_5=e_6$, & $e_2e_4=e_6$, & $e_3e_4=e_5$, & $e_3e_5=e_6$; \\
	${\mathbb A}_{25}$ & : &  
	$e_1e_2=e_3$, & $e_1e_5=e_6$, & $e_3e_4=e_5$; \\
	${\mathbb A}_{26}$ & : &  
	$e_1e_2=e_3$, & $e_1e_5=e_6$, & $e_3e_4=e_5$, & $e_3e_5=e_6$; \\
	${\mathbb A}_{27}$ & : &  
	$e_1e_2=e_3$, & $e_1e_5=e_6$, & $e_3e_4=e_5$, & $e_4e_5=e_6$; \\
	${\mathbb A}_{28}$ & : &  
	$e_1e_2=e_3$, & $e_3e_4=e_5$, & $e_3e_5=e_6$; \\
	${\mathbb A}_{29}$ & : &  
	$e_1e_2=e_3$, & $e_3e_4=e_5$, & $e_4e_5=e_6$.
	\end{longtable}


\subsubsection{$1$-dimensional  central extensions of ${\mathcal A}_{08}$}
Let us use the notation
\[\nabla_1=[\Delta_{15}], \nabla_2= [\Delta_{23}],  \nabla_3=[\Delta_{24}], \nabla_4=[\Delta_{25}], 
\nabla_5=[\Delta_{34}], \nabla_6=[\Delta_{35}], \nabla_7=[\Delta_{45}] .\]
The automorphism group of ${\mathcal A}_{08}$\ consists of the invertible
matrices of the form 
$$\phi=
\begin{pmatrix}
x& 0& 0& 0& 0\\
z& y& 0& 0& 0\\
t& p& xy& 0& 0\\
q& r& xp& x^2y& 0\\
h&s&xr&x^2p&x^3y\end{pmatrix}.$$
So, we must have $\det\phi=x^7y^4\ne 0$. For $\0=\sum_{i=1}^7\af_i\nb i\in {\rm H}^2_{{\mathcal A}}({\mathcal A}_{08})$, we get $\phi\theta=\sum\limits_{i=1}^7 \af_i^*  \nb i$,
%
%
%
where
$$\begin{array}{rcl}
\alpha_1^* &=&x^3 y (x \alpha_1 + z \alpha_4 + t \alpha_6 + q \alpha_7), \\

\alpha_2^* &=&x (y (y \alpha_2 - r \alpha_5 - s \alpha_6) + r (y \alpha_4 + p \alpha_6 + r \alpha_7) +     p (y \alpha_3 + p \alpha_5 - s \alpha_7)), \\

\alpha_3^* &=& x^2 (p (y \alpha_4 + p \alpha_6 + r \alpha_7) + y (y \alpha_3 + p \alpha_5 - s \alpha_7)), \\

\alpha_4^* &=& x^3 y (y \alpha_4 + p \alpha_6 + r \alpha_7), \\

\alpha_5^* &=& x^3 (y^2 \alpha_5 + p y \alpha_6 + p^2 \alpha_7 - r y \alpha_7), \\ 
\alpha_6^* &=& x^4 y (y \alpha_6 + p \alpha_7), \\
\alpha_7^* &=& x^5 y^2 \alpha_7.
\end{array}$$
 We are interested only in those cocycles with $(\af_5,\af_6,\af_7)\ne(0,0,0)$.
 \begin{enumerate}
     \item $\alpha_7\neq 0.$ 
     Then choosing 
     $p=-\frac{y \alpha_6}{\alpha_7},$
     $r=\frac{y \alpha_5}{\alpha_7},$
     $q=-\frac{x \alpha_1+z \alpha_4+t \alpha_6}{ \alpha_7},$
     $s=\frac{y (\alpha_6^3-\alpha_4 \alpha_6 \alpha_7-2 \alpha_5 \alpha_6 \alpha_7+\alpha_3 \alpha_7^2)}{ \alpha_7^3}$    we have the representative  $\langle \af_2^\star \nabla_2 +\af_4^\star \nabla_4 +\alpha^*_7 \nabla_7\rangle$, where
     \begin{align*}
     	\af_2^\star &= \frac{x y^2}{\alpha_7}(\alpha_4 \alpha_5 - \alpha_3 \alpha_6 + \alpha_2 \alpha_7),\\
     	\af_4^\star &= \frac{x^3 y^2}{\af_7} (\alpha_4\af_7 + \alpha_5\af_7 - \alpha_6^2).
     \end{align*}
\begin{enumerate}
	\item $\alpha_7 \alpha_4 + \alpha_7\alpha_5 - \alpha_6^2\ne 0$. Then we have the family of  representatives $\langle \alpha \nabla_2+\nabla_4+ \nabla_7 \rangle$ of distinct orbits.
	
    \item $\alpha_7 \alpha_4 + \alpha_7\alpha_5 - \alpha_6^2=0$. Then we have two representatives $\langle \nabla_7 \rangle$ and $\la\nb 2+\nb 7\ra$ depending on whether $\alpha_4 \alpha_5 - \alpha_3 \alpha_6 + \alpha_2 \alpha_7=0$ or not.

\end{enumerate}

     \item $\alpha_7=0$ and $\alpha_6 \neq 0$. Then choosing 
     $t=-\frac{x \alpha_1+z \alpha_4}{ \alpha_6},$
     $p=-\frac{y \alpha_5}{\alpha_6},$
     $s=\frac{y(\alpha_5^3-\alpha_3 \alpha_5 \alpha_6+\alpha_2 \alpha_6^2)+r\alpha_6^2(\alpha_4 -2 \alpha_5)}{\alpha_6^3},$
     we have the representative $\langle \af^\star_3 \nabla_3 + \af^\star_4 \nabla_4 + \af^\star_6\nabla_6\rangle$, where
     \begin{align*}
     	\af^\star_3 &= \frac{x^2 y^2}{\alpha_6}(\alpha_3 \alpha_6-\alpha_4 \alpha_5),\\
     	\af^\star_4 &= x^3 y^2 (\alpha_4- \alpha_5),\\
     	\af^\star_6 &= x^4 y^2 \alpha_6.
     \end{align*}
        
     \begin{enumerate}
     	\item  $\alpha_4 \neq \af_5.$ Then we have the family of representatives $\langle \alpha \nabla_3+\nabla_4+\nabla_6\rangle$ of distinct orbits.
     	
        \item $\alpha_4=\af_5$. Then we have two representatives $\langle\nabla_6\rangle$ and $\langle\nb 3+\nabla_6\rangle$ depending on whether and $\alpha_4 \alpha_5 = \alpha_3 \alpha_6$ or not.
     \end{enumerate}

\item $\af_7=0, \af_6=0$ and $\af_5\neq0.$
\begin{enumerate}
    \item $\af_5 \neq \pm \af_4$. Then choosing 
    $p=-\frac{y \af_3}{\alpha_4 + \alpha_5}$ and
    $r= \frac{ y (  \alpha_3^2 \alpha_4 - \alpha_2 (\alpha_4 + \alpha_5)^2)}{(\alpha_4 - \alpha_5) (\alpha_4 + \alpha_5)^2},$
    we have the representative $\langle \af^\star_1\nabla_1 + \af^\star_4\nabla_4 + \af^\star_5\nabla_5 \rangle$, where
    \begin{align*}
    	\af^\star_1 &= x^3 y (x \alpha_1+z \alpha_4),\\
    	\af^\star_4 &= x^3 y^2\alpha_4,\\
    	\af^\star_5 &= x^3 y^2 \alpha_5.
    \end{align*}
\begin{enumerate}
	\item $\af_4\neq 0$. Then we have the family of representatives $\langle \af \nabla_4+ \nabla_5 \rangle_{\af \not\in\{0,\pm 1\}}$ of distinct orbits.
	\item $\af_4=0$ and $\af_1\neq 0$. Then we have the representative $\langle \nabla_1+ \nabla_5 \rangle.$
    \item $\af_4=0$ and $\af_1=0$. Then we have the representative $\langle \nabla_5 \rangle$ which will be joined with the family $\langle \af \nabla_4+ \nabla_5 \rangle_{\af \not\in\{0,\pm 1\}}$.
\end{enumerate}

\item $\af_5=\af_4$. Then by choosing 
$z=-\frac{x \alpha_1}{\alpha_4}$ and $p=-\frac{y \alpha_3}{ 2 \alpha_4}$ we have the representative $\langle \af^\star_2 \nabla_2 + \af^\star_4\nabla_4 + \af^\star_5\nabla_5\rangle$, where
	\begin{align*}
		\af^\star_2 &= \frac{x y^2}{4 \alpha_4} (4 \alpha_2 \alpha_4-\alpha_3^2),\\
		\af^\star_4 &= x^3 y^2 \alpha_4,\\
		\af^\star_5 &= x^3 y^2 \alpha_4.
	\end{align*}
\begin{enumerate}
    \item $4 \alpha_2 \alpha_4-\alpha_3^2\ne 0$. Then we have the representative $ \langle  \nabla_2 + \nabla_4 +\nabla_5 \rangle.$
    \item $4 \alpha_2 \alpha_4-\alpha_3^2=0$. Then we have the representative $ \langle   \nabla_4 +\nabla_5 \rangle$ which will be joined with the family $\langle \af \nabla_4+ \nabla_5 \rangle_{\af \not\in\{0,\pm 1\}}$.

\end{enumerate}

\item $\af_5=-\af_4$. Then by choosing 
$z = -\frac{x \alpha_1}{\alpha_4}$ and $r= \frac{p^2 \alpha_4-y^2 \alpha_2-p y \alpha_3}{2 y \alpha_4}$ we have the representative $\langle \af^\star_3 \nabla_3 + \af^\star_4\nabla_4 +\af^\star_5\nabla_5 \rangle$, where 
	\begin{align*}
		\af^\star_3 &= x^2 y^2 \alpha_3,\\
		\af^\star_4 &= x^3 y^2 \alpha_4,\\
		\af^\star_5 &= -x^3 y^2 \alpha_4.
	\end{align*}
\begin{enumerate}
    \item $\alpha_3 \neq 0$. Then we have the representative $ \langle  \nabla_3 - \nabla_4 +\nabla_5 \rangle.$
    \item $\alpha_3=0$. Then we have the representative $ \langle  - \nabla_4 +\nabla_5 \rangle$ which will be joined with the family $\langle \af \nabla_4+ \nabla_5 \rangle_{\af \not\in\{0,\pm 1\}}$.

\end{enumerate}

\end{enumerate}

 \end{enumerate}
 
 Summarizing, we obtain the following representatives:
 $\langle \nabla_1+ \nabla_5 \rangle$,
 $\langle\nabla_2 + \nabla_4 +\nabla_5 \rangle$,  
 $\langle \alpha \nabla_2+\nabla_4+ \nabla_7 \rangle$,  
 $\la\nb 2+\nb 7\ra$,
 $\langle \nabla_3 - \nabla_4 +\nabla_5 \rangle$,
 $\langle \alpha \nabla_3+\nabla_4+\nabla_6\rangle$,
 $\langle\nb 3+\nabla_6\rangle$,
 $\langle \af \nabla_4+ \nabla_5 \rangle$,
 $\langle\nabla_6\rangle$,
 $\langle \nabla_7 \rangle$.
 All of these representatives belong to distinct orbits.
 The corresponding algebras are:
 
 \begin{longtable}{llllllll} 
 ${\mathbb A}_{30}$ & : &  
 $e_1e_2=e_3$, & $e_1e_3=e_4$, & $e_1e_4=e_5$, & $e_1e_5 = e_6$, & $e_3e_4 = e_6$;\\
 ${\mathbb A}_{31}$ & : &  
 $e_1e_2=e_3$, & $e_1e_3=e_4$, & $e_1e_4=e_5$, & $e_2e_3 = e_6$, & $e_2e_5 = e_6$, & $e_3e_4 = e_6$;\\
 ${\mathbb A}_{32}(\af)$ & : &  
 $e_1e_2=e_3$, & $e_1e_3=e_4$, & $e_1e_4=e_5$, & $e_2e_3 = \af e_6$, & $e_2e_5 = e_6$, & $e_4e_5 = e_6$;\\
 ${\mathbb A}_{33}$ & : &  
 $e_1e_2=e_3$, & $e_1e_3=e_4$, & $e_1e_4=e_5$, & $e_2e_3 = e_6$, & $e_4e_5 = e_6$;\\
 ${\mathbb A}_{34}$ & : &  
 $e_1e_2=e_3$, & $e_1e_3=e_4$, & $e_1e_4=e_5$, & $e_2e_4 = e_6$, & $e_2e_5 = -e_6$, & $e_3e_4 = e_6$;\\
 ${\mathbb A}_{35}(\af)$ & : &  
 $e_1e_2=e_3$, & $e_1e_3=e_4$, & $e_1e_4=e_5$, & $e_2e_4 = \af e_6$, & $e_2e_5 = e_6$, & $e_3e_5 = e_6$;\\
 ${\mathbb A}_{36}$ & : &  
 $e_1e_2=e_3$, & $e_1e_3=e_4$, & $e_1e_4=e_5$, & $e_2e_4 = e_6$, & $e_3e_5 = e_6$;\\
 ${\mathbb A}_{37}(\af)$ & : &  
 $e_1e_2=e_3$, & $e_1e_3=e_4$, & $e_1e_4=e_5$, & $e_2e_5 = \af e_6$, & $e_3e_4 = e_6$;\\
 ${\mathbb A}_{38}$ & : &  
 $e_1e_2=e_3$, & $e_1e_3=e_4$, & $e_1e_4=e_5$, & $e_3e_5 = e_6$;\\
 ${\mathbb A}_{39}$ & : &  
 $e_1e_2=e_3$, & $e_1e_3=e_4$, & $e_1e_4=e_5$, & $e_4e_5 = e_6$.
 \end{longtable}



\subsubsection{$1$-dimensional  central extensions of ${\mathcal A}_{09}$}
Let us use the notation
\[
\nabla_1=[\Delta_{14}], \nabla_2= [\Delta_{15}],  \nabla_3=[\Delta_{24}], \nabla_4=[\Delta_{25}],
\nabla_5=[\Delta_{34}], \nabla_6=[\Delta_{35}],  \nabla_7=[\Delta_{45}].\\
\]
The automorphism group of ${\mathcal A}_{09}$\ consists of the invertible
matrices of the form 
$$\phi=
\begin{pmatrix}
x& 0& 0& 0& 0\\
y& x^2& 0& 0& 0\\
z& t& x^3& 0& 0\\
p& q& xt& x^4& 0\\
r&s&-x^2z+xq+yt&x^3y+x^2t&x^5\end{pmatrix}.$$
Thus, $\det\phi=x^{15}\ne 0$. For $\0=\sum_{i=1}^7\af_i\nb i\in {\rm H}^2_{{\mathcal A}}({\mathcal A}_{09})$, we get $\phi\theta=\sum\limits_{i=1}^7 \af_i^*  \nb i$,
where
$$
\begin{array}{rcl}
\alpha_1^*&=&x^3 (q \alpha_5+s \alpha_6)+x^2 (t+x y) (x \alpha_2+y \alpha_4+z \alpha_6+p \alpha_7)-\\&&(q x+t y-x^2 z) (x^2 \alpha_4+t \alpha_6+q \alpha_7)+x^4 (x \alpha_1+y \alpha_3+z \alpha_5-r \alpha_7)-t x (x^2 \alpha_3+t \alpha_5-s \alpha_7),\\

\alpha_2^* &=&x^5 (x \alpha_2 + y \alpha_4 + z \alpha_6 + p \alpha_7),\\

\alpha_3^* &=& x^2 (t + x y) (x^2 \alpha_4 + t \alpha_6 + q \alpha_7) + 
 x^4 (x^2 \alpha_3 + t \alpha_5 - s \alpha_7),\\

\alpha_4^*&=& x^5 (x^2 \alpha_4 + t \alpha_6 + q \alpha_7), \\
\alpha_5^*&=&  x^3(x^4 \alpha_5 + t^2 \alpha_7 + x^2 (t \alpha_6 - q \alpha_7) + 
  x^3 (y \alpha_6 + z \alpha_7)),\\
\alpha_6^*&=&  x^6(x^2 \alpha_6 + t \alpha_7), \\   
\alpha_7^*&=&  x^9 \alpha_7.
\end{array}
$$
We are only interested in cocycles with $(\af_5,\af_6,\af_7)\ne(0,0,0)$.

\begin{enumerate}
    \item $\alpha_7\ne 0$. Then choosing $y=0$,
        $t=-\frac{\af_6x^2}{\af_7}$, 
        $q=\frac{x^2}{\af_7^2}(\af_6^2 - \af_4\af_7)$,
        $z=\frac x{\af_7^2}(\af_6^2 - \af_4\af_7 - \af_5\af_7)$,
        $s=\frac{x^2}{\af_7^2}(\af_3\af_7-\af_5\af_6)$,
        $p=\frac x{\af_7^3}(\af_4\af_6\af_7 + \af_5\af_6\af_7 - \af_2\af_7^2 -\af_6^3)$,
        $r=\frac x{\af_7^3}(\af_5\af_6^2 - 2\af_4\af_5\af_7 - \af_5^2\af_7 + \af_3\af_6\af_7 + \af_1\af_7^2)$,
        we have the representative 
        $\langle \nabla_7 \rangle$.

    \item $\alpha_7=0$ and $\af_6 \ne 0.$ Then choosing
    $q=0$,
    $t=-\frac{\af_4x^2}{\af_6}$,
    $y=\frac x{\af_6}(\af_4 - \af_5)$,
    $z=\frac x{\af_6^2}(\af_4\af_5 - \af_2\af_6-\af_4^2)$,
    $s=\frac{x^2}{\af_6^3}(2\af_4^2\af_5 - \af_4\af_5^2 - 2\af_3\af_4\af_6 + \af_2\af_5\af_6 + \af_3\af_5\af_6 -\af_1\af_6^2)$
    we have the representative $\la\as 3\nb 3+\as 6\nb 6\ra$, where
    \begin{align*}
    	\as 3 &= \frac{x^6}{\alpha_6} (\alpha_3\alpha_6-\alpha_4 \alpha_5),\\
    	\as 6 &= x^8 \alpha_6.
    \end{align*}
    Then we have two representatives $\langle \nabla_6 \rangle$ and $\langle  \nabla_3+ \nabla_6 \rangle$ depending on whether $\alpha_3\alpha_6-\alpha_4 \alpha_5=0$ or not.
    
    \item $\af_7=0, \af_6=0$ and $\af_5 \neq 0.$ 
    \begin{enumerate}
        \item $\af_5 \neq \pm\af_4$ and $\af_4 \neq 0$. Then choosing 
        $z=0$,
        $y=-\frac{\af_2x}{\af_4}$,
        $t=\frac{(\af_2 - \af_3)x^2}{\af_4 + \af_5}$,
        $q=\frac{x^2}{\af_4(\af_4 + \af_5)^2(\af_4 - \af_5)}(\af_2^2\af_4^2 - 3\af_2\af_3\af_4^2 + \af_3^2\af_4^2 + \af_1\af_4^3 - 2\af_2\af_3\af_4\af_5 + 2\af_1\af_4^2\af_5 - \af_2\af_3\af_5^2 + \af_1\af_4\af_5^2)$ we have the family of representatives  $\langle  \alpha \nabla_4+ \nabla_5 \rangle_{\alpha\not\in \{0,\pm 1\}}$ of distinct orbits.

        \item $\af_5 =\af_4$. Then choosing 
        $y=-\frac{\af_2x}{\af_4}$,
        $t=\frac{x^2}{2\af_4}(\af_2 - \af_3)$,
        $z=-\frac x{8\af_4^2}(\af_2^2 - 6\af_2\af_3 + \af_3^2 + 4\af_1\af_4)$,
        we have the  representative $\langle \nabla_4+ \nabla_5 \rangle$ which will be joined with the family $\langle \alpha \nabla_4+ \nabla_5 \rangle_{\alpha\not\in \{0,\pm 1\}}$.
  
    \item $\af_5 =-\af_4$. Then choosing
        $t=0$,
        $y=-\frac{\af_3x}{\af_4}$,
        $q=\frac {x^2}{2\af_4^2}(\af_1\af_4-\af_2\af_3)$  
         we have the  representative  $\la\as 2\nb 2+\as 4\nb 4+\as 5\nb 5\ra$, where
         \begin{align*}
         \as 2 &= (\af_2- \af_3)x^6,\\
         \as 4 &= \af_4x^7,\\
         \as 5 &= -\af_4x^7.
         \end{align*}
         So, we have two representatives $\la\nb 4-\nb 5\ra$ and $\la\nb 2+\nb 4-\nb 5\ra$ depending on whether $\af_2= \af_3$ or not. The representative $\la\nb 4-\nb 5\ra$ will be joined with the family $\langle \alpha \nabla_4+ \nabla_5 \rangle_{\alpha\not\in \{0,\pm 1\}}$.

  \item $\af_4=0$. Then choosing $y=z=0$,
        $t=-\frac{\af_3x^2}{\af_5}$,
        $q=\frac{x^2}{\af_5^2}(\af_2\af_3 - \af_1\af_5)$,
        we have the  representative $\la\as 2\nb 2+\as 5\nb 5\ra$, where
        \begin{align*}
        	\as 2 &= \af_2x^6,\\
        	\as 5 &= \af_5x^7.
        \end{align*} 
         So, we have two representatives $\langle  \nabla_5 \rangle$ and $\langle \nabla_2 + \nabla_5 \rangle$ depending on whether $\af_2=0$ or not. The representative $\langle  \nabla_5 \rangle$  will be joined with the family $\langle \alpha \nabla_4+ \nabla_5 \rangle_{\alpha\not\in \{0,\pm 1\}}$.
    \end{enumerate}
\end{enumerate}

Summarizing, we obtain the following representatives:
$\la\nb 2+\nb 4-\nb 5\ra$,  
$\langle \nabla_2 + \nabla_5 \rangle$, 
$\langle  \nabla_3+ \nabla_6 \rangle$, 
$\langle  \alpha \nabla_4+ \nabla_5 \rangle$, 
$\langle \nabla_6 \rangle$, 
$\langle \nabla_7 \rangle$.
All of these representatives belong to distinct orbits.
The corresponding algebras are:

\begin{longtable}{llllllllll} 
${\mathbb A}_{40}$ & : &  
$e_1e_2=e_3$, & $e_1e_3=e_4$, & $e_1e_4=e_5$, & $e_1e_5=e_6$, & $e_2e_3=e_5$, & $e_2e_5=e_6$, & $e_3e_4=-e_6$;\\
${\mathbb A}_{41}$ & : &  
$e_1e_2=e_3$, & $e_1e_3=e_4$, & $e_1e_4=e_5$, & $e_1e_5=e_6$, & $e_2e_3=e_5$, & $e_3e_4=e_6$;\\
${\mathbb A}_{42}$ & : &  
$e_1e_2=e_3$, & $e_1e_3=e_4$, & $e_1e_4=e_5$, & $e_2e_3=e_5$, & $e_2e_4=e_6$, & $e_3e_5=e_6$;\\
${\mathbb A}_{43}(\af)$ & : &  
$e_1e_2=e_3$, & $e_1e_3=e_4$, & $e_1e_4=e_5$, & $e_2e_3=e_5$, & $e_2e_5=\af e_6$, & $e_3e_4=e_6$;\\
${\mathbb A}_{44}$ & : &  
$e_1e_2=e_3$, & $e_1e_3=e_4$, & $e_1e_4=e_5$, & $e_2e_3=e_5$, & $e_3e_5=e_6$;\\
${\mathbb A}_{45}$ & : &  
$e_1e_2=e_3$, & $e_1e_3=e_4$, & $e_1e_4=e_5$, & $e_2e_3=e_5$, & $e_4e_5=e_6$.
\end{longtable}



\subsubsection{$1$-dimensional  central extensions of ${\mathcal A}_{10}$}
Let us use the notation
\[\nabla_1=[\Delta_{14}], \nabla_2= [\Delta_{23}],  \nabla_3=[\Delta_{15}]+[\Delta_{34}],
\nabla_4= [\Delta_{15}], \nabla_5=[\Delta_{25}], \nabla_6=[\Delta_{35}],  \nabla_7=[\Delta_{45}].
\]
The automorphism group of ${\mathcal A}_{10}$\ consists of the invertible
matrices of the form 
$$\phi=
\begin{pmatrix}
x& 0& 0& 0& 0\\
0& y& 0& 0& 0\\
z& 0& xy& 0& 0\\
p&q& 0& x^2y& 0\\
r& h& -yp& 0& y^2x^2
\end{pmatrix}.$$
Thus, $\det\phi=x^6y^5\ne 0$. For $\0=\sum_{i=1}^7\af_i\nb i\in {\rm H}^2_{{\mathcal A}}({\mathcal A}_{10})$, we get $\phi\theta=\sum\limits_{i=1}^7 \af_i^*  \nb i$,
%
where
$$\begin{array}{rcl}
\alpha^*_1 &=&x^2 y (x \af_1 + z \af_3 - r \af_7),\\
\alpha^*_2 &=&x y (y \af_2 - q \af_3 - h \af_6) - p y (y \af_5 + q \af_7),\\
\alpha^*_3 &=&x^2 y^2 (x \af_3 + p \af_7),\\
\alpha^*_4 &=&x^2 y^2 (x \af_4 + z \af_6),\\
\alpha^*_5 &=&x^2 y^2 (y \af_5 + q \af_7),\\
\alpha^*_6 &=&x^3 y^3 \af_6,\\
\alpha^*_7 &=&x^4 y^3 \af_7.
\end{array}$$
We are only interested in cocycles with $(\alpha_4,\alpha_5, \af_6, \af_7) \neq (0,0,0,0).$

\begin{enumerate}
    \item $\alpha_7 \neq 0$ and $\af_6\neq 0$.
    Then choosing 
    $x= \frac{\af_6}{\af_7},$ 
    $z= -\frac{x \af_4}{\af_6},$ 
    $p=-\frac{x \af_3}{\af_7},$ 
    $h= \frac{y (\af_3 \af_5 +  \af_2 \af_7)}{\af_6 \af_7},$
    $r=\frac{x \af_1+z \af_3}{\af_7},$ 
    $q=-\frac{y \af_5}{\af_7}$ we have the representative 
    $ \langle  \nabla_6 +  \nabla_7\rangle.$

\item $\af_7\neq 0$ and $\af_6=0.$ Then by choosing 
$q = -\frac{y \af_5}{\af_7},$
$p = -\frac{x \af_3}{\af_7},$
$r = \frac{x \af_1 + z \af_3}{\af_7},$ we have the representative $\la\as 2\nb 2+\as 4\nb 4+\as 7\nb 7\ra$, where
\begin{align*}
	\as 2 &= \frac{x y^2}{\af_7} ( \af_2\af_7+ \af_3 \af_5),\\
	\as 4 &= x^3 y^2 \af_4,\\
	\as 7 &= x^4 y^3 \af_7.
\end{align*}

\begin{enumerate}
    \item $\af_2\af_7+ \af_3 \af_5 \neq 0$ and $\af_4 \neq 0$. Then we have the representative $\langle \nabla_2+\nabla_4+\nabla_7 \rangle.$

    \item $\af_2\af_7+ \af_3 \af_5 =0$ and $\af_4 \neq 0$. Then we have the representative $\langle \nabla_4+\nabla_7 \rangle.$

    \item $\af_2\af_7+ \af_3 \af_5 \neq 0$ and $\af_4 = 0$. Then we have the representative $\langle \nabla_2+\nabla_7 \rangle.$

    \item $\af_2\af_7+ \af_3 \af_5 = 0$ and $\af_4 = 0$. Then we have the representative $\langle \nabla_7 \rangle.$

\end{enumerate}

\item $\af_7 = 0$ and $\af_6\neq 0.$ 
Then choosing 
$z = - \frac{x \af_4}{\af_6},$
$h = \frac{x y \af_2 - q x \af_3 - p y \af_5}{x \af_6},$
we have the representative $\la\as 1\nb 1+\as 2\nb 2+\as 3\nb 3+\as 5\nb 5+\as 6\nb 6\ra$, where
\begin{align*}
	\as 1 &= \frac{x^3 y}{\af_6} (\af_1\af_6- \af_3 \af_4),\\
	\as 2 &= (qx\af_3 + py\af_5 - xy\af_2)(x - 1)y,\\
	\as 3 &= x^3 y^2 \af_3,\\
	\as 5 &= x^2 y^3 \af_5,\\
	\as 6 &= x^3 y^3 \af_6.
\end{align*}

\begin{enumerate}

    \item $\af_1\af_6-\af_3 \af_4=0$, $\af_3=0 $ and $\af_5=0$. Then choosing $x=1$ we have the representative  $\langle \nabla_6 \rangle.$

    \item $\af_1\af_6-\af_3 \af_4=0$, $\af_3=0 $ and $\af_5\neq 0$. Then choosing $p=\frac{\af_2}{\af_6}$ and $x=\frac{\af_5}{\af_6}$ we have the representative  $\langle \nabla_5+\nabla_6 \rangle.$

    \item $\af_1\af_6-\af_3 \af_4=0$, $\af_3\neq 0 $ and $\af_5= 0$. Then choosing $q=\frac{\af_2}{\af_6}$ and $y=\frac{\af_3}{\af_6}$ we have the representative  $\langle \nabla_3 + \nabla_6 \rangle.$

    \item $\af_1\af_6-\af_3 \af_4=0$, $\af_3\neq 0 $ and $\af_5\neq 0$. Then choosing $p=\frac{\af_2-q\af_6}{\af_6}$, $x=\frac{\af_5}{\af_6}$ and $y=\frac{\af_3}{\af_6}$ we have the representative  $\langle \nabla_3 +\nabla_5 + \nabla_6 \rangle.$
 
    \item $\af_1\af_6-\af_3 \af_4\ne 0$, $\af_3=0$ and $\af_5=0$. Then choosing $x=1$ and $y=\sqrt{\frac{\af_1}{\af_6}}$ we have the representative  $\langle \nabla_1+\nabla_6 \rangle.$

    \item $\af_1\af_6-\af_3 \af_4\ne 0$, $\af_3=0 $ and $\af_5\neq 0$. Then choosing $p=\frac{\af_2}{\af_6}$, $x=\frac{\af_5}{\af_6}$ and $y=\sqrt{\frac{\af_1}{\af_6}}$ we have the representative  $\langle \nabla_1+\nabla_5+\nabla_6 \rangle.$

    \item $\af_1\af_6-\af_3 \af_4\ne 0$, $\af_3\neq 0$ and $\af_5= 0$. Then choosing $q=\frac{\af_2}{\af_6}$ and $y=\frac{\af_3}{\af_6}$ we have the family of representatives $\langle \af\nabla_1+\nabla_3 + \nabla_6 \rangle_{\af\ne 0}$ of distinct orbits. We will join $\langle \nabla_3 + \nabla_6 \rangle$ with this family.

    \item $\af_1\af_6-\af_3 \af_4\ne 0$, $\af_3\neq 0$ and $\af_5\neq 0$. Then choosing $p=\frac{\af_2-q\af_6}{\af_6}$, $x=\frac{\af_5}{\af_6}$ and $y=\frac{\af_3}{\af_6}$ we have the family of representatives  $\langle \af \nabla_1+\nabla_3 +\nabla_5 + \nabla_6 \rangle_{\alpha \neq 0}$ of distinct orbits. We will join $\langle \nabla_3 +\nabla_5 + \nabla_6 \rangle$ with this family.

\end{enumerate}

\item $\af_7=0, \af_6 = 0, \af_5\neq0.$ 
Then choosing $p =\frac{x (y \af_2 - q \af_3)}{y \af_5},$ we have the representative $\la\as 1\nb 1+\as 3\nb 3+\as 4\nb 4+\as 5\nb 5\ra$, where
\begin{align*}
	\as 1 &= x^2 y (x \af_1+z \af_3),\\
	\as 3 &= x^3 y^2 \af_3,\\
	\as 4 &= x^3 y^2 \af_4,\\
	\as 5 &= x^2 y^3 \af_5.
\end{align*} 
\begin{enumerate}
    \item $\af_3\neq 0$ and $\af_4 = 0$. Then we have the representative  $\langle \nabla_3 +  \nabla_5 \rangle.$
    \item $\af_3\neq 0$ and $\af_4 \neq 0$. Then we have the family of representatives  $\langle \af \nabla_3 +  \nabla_4+ \nabla_5 \rangle_{\af \neq 0}.$
    \item $\af_3= 0$, $\af_4 = 0$ and $\af_1=0$. Then we have the representative  $\langle   \nabla_5 \rangle.$
    \item $\af_3= 0$, $\af_4 = 0$ and $\af_1\neq 0$. Then we have the representative  $\langle  \nabla_1+ \nabla_5 \rangle.$
 	\item $\af_3= 0$, $\af_4 \neq  0$ and $\af_1=0$. Then we have the  representative  $\langle \nabla_4+ \nabla_5 \rangle$ which will be joined with the family $\langle \af \nabla_3 +  \nabla_4+ \nabla_5 \rangle_{\af \neq 0}$.
 	\item $\af_3= 0$, $\af_4 \neq  0$ and $\af_1\neq 0$. Then we have the  representative  $\langle \nabla_1+\nabla_4+ \nabla_5 \rangle.$

\end{enumerate}

\item $\af_7=0$, $\af_6=0$, $\af_5=0$ and $\af_4 \neq0$. Then we have the representative $\la\as 1\nb 1+\as 2\nb 2+\as 3\nb 3+\as 4\nb 4\ra$, where
\begin{align*}
	\as 1 &= x^2 y (x \af_1+z \af_3),\\
	\as 2 &= x y (y \af_2-q \af_3),\\
	\as 3 &= x^3 y^2 \af_3,\\
	\as 4 &= x^3 y^2 \af_4.
\end{align*}

\begin{enumerate}
    \item $\af_3=0$, $\af_1=0$ and $\af_2=0$. Then we have the representative $\langle \nabla_4 \rangle.$
    \item $\af_3=0$, $\af_1\neq 0$ and $\af_2=0$. Then we have the representative $\langle \nabla_1+\nabla_4 \rangle.$
    \item $\af_3=0$, $\af_1= 0$ and $\af_2\neq 0$. Then we have the representative $\langle \nabla_2+\nabla_4 \rangle.$
    \item $\af_3=0$, $\af_1\neq 0$ nd $\af_2\neq 0$. Then we have the representative $\langle \nabla_1+\nabla_2+\nabla_4 \rangle.$
    \item $\af_3\neq 0$. Then we have the family of representatives $\langle \af\nabla_3+\nabla_4 \rangle_{\af\neq 0}$ of distinct orbits. We will join $\langle \nabla_4 \rangle$ with this family.

\end{enumerate}

\end{enumerate}

Summarizing, we obtain the following representatives:
$\langle \nabla_1+\nabla_2+\nabla_4 \rangle$,
$\langle \af \nabla_1+\nabla_3 +\nabla_5 + \nabla_6 \rangle$,
$\langle \af\nabla_1+\nabla_3 + \nabla_6 \rangle$, 
$\langle \nabla_1+\nabla_4 \rangle$,
$\langle \nabla_1+\nabla_4+ \nabla_5 \rangle$,
$\langle \nabla_1+ \nabla_5 \rangle$,
$\langle \nabla_1+\nabla_5+\nabla_6 \rangle$, 
$\langle \nabla_1+\nabla_6 \rangle$,
$\langle \nabla_2+\nabla_4 \rangle$,
$\langle \nabla_2+\nabla_4+\nabla_7 \rangle$,
$\langle \nabla_2+\nabla_7 \rangle$,
$\langle \af\nabla_3+\nabla_4 \rangle$,
$\langle \af \nabla_3 +  \nabla_4+ \nabla_5 \rangle$,
$\langle \nabla_3 +  \nabla_5 \rangle$,
$\langle \nabla_4+\nabla_7 \rangle$,
$\langle \nabla_5 \rangle$,
$\langle \nabla_5+\nabla_6 \rangle$, 
$\langle \nabla_6 \rangle$,
$\langle  \nabla_6 +  \nabla_7\rangle$,
$\langle \nabla_7 \rangle$.
All of them belong to distinct orbits.
The corresponding algebras are:

\begin{longtable}{llllllllll} 
${\mathbb A}_{46}$ & : &  
$e_1e_2=e_3$, & $e_1e_3=e_4$, & $e_1e_4=e_6$, & $e_1e_5=e_6$, & $e_2e_3=e_6$, & $e_2e_4=e_5$;\\
${\mathbb A}_{47}(\af)$ & : &  
$e_1e_2=e_3$, & $e_1e_3=e_4$, & $e_1e_4=\af e_6$, & $e_1e_5=e_6$, & $e_2e_4=e_5$, & $e_2e_5=e_6$, & $e_3e_4=e_6$, & $e_3e_5=e_6$;\\
${\mathbb A}_{48}(\af)$ & : &  
$e_1e_2=e_3$, & $e_1e_3=e_4$, & $e_1e_4=\af e_6$, & $e_1e_5=e_6$, & $e_2e_4=e_5$, & $e_3e_4=e_6$, & $e_3e_5=e_6$;\\
${\mathbb A}_{49}$ & : &  
$e_1e_2=e_3$, & $e_1e_3=e_4$, & $e_1e_4=e_6$, & $e_1e_5=e_6$, & $e_2e_4=e_5$;\\
${\mathbb A}_{50}$ & : &  
$e_1e_2=e_3$, & $e_1e_3=e_4$, & $e_1e_4=e_6$, & $e_1e_5=e_6$, & $e_2e_4=e_5$, & $e_2e_5=e_6$;\\
${\mathbb A}_{51}$ & : &  
$e_1e_2=e_3$, & $e_1e_3=e_4$, & $e_1e_4=e_6$, & $e_2e_4=e_5$, & $e_2e_5=e_6$;\\
${\mathbb A}_{52}$ & : &  
$e_1e_2=e_3$, & $e_1e_3=e_4$, & $e_1e_4=e_6$, & $e_2e_4=e_5$, & $e_2e_5=e_6$, & $e_3e_5=e_6$;\\
${\mathbb A}_{53}$ & : &  
$e_1e_2=e_3$, & $e_1e_3=e_4$, & $e_1e_4=e_6$, & $e_2e_4=e_5$, & $e_3e_5=e_6$;\\
${\mathbb A}_{54}$ & : &  
$e_1e_2=e_3$, & $e_1e_3=e_4$, & $e_1e_5=e_6$, & $e_2e_3=e_6$, & $e_2e_4=e_5$;\\
${\mathbb A}_{55}$ & : &  
$e_1e_2=e_3$, & $e_1e_3=e_4$, & $e_1e_5=e_6$, & $e_2e_3=e_6$, & $e_2e_4=e_5$, & $e_4e_5=e_6$;\\
${\mathbb A}_{56}$ & : &  
$e_1e_2=e_3$, & $e_1e_3=e_4$, & $e_2e_3=e_6$, & $e_2e_4=e_5$, & $e_4e_5=e_6$;\\
${\mathbb A}_{57}(\af)$ & : &  
$e_1e_2=e_3$, & $e_1e_3=e_4$, & $e_1e_5=(\af+1)e_6$, & $e_2e_4=e_5$, & $e_3e_4=\af e_6$;\\
${\mathbb A}_{58}(\af)$ & : &  
$e_1e_2=e_3$, & $e_1e_3=e_4$, & $e_1e_5=(\af+1)e_6$, & $e_2e_4=e_5$, & $e_2e_5=e_6$, & $e_3e_4=\af e_6$;\\
${\mathbb A}_{59}$ & : &  
$e_1e_2=e_3$, & $e_1e_3=e_4$, & $e_1e_5=e_6$, & $e_2e_4=e_5$, & $e_2e_5=e_6$, & $e_3e_4=e_6$;\\
${\mathbb A}_{60}$ & : &  
$e_1e_2=e_3$, & $e_1e_3=e_4$, & $e_1e_5=e_6$, & $e_2e_4=e_5$, & $e_4e_5=e_6$;\\
${\mathbb A}_{61}$ & : &  
$e_1e_2=e_3$, & $e_1e_3=e_4$, & $e_2e_4=e_5$, & $e_2e_5=e_6$;\\
${\mathbb A}_{62}$ & : &  
$e_1e_2=e_3$, & $e_1e_3=e_4$, & $e_2e_4=e_5$, & $e_2e_5=e_6$, & $e_3e_5=e_6$;\\
${\mathbb A}_{63}$ & : &  
$e_1e_2=e_3$, & $e_1e_3=e_4$, & $e_2e_4=e_5$, & $e_3e_5=e_6$;\\
${\mathbb A}_{64}$ & : &  
$e_1e_2=e_3$, & $e_1e_3=e_4$, & $e_2e_4=e_5$, & $e_3e_5=e_6$, & $e_4e_5=e_6$;\\
${\mathbb A}_{65}$ & : &  
$e_1e_2=e_3$, & $e_1e_3=e_4$, & $e_2e_4=e_5$, & $e_4e_5=e_6$.
\end{longtable}

\subsubsection{$1$-dimensional  central extensions of ${\mathcal A}_{11}$}
Let us use the notation
\[
\nabla_1= [\Delta_{14}], \nabla_2=[\Delta_{23}], \nabla_3 = [\Delta_{24}], \nabla_4=[\Delta_{15}], \nabla_5 = [\Delta_{25}], \nabla_6= [\Delta_{35}], \nabla_7=[\Delta_{45}]. 
\]

The automorphism group of ${\mathcal A}_{11}$\ consists of the invertible
matrices of the form 
$$\phi=
\begin{pmatrix}
x&  0&   0&   0&         0\\
z&  y&   0&   0&         0\\
0&  0&  xy&   0&         0\\
u&  0&   0&  x^2y&       0\\
w&  v&   0& -xyu&   x^3y^2\\
\end{pmatrix}.$$
So, we must have $\det\phi=x^7y^5\ne 0$. For $\0=\sum_{i=1}^7\af_i\nb i\in {\rm H}^2_{{\mathcal A}}({\mathcal A}_{11})$, we get $\phi\theta=\sum\limits_{i=1}^7 \af_i^*  \nb i$,
%
where
$$\begin{array}{rcl}

\alpha^*_1 &=& x y (-u (x \af_4 + z \af_5 + u \af_7) +     x (x \af_1 + z \af_3 - w \af_7)), \\
\alpha^*_2 &=& x y (y \af_2 - v \af_6), \\
\alpha^*_3 &=& x y (x y \af_3 - u y \af_5 - v x \af_7), \\
\alpha^*_4 &=& x^3 y^2 (x \af_4 + z \af_5 + u \af_7), \\
\alpha^*_5 &=& x^3 y^3 \af_5, \\
\alpha^*_6 &=& x^4 y^3 \af_6, \\
\alpha^*_7 &=& x^5 y^3 \af_7.
\end{array}$$
We are only interested  in cocycles with $(\af_4, \af_5, \af_6, \af_7) \neq (0,0,0,0)$, since otherwise we would obtain an algebra with an annihilator of dimension greater than one.

\begin{enumerate}
    \item $\af_7\neq 0.$ Then choosing 
    $u = -\frac 1{\af_7}(x \af_4 + z \af_5), $
    $w = \frac 1{\af_7}(x \af_1 + z \af_3), $
    $v = \frac y{x \af_7^2} (x \af_4 \af_5 + z \af_5^2 + x \af_3 \af_7),$ we have the representative $\la\as 2\nb 2+\af^*_5\nb 5+\af^*_6\nb 6+\af^*_7\nb 7\ra$, where
    \begin{align*}
    	\as 2 = \frac{y^2}{\af_7^2} (x \af_2\af_7^2 - \af_6 (x \af_4 \af_5 + z \af_5^2 + x \af_3 \af_7)).
    \end{align*}

    \begin{enumerate}
        \item $\af_5\neq 0$ and $\af_6 \neq 0$. Then choosing $x=\frac{\af_6}{\af_7}$ and $z=\frac 1{\af_5^2\af_7}(\af_2\af_7^2-\af_4\af_5\af_6 - \af_3\af_6\af_7)$ we have the family of representatives $\langle \af \nabla_5 +\nabla_6+\nabla_7 \rangle_{\af \ne 0}$ of distinct orbits.
        
        \item $\af_5= 0$, $\af_6=0$ and $\af_2=0$. Then we have the representative $\langle \nabla_7 \rangle.$

        \item $\af_5= 0$, $\af_6=0$ and $\af_2\neq 0$. Then we have the representative $\langle \nabla_2+ \nabla_7 \rangle.$

        \item $\af_5\neq 0$, $\af_6=0$ and $\af_2=0$. Then we have the representative $\langle \nabla_5+ \nabla_7 \rangle.$

        \item $\af_5\neq 0$, $\af_6=0$ and $\af_2\ne 0$. Then we have the representative $\langle \nabla_2+ \nabla_5+ \nabla_7 \rangle.$

        \item $\af_5=0$, $\af_6\neq 0$ and $\af_2\af_7-\af_3\af_6=0$. Then we have the representative $\langle   \nabla_6+\nabla_7 \rangle$ which will be joined with the family $\langle \af \nabla_5 +\nabla_6+\nabla_7 \rangle_{\af \ne 0}$. 
        \item $\af_5=0$, $\af_6\neq 0$ and $\af_2\af_7-\af_3\af_6\ne 0$. Then we have the representative $\langle  \nabla_2+\nabla_6 +\nabla_7 \rangle.$    
    \end{enumerate}
    
    \item $\af_7=0$, $\af_6\neq 0$ and $\af_5 \neq 0.$
    Then choosing 
    $v = \frac{y \af_2}{\af_6},$
    $z = -\frac{x \af_4}{\af_5},$
    $u = \frac{x \af_3}{\af_5},$  we have the representative $\la\as 1\nb 1+\af^*_5\nb 5+\af^*_6\nb 6\ra$, where
    \begin{align*}
    	\as 1 = \frac{x^3 y}{\af_5}( \af_1\af_5- \af_3 \af_4).
    \end{align*}
    
    \begin{enumerate}
    \item $\af_1\af_5-\af_3\af_4=0$. Then we have the representative $\langle \nabla_5+\nabla_6 \rangle.$
    \item $\af_1\af_5-\af_3\af_4\ne 0$. Then we have the representative $\langle \nabla_1+\nabla_5+\nabla_6 \rangle.$

    \end{enumerate}

     \item $\af_7=0$, $\af_6\neq 0$ and $\af_5 = 0.$
    Then choosing 
    $v =\frac{y \af_2}{\af_6},$  we have the representative $\la\as 1\nb 1+\as 3\nb 3+\as 4\nb 4+\af^*_6\nb 6\ra$, where
    \begin{align*}
    	\as 1 &= x^2 y (x \af_1+z \af_3-u\af_4),\\
    	\as 3 &= x^2 y^2 \af_3,\\
    	\as 4 &= x^4 y^2 \af_4.
    \end{align*}
    
    \begin{enumerate}
    \item $\af_3 \neq 0$ and $\af_4\neq 0$. Then choosing $x=\sqrt{\frac{\af_3}{\af_4}}$, $y=\frac{\af_4}{\af_6}$ and $u=\frac{x \af_1+z \af_3}{\af_4}$ we have the representative $\langle \nabla_3+\nabla_4+\nabla_6 \rangle.$
    \item $\af_3 \neq 0$ and $\af_4= 0$. Then we have the representative $\langle \nabla_3 +\nabla_6 \rangle.$
    \item $\af_3 = 0 $ and $\af_4\neq 0$. Then we have the representative $\langle \nabla_4+\nabla_6 \rangle.$
    \item $\af_3 = 0$ and $\af_4= 0$. Then we have two representatives $\la\nb 6\ra$ and $\langle \nabla_1+\nabla_6 \rangle$ depending on whether $\af_1=0$ or not.

    \end{enumerate}    
   
  \item $\af_7=0$, $\af_6=0$ and $\af_5 \neq 0.$ Then choosing 
  $z = -\frac{x \af_4}{\af_5},$
  $u = \frac{x \af_3}{\af_5}$
    we have the representative $\la\as 1\nb 1+\as 2\nb 2+\af^*_5\nb 5\ra$, where
    \begin{align*}
    	\as 1 &= \frac{x^3 y}{\af_5}( \af_1\af_5- \af_3 \af_4),\\
    	\as 2 &= x y^2 \af_2.
    \end{align*}
    
    \begin{enumerate}
        \item $\af_1\af_5 - \af_3 \af_4=0$ and $\af_2=0$. Then we have the representative $\langle \nabla_5 \rangle.$

        \item $\af_1\af_5 - \af_3 \af_4=0$ and $\af_2 \neq0$. Then we have the representative $\langle \nabla_2+ \nabla_5 \rangle.$
        \item $\af_1\af_5 -  \af_3 \af_4\ne 0$ and $\af_2 =0$. Then we have the representative $\langle \nabla_1+ \nabla_5 \rangle.$
        \item $\af_1\af_5 - \af_3 \af_4\ne 0$ and $\af_2 \neq 0$. Then we have the representative $\langle \nabla_1+\nabla_2+ \nabla_5 \rangle.$

    \end{enumerate}

  \item $\af_7=0$, $\af_6=0$, $\af_5 = 0$ and $\af_4\neq 0.$ Then choosing 
  $u= \frac{x \af_1 + z \af_3}{\af_4},$ we have the representative $\la\as 2\nb 2+\as 3\nb 3+\as 4\nb 4\ra$, where
  \begin{align*}
  	\as 2 &= x y^2  \af_2,\\
  	\as 3 &= x^2 y^2 \af_3,\\
  	\as 4 &= x^4 y^2  \af_4.
  \end{align*}
 \begin{enumerate}
     \item $\af_2=0$ and $\af_3=0$. Then we have the representative $\langle \nabla_4 \rangle.$
     \item $\af_2\neq 0$ nd $\af_3=0$. Then we have the representative $\langle \nabla_2+ \nabla_4 \rangle.$
     \item $\af_2= 0$ and $\af_3 \neq 0$. Then we have the representative $\langle \nabla_3+ \nabla_4 \rangle.$
     \item $\af_2\neq 0$ and $\af_3\neq 0$. Then we have the family of representatives $\langle \af\nabla_2+ \nabla_3+\nabla_4 \rangle_{\af\ne 0}$ of distinct orbits. We will adjoin $\langle \nabla_3+ \nabla_4 \rangle$ to this family.
 \end{enumerate}   

\end{enumerate}

Summarizing, we obtain the following representatives:
$\langle \nabla_1+\nabla_2+ \nabla_5 \rangle$,
$\langle \nabla_1+ \nabla_5 \rangle$,
$\langle \nabla_1+\nabla_5+\nabla_6 \rangle$,
$\langle \nabla_1+\nabla_6 \rangle$,
$\langle \af\nabla_2+ \nabla_3+\nabla_4 \rangle$,
$\langle \nabla_2+ \nabla_4 \rangle$,
$\langle \nabla_2+ \nabla_5 \rangle$,
$\langle \nabla_2+ \nabla_5+ \nabla_7 \rangle$,
$\langle  \nabla_2+\nabla_6 +\nabla_7 \rangle$,
$\langle \nabla_2+ \nabla_7 \rangle$,
$\langle \nabla_3+\nabla_4+\nabla_6 \rangle$,
$\langle \nabla_3 +\nabla_6 \rangle$,
$\langle \nabla_4 \rangle$,
$\langle \nabla_4+\nabla_6 \rangle$,
$\langle \nabla_5 \rangle$,
$\langle \nabla_5+\nabla_6 \rangle$,
$\langle \af \nabla_5 +\nabla_6+\nabla_7 \rangle$,
$\langle \nabla_5+ \nabla_7 \rangle$,
$\la\nb 6\ra$,
$\langle \nabla_7 \rangle$.
All of them belong to distinct orbits.
The corresponding algebras are:

\begin{longtable}{llllllllll} 
${\mathbb A}_{66}$ & : &  
$e_1e_2=e_3$, & $e_1e_3=e_4$,& $e_1e_4=e_6$, & $e_2e_3=e_6$, & $e_2e_5=e_6$, & $e_3e_4=e_5$;\\
${\mathbb A}_{67}$ & : &  
$e_1e_2=e_3$, & $e_1e_3=e_4$,& $e_1e_4=e_6$, & $e_2e_5=e_6$, & $e_3e_4=e_5$;\\
${\mathbb A}_{68}$ & : &  
$e_1e_2=e_3$, & $e_1e_3=e_4$,& $e_1e_4=e_6$, & $e_2e_5=e_6$, & $e_3e_4=e_5$, & $e_3e_5=e_6$;\\
${\mathbb A}_{69}$ & : &  
$e_1e_2=e_3$, & $e_1e_3=e_4$,& $e_1e_4=e_6$, & $e_3e_4=e_5$, & $e_3e_5=e_6$;\\
${\mathbb A}_{70}(\af)$ & : &  
$e_1e_2=e_3$, & $e_1e_3=e_4$, & $e_1e_5=e_6$, & $e_2e_3=\af e_6$, & $e_2e_4=e_6$, & $e_3e_4=e_5$;\\
${\mathbb A}_{71}$ & : &  
$e_1e_2=e_3$, & $e_1e_3=e_4$, & $e_1e_5=e_6$, & $e_2e_3= e_6$, & $e_3e_4=e_5$;\\
${\mathbb A}_{72}$ & : &  
$e_1e_2=e_3$, & $e_1e_3=e_4$, & $e_2e_3= e_6$, & $e_2e_5=e_6$, & $e_3e_4=e_5$;\\
${\mathbb A}_{73}$ & : &  
$e_1e_2=e_3$, & $e_1e_3=e_4$, & $e_2e_3= e_6$, & $e_2e_5=e_6$, & $e_3e_4=e_5$, & $e_4e_5=e_6$;\\
${\mathbb A}_{74}$ & : &  
$e_1e_2=e_3$, & $e_1e_3=e_4$, & $e_2e_3= e_6$, & $e_3e_4=e_5$, & $e_3e_5=e_6$, & $e_4e_5=e_6$;\\
${\mathbb A}_{75}$ & : &  
$e_1e_2=e_3$, & $e_1e_3=e_4$, & $e_2e_3= e_6$, & $e_3e_4=e_5$, & $e_4e_5=e_6$;\\
${\mathbb A}_{76}$ & : &  
$e_1e_2=e_3$, & $e_1e_3=e_4$, & $e_1e_5=e_6$, & $e_2e_4= e_6$, & $e_3e_4=e_5$, & $e_3e_5=e_6$;\\
${\mathbb A}_{77}$ & : &  
$e_1e_2=e_3$, & $e_1e_3=e_4$, & $e_2e_4= e_6$, & $e_3e_4=e_5$, & $e_3e_5=e_6$;\\
${\mathbb A}_{78}$ & : &  
$e_1e_2=e_3$, & $e_1e_3=e_4$, & $e_1e_5=e_6$, & $e_3e_4=e_5$;\\
${\mathbb A}_{79}$ & : &  
$e_1e_2=e_3$, & $e_1e_3=e_4$, & $e_1e_5=e_6$, & $e_3e_4=e_5$, & $e_3e_5=e_6$;\\
${\mathbb A}_{80}$ & : &  
$e_1e_2=e_3$, & $e_1e_3=e_4$, & $e_2e_5=e_6$, & $e_3e_4=e_5$;\\
${\mathbb A}_{81}$ & : &  
$e_1e_2=e_3$, & $e_1e_3=e_4$, & $e_2e_5=e_6$, & $e_3e_4=e_5$, & $e_3e_5=e_6$;\\
${\mathbb A}_{82}(\af)$ & : &  
$e_1e_2=e_3$, & $e_1e_3=e_4$, & $e_2e_5=\af e_6$, & $e_3e_4=e_5$, & $e_3e_5=e_6$, & $e_4e_5=e_6$;\\
${\mathbb A}_{83}$ & : &  
$e_1e_2=e_3$, & $e_1e_3=e_4$, & $e_2e_5=e_6$, & $e_3e_4=e_5$, & $e_4e_5=e_6$;\\
${\mathbb A}_{84}$ & : &  
$e_1e_2=e_3$, & $e_1e_3=e_4$, & $e_3e_4=e_5$, & $e_3e_5=e_6$;\\
${\mathbb A}_{85}$ & : &  
$e_1e_2=e_3$, & $e_1e_3=e_4$, & $e_3e_4=e_5$, & $e_4e_5=e_6$.
\end{longtable}

All of the above, combined with the algebraic classification of $6$-dimensional nilpotent Tortkara and Malcev algebras in \cite{gkk} and \cite{hac18}, respectively, yields our main result of this section.

\begin{theorem}
The distinct isomorphism classes of $6$-dimensional complex nilpotent anticommutative algebras are given explicitly in Appendix~\ref{appb}. The list is comprised of $14$ one-parameter families and $130$ additional isomorphism classes.
\end{theorem}

\section{The geometric classification of $6$-dimensional nilpotent anticommutative algebras}\label{S:geo}

\subsection{Degenerations of algebras}
Given an $n$-dimensional vector space ${\bf V}$, the set ${\rm Hom}({\bf V} \otimes {\bf V},{\bf V}) \cong {\bf V}^* \otimes {\bf V}^* \otimes {\bf V}$ 
is a vector space of dimension $n^3$. This space inherits the structure of the affine variety $\mathbb{C}^{n^3}.$ 
Indeed, let us fix a basis $e_1,\dots,e_n$ of ${\bf V}$. Then any $\mu\in {\rm Hom}({\bf V} \otimes {\bf V},{\bf V})$ is determined by $n^3$ structure constants $c_{i,j}^k\in\mathbb{C}$ such that
$\mu(e_i\otimes e_j)=\sum_{k=1}^nc_{i,j}^ke_k$. A subset of ${\rm Hom}({\bf V} \otimes {\bf V},{\bf V})$ is {\it Zariski-closed} if it can be defined by a set of polynomial equations in the variables $c_{i,j}^k$ ($1\le i,j,k\le n$).

The general linear group ${\rm GL}({\bf V})$ acts by conjugation on the variety ${\rm Hom}({\bf V} \otimes {\bf V},{\bf V})$ of all algebra structures on ${\bf V}$:
$$ (g * \mu )(x\otimes y) = g\mu(g^{-1}x\otimes g^{-1}y),$$ 
for $x,y\in {\bf V}$, $\mu\in {\rm Hom}({\bf V} \otimes {\bf V},{\bf V})$ and $g\in {\rm GL}({\bf V})$. Clearly, the ${\rm GL}({\bf V})$-orbits correspond to the isomorphism classes of algebras structures on ${\bf V}$. Let $T$ be a set of polynomial identities which is invariant under isomorphism. Then the subset $\mathbb{L}(T)\subset {\rm Hom}({\bf V} \otimes {\bf V},{\bf V})$ of the algebra structures on ${\bf V}$ which satisfy the identities in $T$ is ${\rm GL}({\bf V})$-invariant and Zariski-closed. It follows that $\mathbb{L}(T)$ decomposes into ${\rm GL}({\bf V})$-orbits. The ${\rm GL}({\bf V})$-orbit of $\mu\in\mathbb{L}(T)$ is denoted by $O(\mu)$ and its Zariski closure by $\overline{O(\mu)}$.

Let ${\bf A}$ and ${\bf B}$ be two $n$-dimensional algebras satisfying the identities from $T$ and $\mu,\lambda \in \mathbb{L}(T)$ represent ${\bf A}$ and ${\bf B}$ respectively.
We say that ${\bf A}$ {\it degenerates} to ${\bf B}$ and write ${\bf A}\to {\bf B}$ if $\lambda\in\overline{O(\mu)}$.
Note that in this case we have $\overline{O(\lambda)}\subset\overline{O(\mu)}$. Hence, the definition of a degeneration does not depend on the choice of $\mu$ and $\lambda$. It is easy to see that any algebra degenerates to the algebra with zero multiplication. If ${\bf A}\not\cong {\bf B}$, then the assertion ${\bf A}\to {\bf B}$ 
is called a {\it proper degeneration}. We write ${\bf A}\not\to {\bf B}$ if $\lambda\not\in\overline{O(\mu)}$. 

Let ${\bf A}$ be represented by $\mu\in\mathbb{L}(T)$. Then  ${\bf A}$ is  {\it rigid} in $\mathbb{L}(T)$ if $O(\mu)$ is an open subset of $\mathbb{L}(T)$.
Recall that a subset of a variety is called {\it irreducible} if it cannot be represented as a union of two non-trivial closed subsets. A maximal irreducible closed subset of a variety is called an {\it irreducible component}.
It is well known that any affine variety can be represented as a finite union of its irreducible components in a unique way.
The algebra ${\bf A}$ is rigid in $\mathbb{L}(T)$ if and only if $\overline{O(\mu)}$ is an irreducible component of $\mathbb{L}(T)$.

In the present work we use the methods applied to Lie algebras in \cite{BC99,GRH,GRH2,S90}.
To prove 
degenerations, we will construct families of matrices parametrized by $t$. Namely, let ${\bf A}$ and ${\bf B}$ be two algebras represented by the structures $\mu$ and $\lambda$ from $\mathbb{L}(T)$, respectively. Let $e_1,\dots, e_n$ be a basis of ${\bf V}$ and $c_{i,j}^k$ ($1\le i,j,k\le n$) be the structure constants of $\lambda$ in this basis. If there exist $a_i^j(t)\in\mathbb{C}$ ($1\le i,j\le n$, $t\in\mathbb{C}^*$) such that the elements $E_i^t=\sum_{j=1}^na_i^j(t)e_j$ ($1\le i\le n$) form a basis of ${\bf V}$ for any $t\in\mathbb{C}^*$, and the structure constants $c_{i,j}^k(t)$ of $\mu$ in the basis $E_1^t,\dots, E_n^t$ satisfy $\lim\limits_{t\to 0}c_{i,j}^k(t)=c_{i,j}^k$, then ${\bf A}\to {\bf B}$. In this case  $E_1^t,\dots, E_n^t$ is called a {\it parametric basis} for ${\bf A}\to {\bf B}$.




When the number of orbits under the action of ${\rm GL}({\bf V})$ on  $\mathbb{L}(T)$ is finite, then the graph of primary degenerations gives the whole picture. In particular, the description of rigid algebras and irreducible components can be easily obtained. Since the variety of $6$-dimensional nilpotent anticommutative algebras contains infinitely many non-isomorphic algebras, we have to fulfill some additional work. Let $A(*):=\{A(\alpha)\}_{\alpha\in I}$ be a family of algebras and $B$ be another algebra. Suppose that, for $\alpha\in I$, $A(\alpha)$ is represented by the structure $\mu(\alpha)\in\mathbb{L}(T)$ and $B$ is represented by the structure $\lambda\in\mathbb{L}(T)$. Then $A(*)\to B$ means $\lambda\in\overline{\cup\{O(\mu(\alpha))\}_{\alpha\in I}}$, and $A(*)\not\to B$ means $\lambda\not\in\overline{\cup\{O(\mu(\alpha))\}_{\alpha\in I}}$.

Let ${\bf A}(*)$, ${\bf B}$, $\mu(\alpha)$ ($\alpha\in I$) and $\lambda$ be as above. To prove ${\bf A}(*)\to {\bf B}$ it is enough to construct a family of pairs $(f(t), g(t))$ parametrized by $t\in\mathbb{C}^*$, where $f(t)\in I$ and $g(t)=\left(a_i^j(t)\right)_{i,j}\in {\rm GL}({\bf V})$. Namely, let $e_1,\dots, e_n$ be a basis of ${\bf V}$ and $c_{i,j}^k$ ($1\le i,j,k\le n$) be the structure constants of $\lambda$ in this basis. If we construct $a_i^j:\mathbb{C}^*\to \mathbb{C}$ ($1\le i,j\le n$) and $f: \mathbb{C}^* \to I$ such that $E_i^t=\sum_{j=1}^na_i^j(t)e_j$ ($1\le i\le n$) form a basis of ${\bf V}$ for any  $t\in\mathbb{C}^*$, and the structure constants $c_{i,j}^k(t)$ of $\mu\big(f(t)\big)$ in the basis $E_1^t,\dots, E_n^t$ satisfy $\lim\limits_{t\to 0}c_{i,j}^k(t)=c_{i,j}^k$, then ${\bf A}(*)\to {\bf B}$. In this case, $E_1^t,\dots, E_n^t$ and $f(t)$ are called a {\it parametric basis} and a {\it parametric index} for ${\bf A}(*)\to {\bf B}$, respectively. In the construction of degenerations of this sort, we will write $\mu\big(f(t)\big)\to \lambda$, emphasizing that we are proving the assertion $\mu(*)\to\lambda$ using the parametric index $f(t)$.


Through a series of degenerations summarized in Appendix~\ref{app:A} by the corresponding parametric bases and indices, we obtain the main result of this section (compare Theorem A).

\begin{theorem}\label{geotort}
The variety of $6$-dimensional complex nilpotent anticommutative algebras is irreducible of dimension $34$.
It is defined by the family of algebras ${\mathbb A}_{82}\big(\af\big)$, with $\alpha\in\mathbb{C}^*$. 
\end{theorem}

\begin{proof}[{\bf Proof}]
Thanks to \cite{gkk19} the algebras ${\mathbb T}_{10},$ ${\mathbb T}_{17}$ and ${\mathbb T}_{19}$ (see Appendix B) are rigid in the variety of $6$-dimensional nilpotent Tortkara algebras. 
These algebras and the remaining $6$-dimensional nilpotent anticommutative algebras degenerate from ${\mathbb A}_{82}\big(\af\big),$  as shown in the table below. Clearly, we need only consider algebras in the family with $\af\neq 0$. 
Since $\mathfrak{Der} ({\mathbb A}_{82}\big(\af\big) )=3$, for all $\af\neq 0$,
it follows that the dimension of the irreducible component is $\dim \mathrm{GL}_6(\mathbb{C})-\dim\mathfrak{Der} ({\mathbb A}_{82}\big(\af\neq 0\big) )+1=36-3+1=34$.
\end{proof}

\appendix
\section{Degenerations of $6$-dimensional nilpotent anticommutative algebras}\label{app:A}

{\tiny 

\begin{longtable}{|lcl|ll|}

%
%
%
%
\hline

${\mathbb A}_{82}\big(-t^{-2}\big)$ &$\to$& ${\mathbb T}_{10}$ & 

$E^t_1= -e_1 + e_2 + -e_3 + ( 1 + 1/t^2) e_4$&
$E^t_2= t e_2 - t  e_3 + (1/t + t)  e_4$ \\ &&&
$E^t_3= -t e_3 + t  e_4$ & 
$E^t_4= -e_4 + (1 - 1/t^2)  e_5 - 1/t e_6$ \\ &&&
$E^t_5= e_5$&
$E^t_6= -e_6$\\
\hline

${\mathbb A}_{82}\left(\frac{2-t}{t^3}\right)$ &$\to$& ${\mathbb T}_{17}$ &

$E^t_1= t^{-1} e_1 + ( 1 + t^{-2}) e_2 + \frac{t-2}{t^2} e_3 + \frac{t-2 - 2 t^2 + 3 t^3 - t^4}{t^5} e_4$&
$E^t_2= (t - 2) e_2 + \frac{t-2}{t} e_3 + \frac{2 - 3 t + 3 t^2 - t^3}{t^3} e_4 - \frac{(t-2)^2}{t^7} e_5$ \\ &&&
$E^t_3= (t-2)/t e_3 - \frac{(t-2) (t-1)}{t^2} e_4 - \frac{(t-2)^2}{t^6} e_5$ &
$E^t_4= \frac{t-2}{t^2} e_4 - \frac{(t-2)^2}{t^6} e_5$  \\ &&&
$E^t_5= \frac{(t-2)^2}{t^4} e_5$ &
$E^t_6= \frac{(t-2)^3}{t^7} e_6$\\
\hline

${\mathbb A}_{82}\left(\frac{t ( 2 t-1)}{(1 - t + t^2)^2}\right)$ &$\to$& ${\mathbb T}_{19}$ & 

\multicolumn{2}{l|}{$E^t_1= -\frac{t}{ 1 - t + t^2} e_1 + \frac{  t (2 - 5 t + 7 t^2 - 5 t^3 + 2 t^4)}{ (1 - 2 t)^2} e_2 + \frac{(t-1) t}{  (-1 + 2 t) (1 - t + t^2)}   e_4 + \frac{(t-1) t}{2 t-1} e_5$} \\ &&&
\multicolumn{2}{l|}{$E^t_2= (1 - t + t^2)^2 e_2 + (t-1 - t^2) e_3 +  e_4 +  \frac{t}{1 - 2 t} e_5$} \\ &&& 
\multicolumn{2}{l|}{$E^t_3= t (t-1 - t^2) e_3 + ( t + t^3)  e_4 + \frac{  t (t-1 - t^2 + t^3)}{-1 + 2 t} e_5 + \frac{  t^2 (-1 + t^2 - 2 t^3 + t^4)}{2 t-1} e_6$} \\ &&&
\multicolumn{2}{l|}{$E^t_4= t^2 e_4 + \frac{(t-1) t^2}{2 t-1} e_5 + \frac{  t^2 (-1 + 2 t - 2 t^2 + t^3)}{2 t-1} e_6$}\\ &&&
$E^t_5= -\frac{t^2}{1 - t + t^2} e_5 + \frac{t^3 (1 - t + t^2)}{2 t-1} e_6$ &
$E^t_6= -t^3 (1 - t + t^2) e_6$\\

\hline

${\mathbb A}_{82}\big(\alpha\big)$   &$\to$& ${\mathbb A}_{00}$ & 

$E^t_1= e_1$ &
$E^t_2= e_2$ \\ &&&
$E^t_3= e_3$ &
$E^t_4= e_4$ \\ &&&
$E^t_5= e_5$ &
$E^t_6= 1/t e_6$\\

\hline

${\mathbb A}_{82}\big(t^2\big)$    &$\to$& ${\mathbb A}_{01}$ &

$E^t_1= \frac{t}{t^2-1} e_1$ &
$E^t_2= ( 1 - t^2) e_2 - e_3 + e_4 - \frac{1}{t^2-1} e_5$ \\ &&&
$E^t_3= -t e_3 + t  e_4$ &
$E^t_4= \frac{t^2}{1 - t^2} e_4$ \\ &&&
$E^t_5= \frac{t^2}{t^2-1} e_5 - \frac{t^2}{(t^2-1)^2} e_6$ & 
$E^t_6= \frac{t^3}{ (t^2-1)^2} e_6$\\

\hline

${\mathbb A}_{82}\big(1/t^3\big)$ &$\to$& ${\mathbb A}_{02}$ & 

$E^t_1= t^{-1} e_1 + \frac{(t-1)^2}{t^3} e_2 - t^{-4} e_3 + \frac{-1 + 2 t}{t^6} e_4 + t^{-9} e_5$&
$E^t_2= -t^{-2} e_2 +  t^{-3} e_3 - \frac{t-1 + t^2}{t^5} e_4 + \frac{t-1}{t^9} e_5$ \\ &&&
$E^t_3= -t^{-3} e_3 + \frac{1 + t}{t^4} e_4 + t^{-13} e_6$ &
$E^t_4= -t^{-4} e_4 + \frac{t-1 - t^2}{t^9} e_5 + \frac{t-1}{t^{13}} e_6$ \\ &&&
$E^t_5= t^{-8} e_5 + \frac{1 - t}{t^{13}} e_6$ &
$E^t_6= -t^{-13} e_6$ \\
\hline

${\mathbb A}_{82}\left(\frac{1}{(t-1) t}\right)$ &$\to$& ${\mathbb A}_{03}$ & 

$E^t_1= \frac{1}{t-1} e_1 - \frac{1}{(t-1)^3 t} e_2 +  \frac{1}{(t-1)^4 t^2} e_4$&
$E^t_2= e_2 +  \frac{1}{t-1} e_3 + \frac{1 + t - t^2}{(t-1)^2 t}   e_4 + \frac{-1 - t + t^2}{(t-1)^5 t^2} e_5$ \\ &&&
$E^t_3= \frac{1}{t-1}  e_3 + \frac{1}{1 - t} e_4 + \frac{1}{(t-1)^4 t^2} e_5$&
$E^t_4= \frac{1}{(t-1)^2} e_4 - \frac{1}{(t-1)^5 t^2} e_5$ \\ &&&
$E^t_5= -\frac{1}{(t-1)^3 t} e_5 + \frac{1}{(t-1)^6 t^2} e_6$&
$E^t_6= \frac{1}{(t-1)^5 t^2} e_6$\\

\hline

${\mathbb A}_{82}\big(2\big)$    &$\to$& ${\mathbb A}_{04}$ & 

$E^t_1= -2 e_1 +  e_2 + \frac{2 t}{4 + t} e_3 - \frac{8 + 4t}{4 + t} e_4 + \frac{16 t}{(4 + t)^2} e_5$&
$E^t_2= -\frac{t^2}{4 + t} e_2 + \frac{2 t^2}{4 + t} e_3 - \frac{4 t^2}{4 + t}   e_4 + \frac{32 t^2}{(4 + t)^2} e_5$\\ &&&
$E^t_3= \frac{2 t^2}{4 + t} e_3 - \frac{4 t^2}{4 + t} e_4 + \frac{16 t^2}{(4 + t)^2} e_5 + \frac{64 t^3}{(4 + t)^3} e_6$&
$E^t_4= -\frac{4 t^2}{4 + t} e_4 + \frac{16 t^2}{(4 + t)^2} e_5 + \frac{  32 t^3}{(4 + t)^3} e_6$ \\ &&&
$E^t_5= -\frac{8 t^3}{(4 + t)^2} e_5 + \frac{64 t^3}{(4 + t)^3} e_6$&
$E^t_6= -\frac{32 t^4}{(4 + t)^3} e_6$\\ 
\hline

${\mathbb A}_{82}\big( t^{-2}\big)$    &$\to$& ${\mathbb A}_{05}$ & 
$E^t_1= 1/t e_1$ &
$E^t_2= t^3 e_2$ \\ &&&
$E^t_3= t^2 e_3$ &
$E^t_4= t e_4$ \\ &&&
$E^t_5= t^2 e_5 - t^2 e_6$ &
$E^t_6= t^3 e_6$\\ 

\hline

${\mathbb A}_{82}\big(t\big)$    &$\to$& ${\mathbb A}_{06}$ & 

$E^t_1= 1/t e_1 + (1 + t^3) e_2 + t e_4 - t^3 (2 + t^3) e_5$&
$E^t_2= t^2 (2 + t^3) e_2 - t^3 (2 + t^3) e_4$ \\ &&&
$E^t_3= ( 2 t + t^4) e_3 - (2 t + t^4 ) e_4 + t^2 (2 t + t^4) e_5 -  t^5 (2 + t^3)^2 e_6$&
$E^t_4= (2 + t^3) e_4 - t^2 (2 + t^3) e_5 + t^4 (2 + t^3)^2 e_6$ \\ &&&
$E^t_5= (2 + t^3)^2 e_5 - (2 (1 + t^3) (2 + t^3)^2)/t e_6$ &
$E^t_6= (2 + t^3)^3 e_6$\\
\hline

${\mathbb A}_{82}\left(\frac{1+t}{t}\right)$    &$\to$& ${\mathbb A}_{07}$ & 

$E^t_1= e_1 - \frac{1}{(1 + t)^2} e_2 + \frac{1}{t (1 + t)} e_4$ &
$E^t_2= -e_2 +  e_4 + \frac{1}{t^2 (1 + t)} e_5$ \\ &&&
$E^t_3= -e_3$ &
$E^t_4= -e_4 + \frac{1}{t (1 + t)} e_5$ \\ &&&
$E^t_5= t^{-1} e_5 + \frac{1}{t^3 + t^4} e_6$&
$E^t_6=  -t^{-2} e_6$\\
\hline
 
${\mathbb A}_{82}\big(-t-1\big)$    &$\to$& ${\mathbb A}_{08}$ & 
$E^t_1= e_1 + e_2 + e_4 + t e_5$&
$E^t_2= -t^2 e_2 + t^2 e_3 - t^2 (2 + t) e_4 - t^3 e_5$ \\ &&&
$E^t_3= -t^2 e_3 + t^2 e_4 - t^2 e_5 + t^4 e_6$&
$E^t_4= -t^2 e_4 + t^2 e_5 + t^3 e_6$ \\ &&&
$E^t_5= t^3 e_5$&
$E^t_6= -t^4 e_6$\\

\hline
 
${\mathbb A}_{82}\left( \frac{t^2 (1 + t)}{(t^2-1)^2} \right)$    &$\to$& ${\mathbb A}_{09}$ &

$E^t_1= \frac{t^2}{t^2-1} e_1 + \frac{(t-1)^2 (1 + t - t^2)}{t^2} e_2 - \frac{1 + t - t^2}{1 + t} e_4$ &
$E^t_2= -\frac{(t^2-1)^2}{t} e_2 + t e_4 + (t - t^2) e_5$ \\ &&&
$E^t_3= (t^2-1) e_3 +   e_4 + (1 - t) e_5 + \frac{(t-1)^2 (1 + t)}{t^2} e_6$ &
$E^t_4= t e_4 + (t - t^2)  e_5 + \frac{( t-1)^2 (1 + t)}{t}  e_6$ \\ &&&
$E^t_5= (t^2 - 1) e_5 - (t-1)^2 (1 + t) e_6$ &
$E^t_6= (t^3-t) e_6$ \\

\hline

${\mathbb A}_{82}\big( t^3\big)$ &$\to$& ${\mathbb A}_{10}$ & 

$E^t_1= t^2 e_1 + e_2 - t^3 e_4$  &
$E^t_2= t^2 e_2$ \\ &&&
$E^t_3= t^3 e_3 - t^3 e_4$ &
$E^t_4= t^4 e_4$ \\ &&&
$E^t_5= t^6 e_5$ &
$E^t_6= t^{10} e_6$\\
\hline

${\mathbb A}_{82}\left(\frac{2 t^2-1}{t}\right)$ &$\to$& ${\mathbb A}_{11}$ & 

$E^t_1= - e_1 + \frac{t^2}{2 t^2-1} e_2$ &
$E^t_2= e_2 - e_3 + (1 + t^{-1} -t)  e_4$ \\ &&&
$E^t_3= -e_3 + (1 - t) e_4$ &
$E^t_4= e_4$ \\ &&&
$E^t_5= t^{-1} e_5$ &
$E^t_6= t^{-1} e_6$ \\
\hline

${\mathbb A}_{82}\big(-1/t^2\big)$ &$\to$& ${\mathbb A}_{12}$ &

$E^t_1= t^{-1} e_1 + \frac{1 + t - t^3 - t^8}{t^5} e_2 + \frac{1 + t - t^3 - t^8}{t^7} e_4$&
$E^t_2= t^{-1} e_2 +  t^{-2} e_3 + \frac{1 + t - t^2}{t^4} e_4 + \frac{-1 - t + t^2 + t^3 + t^8}{t^{10}} e_5$ \\ &&&
$E^t_3= t^{-2} e_3 + \frac{1 - t}{t^3} e_4 + ( -1 - t^{-9} + t^{-7} + t^{-6} + t^{-1}) e_5 + \frac{-1 + t^8}{ t^{11}} e_6$&
$E^t_4= t^{-3} e_4 + \frac{-1 - t + t^3 + t^8}{t^9} e_5$ \\&&&
$E^t_5= -t^{-6} e_5 + t^{-12} - t^{-4} e_6$&
$E^t_6= -t^{-10} e_6$\\
\hline

${\mathbb A}_{82}\big(-1/t\big)$ &$\to$& ${\mathbb A}_{13}$ &

$E^t_1= e_1 + ( t - t^2) e_2 +  e_4 + t (-1 + t^2) e_5$&
$E^t_2= t^2 e_2 - t^2 e_3 + ( t + t^2) e_4$ \\ &&&
$E^t_3= t^2 e_3 - t^2  e_4 + t^2 e_5$ &
$E^t_4= t^2 e_4 + (-t^2 + t^4) e_5 + t^3 e_6$ \\ &&&
$E^t_5= -t^3 e_5$ &
$E^t_6= -t^4 e_6$\\

\hline

${\mathbb A}_{82}\big( t^2 \big)$ &$\to$& ${\mathbb A}_{14}$ & 
$E^t_1= t^2 e_1 - t^{-3}e_2  + t^{-1}e_3 - t^{-2}e_5$ &
$E^t_2= t^{-2}e_2  - e_4 - t^{-1}e_5$ \\ &&&
$E^t_3= e_3 - t^{-1}e_5$ &
$E^t_4= t e_4 + t^{-3}e_6$  \\ &&&
$E^t_5= e_5 + t^{-1}e_6$ &
$E^t_6= e_6$ \\
\hline

${\mathbb A}_{82}\big( -t^2 \big)$ &$\to$& ${\mathbb A}_{15}$ & 
$E^t_1= t^2 e_1 + (1 - t^2) e_2 + t^2 e_3 - t^4 e_4$ &
$E^t_2= t e_2 + t^3 e_4$  \\ &&&
$E^t_3= t^3 e_3 + t^5 e_5$ &
$E^t_4= -t^4 e_4 - t^6 e_5$  \\ &&&
$E^t_5= -t^6 e_5$ &
$E^t_6= -t^9 e_6$\\

\hline

${\mathbb A}_{82}\big(1\big)$ &$\to$& ${\mathbb A}_{16}$ & 

$E^t_1= e_1 + (t - 1) e_2 + e_4 - t  e_5$&
$E^t_2= t^3 e_2 + t^4 e_3 - (t^3 + t^4) e_4 + t^4 e_5$ \\ &&& 
$E^t_3= t^3 e_3 -  t^3 e_4 + t^3 e_5 + (-t^4 + t^6 + t^7) e_6$&
$E^t_4= t^2 e_4 - t^2 e_5 - t^3 (-1 + t^2) e_6$ \\ &&&
$E^t_5= t^5 e_5 - t^5 e_6$ &
$E^t_6=  t^6 e_6$ \\
\hline

${\mathbb A}_{82}\big(1/t\big)$ &$\to$& ${\mathbb A}_{17}$ & 

$E^t_1= e_1 + e_4 -  e_5$&
$E^t_2= t e_2 + t e_4 - (1 + t) e_5$ \\ &&& 
$E^t_3= t e_3$ &
$E^t_4= e_4 - e_5$  \\ &&&
$E^t_5= t e_5 - t e_6$&
$E^t_6= t e_6$ \\

\hline

${\mathbb A}_{82}\left(\frac{1}{t^2-1} \right)$ &$\to$& ${\mathbb A}_{18}$ &

$E^t_1= e_1 + e_4 + (-1 + t^2)  e_5$ &
$E^t_2= t^2 e_2$ \\ &&& 
$E^t_3= t^2 e_3 - t^2 e_6$ &
$E^t_4= t e_4 + (-t + t^3)  e_5 + (t - t^2 - 2 t^3 + t^5) e_6$ \\ &&&
$E^t_5= t^3 e_5 + (-t^3 + t^5)  e_6$ &
$E^t_6= t^3 e_6$\\

\hline

${\mathbb A}_{82}\left(\frac{1+t^2}{1+t}\right)$ &$\to$& ${\mathbb A}_{19}$ & 
$E^t_1= -\frac{1 + t^2}{1 + t} e_1 + e_2 + \frac{(1 - t) t}{1 + t} e_4 +  t e_5$&
$E^t_2= \frac{t^2 (1 + t)}{1 + t^2} e_2 - t^2 e_4$  \\ &&&
$E^t_3= -t^2 e_3 + t^2 e_4 +  t^3 e_5 + \frac{2 t^3 (1 + t)}{1 + t^2} e_6$ &
$E^t_4= t e_4 + t^2 e_5 + \frac{2 t^2 (1 + t)}{1 + t^2} e_6$ \\ &&&
$E^t_5= -t^3 e_5$ & 
$E^t_6= -t^3 e_6$\\

\hline

${\mathbb A}_{82}\big(1 - t\big)$ &$\to$& ${\mathbb A}_{20}$ &

$E^t_1= e_1 + e_2 + t  e_4$ &
$E^t_2= t/2 e_2 + \frac{1}{2} (t-1) t  e_4$\\ &&&
$E^t_3= t/2 e_3 + t/2 e_4 - t/2 e_6$&
$E^t_4= e_4  - e_6$ \\ &&&
$E^t_5= t/2 e_5$ &
$E^t_6= t/2 e_6$\\

\hline

${\mathbb A}_{82}\big(\alpha\big)$ &$\to$& ${\mathbb A}_{21}$ &

$E^t_1= 1/t e_1$ &
$E^t_2= -t^3 e_2 + t^3 e_3$ \\ &&&
$E^t_3= -t^2 e_3 + t^3 e_4$ & 
$E^t_4= e_4 + t^2 e_5 - t  e_6$ \\ &&&
$E^t_5= -t^2 e_5$ &
$E^t_6= -t^2 e_6$ \\

\hline

${\mathbb A}_{82}\big(t^3 \big)$ &$\to$& ${\mathbb A}_{22}$ &

$E^t_1= (t^2 + t^3) e_1 + (1/t^4 + t + t^2) e_2 - (t^4 + t^5) e_4 + \frac{-1 + t^2 + t^5}{t^4} e_5$ & 
$E^t_2= \frac{1 - t}{t^4} e_2 + e_4 + t e_5$ \\ &&&
$E^t_3= ( -1 + 1/t^2) e_3 +  e_4 + e_5 + \frac{1 + t + t^2 - t^3 - t^4}{t^6 + t^7} e_6$ & 
$E^t_4= 1/t e_4 + t  e_5 + \frac{1}{t^7 (1 + t)} e_6$ \\ &&& 
$E^t_5= (1/t^3 - 1/t) e_5$ &
$E^t_6= (1/t^5 - 1/t^3) e_6$ \\

\hline

${\mathbb A}_{82}\big( t^{-3/2} \big)$ &$\to$& ${\mathbb A}_{23}$ & 
$E^t_1= - \frac{1}{t} e_1 - e_2 - t  e_5$ &
$E^t_2= -t^2 \sqrt{t} e_2 + t^2 \sqrt{t} e_3 -  \sqrt{t} (t^2-1) e_4 +  e_5$ \\ &&&
$E^t_3= t^{3/2} e_3$ &
$E^t_4= -t^{3/2} e_4$ \\ &&& 
$E^t_5= -t e_5$ &
$E^t_6= 1/\sqrt{t}  e_6$\\ 
\hline

${\mathbb A}_{82}\big( t^{3} \big)$ &$\to$& ${\mathbb A}_{24}$ & 

$E^t_1= \frac{t^2}{1 + t^3} e_1 - \frac{ 1 + t + 3 t^3 + 3 t^6 + t^9}{(t + t^4)^4} e_2 +  \frac{1}{(1 + t^3)^4} e_4 + \frac{1}{t^3 (1 + t^3)^4} e_5$ &
$E^t_2= (-1 - 1/t^3) e_2 +  e_3 +  t^3 e_4 - \frac{t (2 + t^3)}{(1 + t^3)^3} e_5$ \\ &&&
$E^t_3= -\frac{1}{t}  e_3 + \frac{t^2}{1 + t^3} e_4 -  \frac{1}{(1 + t^3)^4} e_5 + \frac{2 + t^3}{(1 + t^3)^4} e_6$ &
$E^t_4= -\frac{1}{1 + t^3} e_4 + \frac{1}{t^2 (1 + t^3)^4} e_5 + \frac{ 1}{t^5 (1 + t^3)^4}  e_6$ \\ &&&
$E^t_5= \frac{1}{t + t^4} e_5 - \frac{1}{t^3 (1 + t^3)^4} e_6$ & 
$E^t_6= -\frac{1}{(t + t^4)^2} e_6$\\
\hline

${\mathbb A}_{82}\big( t^{-1} \big)$ &$\to$& ${\mathbb A}_{25}$ & 
$E^t_1= e_1 + e_2$ &
$E^t_2= t e_2$ \\ &&& 
$E^t_3= t e_3 - t  e_4$  & 
$E^t_4= e_4$ \\ &&& 
$E^t_5= t e_5$ &
$E^t_6= e_6$ \\
\hline

${\mathbb A}_{82}\big( t^{3} \big)$ &$\to$& ${\mathbb A}_{26}$ & 
$E^t_1= t^2 e_1 + e_2$  & 
$E^t_2= t e_2$ \\ &&&
$E^t_3= t^3  e_3$ & 
$E^t_4= t^4 e_4$ \\ &&& 
$E^t_5= t^7 e_5$ &
$E^t_6= t^{10} e_6$ \\

\hline

${\mathbb A}_{82}\big(1 \big)$ &$\to$& ${\mathbb A}_{27}$ & 

$E^t_1= e_1 + e_2$ &
$E^t_2= t e_2 - t  e_4$ \\ &&&
$E^t_3= t e_3$ &
$E^t_4= e_4$ \\ &&& 
$E^t_5= t e_5$ &
$E^t_6= t  e_6$ \\

\hline

${\mathbb A}_{82}\big(t^3 \big)$ &$\to$& ${\mathbb A}_{28}$ & 

$E^t_1= t^2 e_1$ &
$E^t_2= e_2$ \\ &&& 
$E^t_3= t^2 e_3$ &
$E^t_4= t^3 e_4$ \\ &&& 
$E^t_5= t^5 e_5$ &
$E^t_6= t^7 e_6$ \\

\hline

${\mathbb A}_{82}\big( \af \big)$ &$\to$& ${\mathbb A}_{29}$ & 

$E^t_1= e_1$ &
$E^t_2= t e_2$ \\ &&& 
$E^t_3= t e_3$ &
$E^t_4= e_4$ \\ &&& 
$E^t_5= t  e_5$ &
$E^t_6= t e_6$ \\

\hline

${\mathbb A}_{82}\big( 2/t^2 \big)$ &$\to$& ${\mathbb A}_{30}$ & 

$E^t_1= 1/t e_1 + (\frac{1}{2} - \frac{1}{4t} -t/2) e_2 -  e_3 + (1 + \frac{1}{2 t^3} - 1/t^2)  e_4 + \frac{1 - t}{2 t^3} e_5$ &
$E^t_2= -t^2 e_2 +  e_4 + \frac{1 - 2 t}{2 t^2} e_5$ \\ &&&
$E^t_3= -t e_3 + t  e_4 - \frac{1}{2 t}  e_5 + \frac{1 - t}{2 t^2} e_6$ &
$E^t_4= -e_4 + \frac{1 - 2 t}{2 t^2} e_5 + \frac{1}{2 t^2} e_6$\\ &&&
$E^t_5= e_5 + \frac{1}{2 t^2} e_6$&
$E^t_6= -1/t e_6$\\ 

\hline

${\mathbb A}_{82}\big( -1 + 1/t \big)$ &$\to$& ${\mathbb A}_{31}$ & 

$E^t_1= -e_1 +  e_2 - \frac{(t-1)^2}{t (1 + t)} e_3 + \frac{2 ( t-1)}{1 + t}  e_4 + \frac{(t-1)^3}{t^3 (1 + t)} e_5$  &
$E^t_2= \frac{(t-1)^2}{1 + t} e_2 - \frac{(t-1)^2}{1 + t} e_3 + \frac{  2 (t-1)^2}{1 + t} e_4 + \frac{2 ( t-1)^3}{t^2 (1 + t)} e_5$ \\ &&& 
$E^t_3= -\frac{(t-1)^2}{1 + t} e_3 + \frac{(t-1)^2}{1 + t} e_4 +   \frac{2 (t-1)^3}{t (1 + t)^2} e_5 + \frac{(1 - t)^5}{t^4 (1 + t)^2} e_6$ &
$E^t_4= \frac{(t-1)^2}{1 + t} e_4 + \frac{(t-1)^3}{t (1 + t)} e_5$ \\ &&& 
$E^t_5= -\frac{( t-1)^4}{t (1 + t)^2} e_5 + \frac{(1 - t)^5}{t^3 (1 + t)^2}   e_6$ &
$E^t_6= -\frac{(t-1)^6}{t^2 (1 + t)^3}e_6$ \\

\hline

${\mathbb A}_{82}\left(\frac{\af}{t^2 (1-t^2)} \right)$ &$\to$& ${\mathbb A}_{32}(\af)$ & 

$E^t_1= 1/t e_1 + e_2 -  e_3 + \frac{\af - t^2 + t^4}{t^2 (t^2-1)}  e_4 + (1/t - t)  e_5$ &
$E^t_2= -t^2 e_2 - \frac{-1 + \af + t^2}{ t^2-1} e_4 - \frac{\af - t^2 + t^4}{t - t^3} e_5$ \\ &&& 
$E^t_3= -t e_3 + \frac{-1 + \af + t^2}{t^2-1}  e_5 + (1/t - t)  e_6$ &
$E^t_4= - e_4 - \frac{\af - t^2 + t^4}{t - t^3} e_5 + (1 - t^2)  e_6$ \\ &&&
$E^t_5= e_5 + ( 1/t - t) e_6$ & 
$E^t_6= -e_6$\\
\hline

${\mathbb A}_{82}\big( -t^{-4} \big)$ &$\to$& ${\mathbb A}_{33}$ & 

$E^t_1= 1/t e_1 + e_2 - e_3 + (1 + 1/t^4)  e_4 - e_5$ &
$E^t_2= -t^2 e_2 - 1/t^2 e_4$ \\ &&&
$E^t_3= -e_3 + 1/t^2 e_5$ &
$E^t_4= -e_4 + (1/t^3 + t) e_5$ \\ &&&
$E^t_5= e_5 - t  e_6$ &
$E^t_6= - e_6$\\

\hline

${\mathbb A}_{82}\big(1/t^3 \big)$ &$\to$& ${\mathbb A}_{34}$ & 

$E^t_1= 1/t e_1 + t e_3 - t  e_4 - e_5$ &
$E^t_2= t^3 e_2 - t^3 e_3 + t^3 e_4 + t (1 + t) e_5$ \\ &&&
$E^t_3= t^2 e_3 - t^2  e_4 +  e_6$ &
$E^t_4= t e_4$ \\ &&& 
$E^t_5= t^2 e_5 + t e_6$&
$E^t_6= -t^2 e_6$\\

\hline

${\mathbb A}_{82}\big( t^2 \big)$ &$\to$& ${\mathbb A}_{35}(\af)$ & 
 
\multicolumn{2}{l|}{$E^t_1= \frac{t}{1 - t} e_1 + \left(-\frac{  \af  t}{(1 - t)^2} + ( 1-t) t (1 + t^2 - t^3)\right) e_2 - (t-1)^2 t^3 e_3 + \frac{ t^3 (\af + (1 - t)^3 t (-1 - t + t^2)}{(1 - t)^2} e_4 +  \af (t-1) t^5 e_5$}\\ && &
\multicolumn{2}{l|}{$E^t_2= (t-1 )^3 t^3 e_2 - (t-1 )^2 t^4 e_3 - t^6 (2 - 3 t + t^2) e_4 -  t^8 ((t-1)^4 - \af (3 - 3 t + t^2)) e_5$} \\ &&& 
\multicolumn{2}{l|}{$E^t_3= -(t-1)^2 t^4 e_3 + (t-1) t^5 e_4 -  t^7 (-\af + (t-1)^4 t) e_5 - \af (t-1)^2 t^9 e_6$} \\ && &
\multicolumn{2}{l|}{$E^t_4= (t-1) t^5 e_4 - t^7 (-\af + (t-1)^4 t^2) e_5 -  \af (t-1)^2 t^9 e_6$} \\ &&&
$E^t_5= -(t-1)^3 t^8 e_5 - \af (t-1)^2 t^{10} e_6$ &
$E^t_6= -(t-1)^4 t^{12} e_6$ \\

\hline

${\mathbb A}_{82}\big(-t \big)$ &$\to$& ${\mathbb A}_{36}$ & 
\multicolumn{2}{l|}{$E^t_1= \frac{t}{ 1 - t} e_1 + (-1 + 2 t - 5 t^2 + 6 t^3 - 4 t^4 + t^5) e_2 - (1 - t)^4 t e_3 +  t^2 (-2 + t + 2 t^2 - 3 t^3 + t^4) e_4 + (1 - t)^5 t^2 e_5$} \\&& &
\multicolumn{2}{l|}{$E^t_2= (t-1)^5 t e_2 - (t-1)^4 t^2 e_3 + (t-1)^4 t^3 e_4 + (t-2) (t-1)^5 t^4 e_5$} \\ &&&
\multicolumn{2}{l|}{$E^t_3=-(t-1 )^4 t^2 e_3 + (t-1)^3 t^3 e_4 + (t-1)^4 t^4 (-3 + 5 t - 4 t^2 + t^3) e_5$} \\ &&&
\multicolumn{2}{l|}{$E^t_4= (t-1 )^3 t^3 e_4 + (t-1)^4 t^4 (-1 - 2 t + 5 t^2 - 4 t^3 + t^4) e_5 + (t-1)^8 t^4 e_6$} \\ &&&
$E^t_5= (1 - t)^7 t^4 e_5 + (t-1)^8 t^5 e_6$&
$E^t_6= -(t-1 )^{10} t^6 e_6$\\

\hline

${\mathbb A}_{82}\big( \af /t^3\big)$ &$\to$& ${\mathbb A}_{37}(\af)$ & 

$E^t_1= 1/t e_1 + t e_3 +  e_5$ &
$E^t_2= t^3 e_2$ \\ &&&
$E^t_3= t^2 e_3 - \af e_6$ &
$E^t_4= t e_4 - t^2 e_6$ \\ &&&
$E^t_5= t^2 e_5 - t e_6$ &
$E^t_6= t^2 e_6$\\

\hline 
${\mathbb A}_{82}\big( -t \big)$ &$\to$& ${\mathbb A}_{38}$ &

$E^t_1= t e_1 + (1 + t) e_2 + t e_3 + t^2 e_4$ &
$E^t_2= t e_2 + t^2 e_3 + t^3  e_4$ \\ &&&
$E^t_3= t^2 e_3 + t^3  e_4$ &
$E^t_4= t^3 e_4$\\ &&&
$E^t_5= t^4 e_5$ &
$E^t_6= t^6 (1 + t) e_6$\\

\hline 
${\mathbb A}_{82}\big( \af  \big)$ &$\to$& ${\mathbb A}_{39}$ & 

$E^t_1= 1/t e_1 - e_3 + e_4$ &
$E^t_2= -t^2 e_2 + t^2 e_3 + (-1 + \af) t^2 e_4$ \\ &&&
$E^t_3= -t e_3 + t e_4$ &
$E^t_4= -e_4$   \\ &&&
$E^t_5= e_5$ &
$E^t_6= -e_6$ \\

\hline 
${\mathbb A}_{82}\big( -2  \big)$ &$\to$& ${\mathbb A}_{40}$ &

$E^t_1= e_1 + \frac{1}{4}  (1 + 6 t^2) e_2 + t^2 e_3 + (\frac{1}{2} + t^2)  e_4 +  t^2 e_5$ &
$E^t_2= t^3 e_2 + t^3 e_3 + (t-1) t^2 e_4 +  \frac{t^2}{2}  (1 + t) e_5$ \\ &&&
$E^t_3= t^3 e_3 + t^3 e_4 - \frac{1}{2} t^3 (1 + 2 t) e_5 +  \frac{1}{2} (1 - t) t^4 e_6$ &
$E^t_4= t^3 e_4 - \frac{t^3}{2} e_5 + (-\frac{3 t^5}{2} + t^6) e_6$ \\ &&&
$E^t_5= t^5 e_5 - \frac{t^5}{2} e_6$ &
$E^t_6= -t^7 e_6$\\

\hline 
${\mathbb A}_{82}\big( -1/t \big)$ &$\to$& ${\mathbb A}_{41}$ & 

$E^t_1= e_1 - 2 (-1 + t) t^2 e_2 - t^2 e_3 + 2 t e_4 - 2 t^3 e_5$ &
$E^t_2= t^3 e_2 + t^2 e_4$ \\ &&&
$E^t_3= t^3 e_3 - t^4 e_5$ &
$E^t_4= t^3 e_4 - 2 t^4 e_5 + t^6 e_6$ \\ &&&
$E^t_5= -t^5 e_5$ &
$E^t_6= -t^7 e_6$\\

\hline 
${\mathbb A}_{82}\left( -\frac{2 t}{(t-1)^2} \right)$ &$\to$& ${\mathbb A}_{42}$ & 

\multicolumn{2}{l|}{$E^t_1= \frac{t}{1-t}  e_1 -  \frac 14 (t-1 )^2 (3 t-2 ) e_2 + (t-1)^2 t e_3 -  \frac{1}{2} t^2 (2 t-1 ) e_4 - \frac{1}{2} ( t-1 )^3 t^2 e_5$}\\ && & 
\multicolumn{2}{l|}{$E^t_2= (t-1)^3 t e_2 - (t-1 )^2 t^2 e_3 +  t^2 (t^2-1 ) e_4 - \frac{1}{2} (t-1 )^3 t^3 e_5$} \\ &&& 
$E^t_3= -(t-1)^2 t^2 e_3 + (t-1 ) t^3 e_4 +  \frac{1}{2} (t-2 ) (t-1 )^2 t^3 e_5$ &
$E^t_4= (t-1 ) t^3 e_4 - \frac{1}{2} (t-1 )^2 t^4 e_5 +  \frac{1}{2} (t-1)^4 t^4 e_6$ \\ &&&
$E^t_5= (t-1 )^3 t^4 e_5 + \frac{1}{2} (t-1 )^4 t^5 e_6$&
$E^t_6= (t-1)^4 t^6 e_6$\\

\hline

${\mathbb A}_{82}\big(2\big)$     &$\to$& ${\mathbb A}_{43}(\af)$ & 

$E^t_1= te_1 +(\frac 1{4\af} - \frac 12)e_2 + e_3 - \frac 1{2\af}e_4 + \frac 1 {2\af t}e_5$ &
$E^t_2= te_2 - te_3 - e_5$ \\ &&&
$E^t_3= t^2e_3 - t^2e_4 - \frac t{2\af}e_5 - \frac 1{2\af}e_6$ &
$E^t_4= t^3e_4 -t^2(1 - \frac 1{2\af})e_5$ \\ &&&
$E^t_5= t^3e_5 - \frac{t^2}{2\af}e_6$ &
$E^t_6= \frac{t^4}{\af}e_6$\\

\hline

${\mathbb A}_{82} \big( -t \big)$     &$\to$& ${\mathbb A}_{44}$ & 

$E^t_1= t e_1 + (t - 1) e_2 - t e_3 + t^2 e_4$  & 
$E^t_2= t e_2 + t^2 e_4$ \\ &&&
$E^t_3= t^2 e_3 - t^3 e_5$ & 
$E^t_4= t^3 e_4 - t^4 e_5$ \\ &&& 
$E^t_5= -t^4 e_5$ &
$E^t_6= -t^6 e_6$ \\

\hline

${\mathbb A}_{82}\big(\af\big)$     &$\to$& ${\mathbb A}_{45}$ & 

$E^t_1= te_1 - \af t^{-1} e_2 + \af t^{-1} e_3 + (\af - 1)\af t^{-1} e_4$ &
$E^t_2= e_2 - e_3 + (1-\af)e_4$ \\ &&&
$E^t_3= te_3 - te_4$ &
$E^t_4= t^2e_4 - \af^2e_5$ \\ &&&
$E^t_5= \af t e_5$ &
$E^t_6= \af t^3 e_6$\\

\hline

${\mathbb A}_{82}\big(t^2\big)$     &$\to$& ${\mathbb A}_{46}$ & 

$E^t_1= te_1 - 2t^{-1} e_2 + te_4$ &
$E^t_2= e_2 - e_3 + (1-t^2)e_4 - e_5$ \\ &&&
$E^t_3= te_3 - te_4 + te_5 + te_6$ &
$E^t_4= t^2e_4 - t^2e_5 - t^2e_6$ \\ &&&
$E^t_5= -t^2e_5 + t^2e_6$ &
$E^t_6= t^3e_6$\\

\hline

${\mathbb A}_{82}\big(4t-2\big)$     &$\to$& ${\mathbb A}_{47}(\af)$ & 

$E^t_1= 2te_1 + \frac{4t(t - 1)}{(2t - 1)^2}e_2 + \frac{4t}{2t - 1}e_4 -
\frac{4t(2\af t - \af + 2)}{2t - 1}e_5$ &
$E^t_2= e_2 + e_3 + e_4 -\frac{4t}{2t - 1}e_5$ \\ &&&
$E^t_3= 2te_3 + 2te_4 -\frac{4t}{2t - 1}e_5 + \frac{16t^2(2\af t - \af + 1)}{2t - 1}e_6$ &
$E^t_4= 4t^2e_4 - \frac{8t^2}{2t - 1}e_5 + \frac{16t^2(2\af t - \af + 1)}{2t - 1}e_6$ \\ &&&
$E^t_5= 4t^2e_5 - \frac{16t^3}{2t - 1}e_6$ &
$E^t_6= 16t^3e_6$\\

\hline

${\mathbb A}_{82}\big(-2\big)$     &$\to$& ${\mathbb A}_{48}(\af)$ & 

$E^t_1= te_1 - 2te_2 - 2te_4 -2(\af - 2)te_5$ &
$E^t_2= e_2 + e_3 + e_4$ \\ &&&
$E^t_3= te_3 + te_4 + 2te_5$ &
$E^t_4= t^2e_4 + 2t^2e_5 + 4(\af - 1)t^2e_6$ \\ &&&
$E^t_5= t^2e_5$ &
$E^t_6= 2t^3e_6$\\

\hline

${\mathbb A}_{82}\big(t^2\big)$     &$\to$& ${\mathbb A}_{49}$ & 

$E^t_1= te_1 - t^{-4}e_2 - t^{-2}e_5$ &
$E^t_2= e_2 - e_3 + (1-t^2)e_4$ \\ &&&
$E^t_3= te_3 - te_4$ &
$E^t_4= t^2e_4$ \\ &&&
$E^t_5= -t^2e_5$ &
$E^t_6= e_6$\\

\hline

${\mathbb A}_{82}\big(\af\big)$     &$\to$& ${\mathbb A}_{50}$ & 

$E^t_1= te_1 + e_2 + \af e_5$ &
$E^t_2= e_2 - e_3 + e_4$ \\ &&&
$E^t_3= te_3 - te_4 - \af^2 e_6$ &
$E^t_4= t^2e_4$ \\ &&&
$E^t_5= -t^2e_5$ &
$E^t_6= -\af t^2e_6$\\

\hline

${\mathbb A}_{82}\big(\af\big)$     &$\to$& ${\mathbb A}_{51}$ & 

$E^t_1= te_1 + \af e_5$ &
$E^t_2= e_2 - e_3 + e_4$ \\ &&&
$E^t_3= te_3 - te_4 - \af^2 e_6$ &
$E^t_4= t^2e_4$ \\ &&&
$E^t_5= -t^2e_5$ &
$E^t_6= -\af t^2e_6$\\

\hline

${\mathbb A}_{82}\left(2t-2\right)$     &$\to$& ${\mathbb A}_{52}$ & 

$E^t_1= te_1 - 2te_5$ &
$E^t_2= e_2 + e_3 + e_4$ \\ &&&
$E^t_3= te_3 + te_4 + 4t^2e_6$ &
$E^t_4= t^2e_4 + 4t^2e_6$ \\ &&&
$E^t_5= t^2e_5$ &
$E^t_6= 2t^3e_6$\\

\hline

${\mathbb A}_{82}\left(t^2 - 2\right)$     &$\to$& ${\mathbb A}_{53}$ & 

$E^t_1= te_1 - 2te_5$ &
$E^t_2= e_2 + e_3 + e_4$ \\ &&&
$E^t_3= te_3 + te_4 + 2t^3e_6$ &
$E^t_4= t^2e_4 + 4t^2e_6$ \\ &&&
$E^t_5= t^2e_5$ &
$E^t_6= 2t^3e_6$\\

\hline

${\mathbb A}_{82}\left(t^2\right)$     &$\to$& ${\mathbb A}_{54}$ & 

$E^t_1= te_1 + e_2 - \frac{t^2}2e_4 + \frac {t^3}4e_5$ &
$E^t_2= te_2 - te_3 + (t-t^3)e_4 + \frac{t^2}2e_5$ \\ &&&
$E^t_3= t^2e_3 - t^2e_4 - \frac{t^3}2e_5 + \frac{t^4}4e_6$ &
$E^t_4= t^3e_4 + \frac{t^4}2e_5 - \frac{t^5}4e_6$ \\ &&&
$E^t_5= -t^4e_5 -\frac{t^5}2e_6$ &
$E^t_6= -\frac{t^6}2e_6$\\

\hline

${\mathbb A}_{82}\left(t^2\right)$     &$\to$& ${\mathbb A}_{55}$ & 

$E^t_1= te_1 + 2e_2 - t^2e_4 + t^3e_5$ &
$E^t_2= e_2 - e_3 + (1-t^2)e_4 + te_5$ \\ &&&
$E^t_3= te_3 - te_4 - t^2e_5 + t^3e_6$ &
$E^t_4= t^2e_4 + t^3e_5 - t^4e_6$ \\ &&&
$E^t_5= -t^2e_5 - t^3e_6$ &
$E^t_6= -t^4e_6$\\

\hline

${\mathbb A}_{82}\left(t^3\right)$     &$\to$& ${\mathbb A}_{56}$ & 

$E^t_1= te_1 + t^{-2}e_2 - te_4$ &
$E^t_2= e_2 - e_3 + e_4 + (t^3 + 1)e_5$ \\ &&&
$E^t_3= te_3 - te_4 - te_5$ &
$E^t_4= t^2e_4 + t^2e_5$ \\ &&&
$E^t_5= -t^2e_5 - t^2e_6$ &
$E^t_6= -t^4e_6$\\

\hline

${\mathbb A}_{82}\left(t^2\right)$     &$\to$& ${\mathbb A}_{57}(\af)$ & 

$E^t_1= te_1 + e_2 + \af t^2e_4 -\af(\af + 1)t^3 e_5$ &
$E^t_2= te_2 - te_3 -t(t^2 - 1)e_4$ \\ &&&
$E^t_3= t^2e_3 - t^2e_4 + \af t^3 e_5$ &
$E^t_4= t^3 e_4 - \af t^4 e_5 + \af(\af + 1)t^5 e_6$ \\ &&&
$E^t_5= -t^4e_5$ &
$E^t_6= -t^6e_6$\\

\hline

${\mathbb A}_{82}\left(1\right)$     &$\to$& ${\mathbb A}_{58}(\af)$ & 

$E^t_1= te_1 + (\af  + 1)e_2$ &
$E^t_2= e_2 - e_3 + e_4 + \af t^{-1}e_5$ \\ &&&
$E^t_3= te_3 - te_4 + \af(\af  + 1)t^{-1}e_6$ &
$E^t_4= t^2e_4$ \\ &&&
$E^t_5= -t^2e_5 - \af te_6$ &
$E^t_6= -t^2e_6$\\

\hline

${\mathbb A}_{82}\left(\af\right)$     &$\to$& ${\mathbb A}_{59}$ & 

$E^t_1= te_1 + e_2$ &
$E^t_2= e_2 - e_3 + e_4 + \af t^{-1} e_5$ \\ &&&
$E^t_3= te_3 - te_4 + \af ^2 t^{-1}e_6$ &
$E^t_4= t^2e_4$ \\ &&&
$E^t_5= -t^2e_5 - \af te_6$ &
$E^t_6= -\af t^2e_6$\\

\hline

${\mathbb A}_{82}\left(t^3\right)$     &$\to$& ${\mathbb A}_{60}$ & 

$E^t_1= te_1 + t^{-1}e_2$ &
$E^t_2= e_2 - e_3 + e_4$ \\ &&&
$E^t_3= te_3 - te_4$ &
$E^t_4= t^2e_4$ \\ &&&
$E^t_5= -t^2e_5$ &
$E^t_6= -t^4e_6$\\

\hline

${\mathbb A}_{82}\left(\af\right)$     &$\to$& ${\mathbb A}_{61}$ & 

$E^t_1= te_1$ &
$E^t_2= e_2 - e_3 + e_4$ \\ &&&
$E^t_3= te_3 - te_4$ &
$E^t_4= t^2e_4$ \\ &&&
$E^t_5= -t^2e_5$ &
$E^t_6= -\af t^2e_6$\\

\hline

${\mathbb A}_{82}\left(2t-2\right)$     &$\to$& ${\mathbb A}_{62}$ & 

$E^t_1= te_1$ &
$E^t_2= e_2 + e_3 + e_4$ \\ &&&
$E^t_3= te_3 + te_4$ &
$E^t_4= t^2e_4$ \\ &&&
$E^t_5= t^2e_5$ &
$E^t_6= 2t^3e_6$\\

\hline

${\mathbb A}_{82}\left(2t^2-2\right)$     &$\to$& ${\mathbb A}_{63}$ & 

$E^t_1= te_1$ &
$E^t_2= e_2 + e_3 + e_4$ \\ &&&
$E^t_3= te_3 + te_4$ &
$E^t_4= t^2e_4$ \\ &&&
$E^t_5= t^2e_5$ &
$E^t_6= 2t^3e_6$\\

\hline

${\mathbb A}_{82}\left(t-t^2\right)$     &$\to$& ${\mathbb A}_{64}$ & 

$E^t_1= te_1$ &
$E^t_2= e_2 + (t - 1)e_3 + (t - 1)^2e_4$ \\ &&&
$E^t_3= te_3 + t(t - 1)e_4$ &
$E^t_4= t^2e_4$ \\ &&&
$E^t_5= t^2(t - 1)e_5$ &
$E^t_6= t^4(t - 1)e_6$\\

\hline

${\mathbb A}_{82}\left(t^3\right)$     &$\to$& ${\mathbb A}_{65}$ & 

$E^t_1= te_1$ &
$E^t_2= e_2 - e_3 + e_4$ \\ &&&
$E^t_3= te_3 - te_4$ &
$E^t_4= t^2e_4$ \\ &&&
$E^t_5= -t^2e_5$ &
$E^t_6= -t^4e_6$\\

\hline

${\mathbb A}_{82}\left(t^{-3}\right)$     &$\to$& ${\mathbb A}_{66}$ & 

$E^t_1= t^{-1}e_1 - e_2 + t^{-3}e_4$ &
$E^t_2= e_2 - e_3 + e_4 - (t^{-2} + t^{-5})e_5$ \\ &&&
$E^t_3= t^{-1}e_3 - t^{-1}e_4 + t^{-3}e_5$ &
$E^t_4= t^{-2}e_4 - t^{-4}e_5$ \\ &&&
$E^t_5= t^{-3}e_5 - t^{-5}e_6$ &
$E^t_6= t^{-6}e_6$\\

\hline

${\mathbb A}_{82}\left(t^{-3}\right)$     &$\to$& ${\mathbb A}_{67}$ & 

$E^t_1= t^{-1}e_1 - t^{-4}e_5$ &
$E^t_2= e_2 - e_3 + e_4$ \\ &&&
$E^t_3= t^{-1}e_3 - t^{-1}e_4 + t^{-7}e_6$ &
$E^t_4= t^{-2}e_4$ \\ &&&
$E^t_5= t^{-3}e_5$  &
$E^t_6= t^{-6}e_6$\\

\hline

${\mathbb A}_{82}\left(t\right)$     &$\to$& ${\mathbb A}_{68}$ & 

$E^t_1= te_1 - t^2e_5$ &
$E^t_2= e_2$ \\ &&&
$E^t_3= te_3 + t^3e_6$ &
$E^t_4= t^2e_4 + t^3e_6$ \\ &&&
$E^t_5= t^3e_5$  &
$E^t_6= t^4e_6$\\

\hline

${\mathbb A}_{82}\left(t^3\right)$     &$\to$& ${\mathbb A}_{69}$ & 

$E^t_1= te_1 - t^2e_5$ &
$E^t_2= e_2 - t^3e_4$ \\ &&&
$E^t_3= te_3$ &
$E^t_4= t^2e_4 + t^3e_6$ \\ &&&
$E^t_5= t^3e_5$  &
$E^t_6= t^4e_6$\\

\hline
 
 ${\mathbb A}_{82}\left(1\right)$     &$\to$& ${\mathbb A}_{70}(\af)$ & 

$E^t_1= t^{-1}e_1 + (t^3 - \af)t^{-4}e_2 + \af t^{-4}e_4 -\af t^{-4}e_5$ &
$E^t_2= t^2e_2 - t^2e_3 + t^2e_4 -(2\af + t)t^{-1}e_5$ \\ &&&
$E^t_3= te_3 - te_4 + \af t^{-2}e_5 -(\af + t)t^{-2}e_6$ &
$E^t_4= e_4 - \af t^{-3}e_5 + \af t^{-3}e_6$ \\ &&&
$E^t_5= te_5 - \af t^{-2}e_6$  &
$E^t_6= e_6$\\

\hline

 ${\mathbb A}_{82}\left(-\frac 1{t^3(t - 1)}\right)$     &$\to$& ${\mathbb A}_{71}$ & 

$E^t_1= t^{-1}e_1 + e_2 + t^{-4}e_4 + (t - 1)^{-1}t^{-7}e_5$ &
$E^t_2= e_2 - e_3 + e_4 - (t^4 - t^3 - 1)(t - 1)^{-1}t^{-6}e_5$ \\ &&&
$E^t_3= t^{-1}e_3 - t^{-1}e_4 + t^{-4}e_5 + (t - 1)^{-1}t^{-7}e_6$ &
$E^t_4= t^{-2}e_4 -t^{-5}e_5 - (t - 1)^{-1}t^{-8}e_6$ \\ &&&
$E^t_5= t^{-3}e_5 - t^{-6}e_6$  &
$E^t_6= -(t - 1)^{-1}t^{-7}e_6$\\

\hline

 ${\mathbb A}_{82}\left(t^{-3}\right)$     &$\to$& ${\mathbb A}_{72}$ & 

$E^t_1= t^{-1}e_1 - e_2 + t^{-3}e_4$ &
$E^t_2= e_2 - e_3 + e_4 - (t^{-2} + t^{-5})e_5$ \\ &&&
$E^t_3= t^{-1}e_3 - t^{-1}e_4 + t^{-3}e_5$ &
$E^t_4= t^{-2}e_4 - t^{-4}e_5$ \\ &&&
$E^t_5= t^{-3}e_5 - t^{-5}e_6$  &
$E^t_6= t^{-6}e_6$\\

\hline

${\mathbb A}_{82}\big( 1/t^2 \big)$     &$\to$& ${\mathbb A}_{73}$ & 

$E^t_1= -1/t e_1 + t e_2 - t^2 e_3 + \frac{t^3-1}{t} e_4 -  t^2 e_5$ & 
$E^t_2= t^2 e_2 - t^2 e_3 + t^2 e_4 - (1 + t^2) e_5$ \\ &&&
$E^t_3= -t e_3 + t e_4 - t  e_5 + t^2  e_6$ &
$E^t_4= e_4 - e_5$ \\ &&&
$E^t_5= -t e_5 + t  e_6$ &
$E^t_6= -t e_6$\\

\hline

${\mathbb A}_{82}\big( -t \big)$     &$\to$& ${\mathbb A}_{74}$ & 

$E^t_1= e_1 + (-1/t + t) e_2 + t^2 e_3 - e_4 - t^2  e_5$ & 
$E^t_2= t e_2 + t^2  e_3 - t^2  e_5$   \\ &&&
$E^t_3= t e_3 - t^4 e_6$ &
$E^t_4= t e_4 + t  e_5 + t^3  e_6$ \\ &&&
$E^t_5= t^2 e_5 + t^2 e_6$ &
$E^t_6= t^3 e_6$ \\

\hline

${\mathbb A}_{82}\big( 1/t \big)$     &$\to$& ${\mathbb A}_{75}$ &

$E^t_1= 1/t e_1 + e_2 - 1/t e_4$ &
$E^t_2= -t^3 e_2 - t^2 e_5$ \\ &&&
$E^t_3= -t^2 e_3$ &
$E^t_4= -t e_4 - t  e_5$ \\ &&&
$E^t_5= t^3 e_5 + t^3  e_6$ &
$E^t_6= -t^4 e_6$ \\

\hline

${\mathbb A}_{82}\left(-\frac{t^2}{(t^2-1)^2} \right)$     &$\to$& ${\mathbb A}_{76}$ &

$E^t_1= \frac{t}{1-t^2} e_1 + (t - t^3) e_2 + \frac{t}{ 1 - t^2} e_4 -  t e_5$ &
$E^t_2= (-1 + t^2)^2 e_2 + (t^2 - t^4) e_3 + t^4 e_4$ \\ &&& 
$E^t_3= (t - t^3) e_3 + t^3 e_4 - t^3 e_5$ &
$E^t_4= t^2 e_4 - t^2 e_5 + (t^2 - t^4) e_6$ \\ &&&
$E^t_5= (t^3 - t^5) e_5 + t^3 (-1 + t^2) e_6$ &
$E^t_6= (t^4 - t^6) e_6$\\

\hline

${\mathbb A}_{82}\big( t^3 \big)$     &$\to$& ${\mathbb A}_{77}$ & 
$E^t_1= -t e_1 - e_2 + t^3  e_4$ &
$E^t_2= t^3 e_2 - t^6  e_4 - t^8 e_5$ \\ &&& 
$E^t_3= -t^4 e_3$ &
$E^t_4= t^5 e_4 + t^7 e_5$ \\ &&& 
$E^t_5= -t^9 e_5 - t^{11} e_6$ &
$E^t_6= t^13 e_6$\\
\hline

${\mathbb A}_{82}\big( 1 \big)$     &$\to$& ${\mathbb A}_{78}$ & 

$E^t_1= 1/t e_1 + e_2$ &
$E^t_2= t^3 e_2 - t^3 e_3$ \\ &&& 
$E^t_3= t^2 e_3 -  t^2 e_4$ &
$E^t_4= t e_4$ \\ &&&
$E^t_5= t^3 e_5$ &
$E^t_6= t^3 e_6$\\

\hline

${\mathbb A}_{82}\big( t^3 \big)$     &$\to$& ${\mathbb A}_{79}$ & 

$E^t_1= t e_1 + e_2$ &
$E^t_2= t^2 e_2 - t^5  e_4$ \\ &&&
$E^t_3= t^3 e_3$ &
$E^t_4= t^4 e_4$ \\ &&&
$E^t_5= t^7 e_5$ &
$E^t_6=  t^{10} e_6$ \\
\hline

${\mathbb A}_{82}\big( 1 \big)$     &$\to$& ${\mathbb A}_{80}$ & 
$E^t_1= t e_1$ &
$E^t_2=  e_2$ \\ &&&
$E^t_3= t e_3$ &
$E^t_4= t^2 e_4$ \\ &&&
$E^t_5= t^3 e_5$ &
$E^t_6= t^3 e_6$\\
\hline

${\mathbb A}_{82}\big(t \big)$     &$\to$& ${\mathbb A}_{81}$ & 
$E^t_1= t e_1$ &
$E^t_2= e_2$\\ &&&
$E^t_3= t e_3$ &
$E^t_4= t^2 e_4$ \\ &&& 
$E^t_5= t^3 e_5$&
$E^t_6= t^4 e_6$\\

\hline

${\mathbb A}_{82}\big(t^{-2} \big)$     &$\to$& ${\mathbb A}_{83}$ & 

$E^t_1= 1/t e_1$ &
$E^t_2= t^2 e_2$ \\ &&&
$E^t_3= t e_3 - t  e_4$ &
$E^t_4= e_4$ \\ &&&
$E^t_5= t e_5$ &
$E^t_6= t e_6$\\

\hline

${\mathbb A}_{82}\big(t^3 \big)$     &$\to$& ${\mathbb A}_{84}$ & 

$E^t_1= t e_1$ &
$E^t_2= e_2 - t^3 e_4$ \\ &&&
$E^t_3= t e_3$ & 
$E^t_4= t^2 e_4$ \\ &&& 
$E^t_5= t^3 e_5$ &
$E^t_6= t^4 e_6$\\

\hline

${\mathbb A}_{82}\big(\alpha\big)$     &$\to$& ${\mathbb A}_{85}$ &

$E^t_1= 1/t e_1$&
$E^t_2= t e_2 - t  e_3 + ( 1 + t - \af t) e_4$ \\ &&&
$E^t_3= e_3$&
$E^t_4= 1/t e_4$ \\&&&
$E^t_5= 1/t e_5$ &
$E^t_6= t^{-2} e_6$\\

\hline

\end{longtable} }

\section{The list of $6$-dimensional nilpotent anticommutative algebras}\label{appb}

{\scriptsize
\begin{longtable}{lllllllllll}



${\mathbb M}_{00}$ 
&:& ---   \\

${\mathbb M}_{01}$ 
&:& $e_1e_2=e_3$;    \\

${\mathbb M}_{02}$ 
&:& $e_1e_2=e_3$, & $e_1e_3=e_4$;      \\

${\mathbb M}_{03}$ 
&:& $e_1e_2=e_5$,& $e_3e_4=e_5$;  \\

${\mathbb M}_{04}$ 
&:& $e_1e_2=e_4$, &$e_1e_3=e_5$; \\

${\mathbb M}_{05}$ 
&:& $e_1e_2=e_3$, &$e_1e_4=e_5$,& $e_2e_3=e_5$;      \\

${\mathbb M}_{06}$ 
&:&  $e_1e_2=e_3$, & $e_1e_3=e_4$, & $e_2e_3=e_5$;     \\

${\mathbb M}_{07}$ 
&:& $e_1e_2=e_3$, &$e_1e_3=e_4$,  & $e_1e_4=e_5$;   \\

${\mathbb M}_{08}$ 
&:& $e_1e_2=e_3$, &$e_1e_3=e_4$, & $e_1e_4=e_5$, & $e_2e_3=e_5$;   \\

${\mathbb M}_{09}$
&:&  $e_1e_2=e_4$, & $e_3e_4=e_5$;     \\

${\mathbb M}_{10}$ 
&:&  $e_1e_3=e_5$, & $e_2e_4=e_6$;  \\

${\mathbb M}_{11}$ 
&:&  $e_1e_2=e_3$, & $e_1e_3=e_4$, & $e_1e_4=e_6$, &$e_2e_3=e_6$, & $e_2e_5=e_6$;  \\

${\mathbb M}_{12}$ 
&:& $e_1e_2=e_3$, & $e_1e_3=e_4$, & $e_1e_4=e_6$,& $e_2e_5=e_6$; \\

${\mathbb M}_{13}$ 
&:& $e_1e_2=e_3$, & $e_1e_3=e_6$,& $e_4e_5=e_6$; \\

${\mathbb M}_{14}$ 
&:&  $e_1e_2=e_3$, &$e_1e_3=e_4$, &$e_1e_4=e_5$, & $e_1e_5=e_6$, &$e_2e_3=e_5$, &$e_2e_4=e_6$; \\

${\mathbb M}_{15}$ 
&:& $e_1e_2=e_3$, &$e_1e_3=e_4$, &$e_1e_4=e_5$, &$e_1e_5=e_6$, &$e_2e_3=e_6$;   \\

${\mathbb M}_{16}$
&:& $e_1e_2=e_3$, & $e_1e_3=e_4$, &$e_1e_4=e_5$, &$e_1e_5=e_6$; \\

${\mathbb M}_{17}$ 
&:& $e_1e_2=e_4$, & $e_1e_3=e_5$,  &$e_1e_4=e_6$, &$e_3e_5=e_6$;\\

${\mathbb M}_{18}$ 
&:& $e_1e_2=e_3$, & $e_1e_3=e_4$, &$e_1e_5=e_6$,  &$e_2e_3=e_5$, &$e_2e_4=e_6$;  \\

${\mathbb M}_{19}$ 
&:& $e_1e_2=e_3$,& $e_1e_3=e_5$, &   $e_1e_4=e_6$,& $e_2e_3=e_6$; \\

${\mathbb M}_{20}$ 
&:& $e_1e_2=e_3$,& $e_1e_3=e_5$,   &$e_1e_4=e_5$, &$e_2e_3=e_6$;\\

${\mathbb M}_{21}$ 
&:& $e_1e_2=e_3$, &$e_1e_3=e_5$, & $e_1e_4=e_6$;    \\

${\mathbb M}_{22}$ 
&:&  $e_1e_2=e_3$, &$e_1e_3=e_5$,&  $e_2e_4=e_6$;    \\

${\mathbb M}_{23}$ 
&:& $e_1e_2=e_3$, &$e_1e_3=e_5$,  &$e_1e_4=e_6$, &$e_2e_4=e_5$;\\

${\mathbb M}_{24}$ 
&:& $e_1e_2=e_5$, & $e_1e_3=e_6$,  &$e_3e_4=e_5$;    \\

${\mathbb M}_{25}$ 
&:& $e_1e_2=e_3$, & $e_1e_3=e_4$, &$e_1e_4=e_5$, &$e_2e_3=e_6$; \\

${\mathbb M}_{26}$ 
&:& $e_1e_2=e_4$, & $e_1e_3=e_5$, & $e_2e_3=e_6$;    \\

${\mathbb M}_{27}$ 
&:& $e_1e_2=e_5$,& $e_3e_4=  e_5$, & $e_3e_5=e_6$;  \\

${\mathbb M}_{28}$ 
&:& $e_1e_2=e_5$,  &$e_1e_4=  e_6$, & $e_3e_5=e_6$;    \\

${\mathbb M}_{29}$ 
&:& $e_1e_2=e_5$, &  $e_1e_4=  e_6$, & $e_3e_4=  e_5$, & $e_3e_5=e_6$; \\

${\mathbb M}_{30} (\epsilon)$
&:& $e_1e_2=e_4$, & $e_1e_3=e_5$, & $e_2e_5= e_6$, & $e_3e_4= \epsilon e_6$; \\

${\mathbb M}_{31}$ 
&:&  $e_1e_2=e_4$, & $e_1e_3=e_5$,  & $e_2e_4=e_6$,& $e_2e_5= e_6$, & $e_3e_4=-e_6$;\\

${\mathbb M}_{32}$ 
&:& $e_1e_2=e_4$, & $e_1e_3=e_5$,   &$e_1e_5=e_6$,  &$e_3e_4= e_6$; \\

${\mathbb M}_{33}$ 
&:& $e_1e_2=e_4$, & $e_2e_4= e_5$,   & $e_3e_4=  e_6$; \\

${\mathbb M}_{34}$ 
&:&  $e_1e_2=e_4$, &$e_1e_3=  e_5$,& $e_2e_4= e_5$, &  $e_3e_4=  e_6$; \\

${\mathbb M}_{35} (\epsilon)$     
&:& $e_1e_2=e_3$, &   $e_1e_3=  e_5$, &$e_1e_5=e_6$, &$e_2e_4=\epsilon e_5$,& $e_3e_4=e_6$; \\

${\mathbb M}_{36}$ 
&:& $e_1e_2=e_4$, &$e_1e_4 =  e_5$,  &$e_1e_5=  e_6$,& $e_2e_3=e_5$;\\

${\mathbb M}_{37}$ 
&:& $e_1e_2=e_4$,  & $e_1e_4 =  e_5$,  & $e_1e_5=  e_6$, & $e_2e_3=e_5$, & $e_2e_4=  e_6$;    \\

${\mathbb T}_{00}$ & : &  $e_1e_2=e_3$, &  $e_1e_3=e_4$, & $e_2e_4=e_5$;\\

${\mathbb T}_{01}$ & : &  
$e_1e_2=e_3$,& $e_1e_3=e_4$, & $e_1e_4=e_5$, & $e_2e_3=e_5$, & $e_2e_4=e_6$;\\

${\mathbb T}_{02}$ & : & 
$e_1e_2=e_3$,& $e_1e_3=e_4$,& $e_2e_3=e_5$,& $e_2e_4=e_6$;\\

${\mathbb T}_{03}$ & : &    
$e_1e_2=e_3$, & $e_1e_3=e_4$, & $e_1e_4=e_5$, & $e_2e_4=e_6$;\\

${\mathbb T}_{04}$ & : &    
$e_1e_2=e_3$, & $e_1e_3=e_4$,& $e_1e_5=e_6$,& $e_2e_4=e_6$;\\ 

${\mathbb T}_{05}$ & : &
$e_1e_2=e_3$, & $e_1e_3=e_4$ , & $e_2e_3=e_6$, & $e_4e_5=e_6$;\\

${\mathbb T}_{06}$ & : &
$e_1e_2=e_3$, & $e_1e_3=e_4$, & $e_2e_4=e_6$, & $e_3e_5=e_6$;\\

${\mathbb T}_{07}$ & : &
$e_1e_2=e_3$, & $e_1e_3=e_4$, &  $e_4e_5=e_6$;\\

${\mathbb T}_{08}$ & : &  
$e_1e_2=e_3$, & $e_1e_3=e_4$, & $e_1e_4=e_6$, & $e_1e_5=-e_6$, & $e_2e_3=e_5$, & $e_2e_4=e_6$;\\

${\mathbb T}_{09}(\alpha)$ & : &  
$e_1e_2=e_3$, & $e_1e_3=e_4$, & $e_1e_5=(\alpha+1) e_6$, & $e_2e_3=e_5$, & $e_2e_4=\alpha e_6$;\\

${\mathbb T}_{10}$ & : &  
$e_1e_2=e_3$, & $e_1e_3=e_6$, & $e_1e_4=e_5$, & $e_2e_3=e_5$, & $e_4e_5=e_6$;\\

${\mathbb T}_{11}$ & : & 
$e_1e_2=e_3$, & $e_1e_4=e_5$, & $e_1e_5=e_6$, &  $e_2e_3=e_5$;\\

${\mathbb T}_{12}$ & : & 
$e_1e_2=e_3$, & $e_1e_4=e_5$, & $e_1e_5=e_6$, &  $e_2e_3=e_5$, &  $e_2e_4=e_6$;\\

${\mathbb T}_{13}$ & : & 
$e_1e_2=e_3$, & $e_1e_4=e_5$, & $e_1e_5=e_6$, &  $e_2e_3=e_5$, & $e_3e_4=e_6$;\\

${\mathbb T}_{14}$ & : & 
$e_1e_2=e_3$, & $e_1e_4=e_5$, & $e_2e_3=e_5$, & $e_4e_5=e_6$;\\

${\mathbb T}_{15}$ & : & 
$e_1e_2=e_3$, & $e_1e_3=e_4$, & $e_1e_4=e_5$, & $e_1e_5=e_6$, & $e_2e_4=e_6$;\\

${\mathbb T}_{16}$ & : &  
$e_1e_2=e_3$, & $e_1e_3=e_4$, & $e_1e_4=e_5$, & $e_2e_5=e_6$;\\

$\mathbb{T}_{17}$ & : &
$e_1e_2=e_3$, & $e_1e_3=e_4$, & $e_1e_4=e_5$, & $e_2e_3=e_5$, & $e_2e_5=e_6$;\\

$\mathbb{T}_{18}(\alpha)$ & : &
$e_1e_2=e_3$, & $e_1e_3=e_4$, & $e_1e_4=e_5$, & $e_1e_5=(\alpha+1)e_6$, & $e_2e_3=e_5$, & $e_2e_4=\alpha e_6$;\\

${\mathbb T}_{19}$ & : & 
$e_1e_2=e_3$, & $e_1e_3=e_4$, & $e_1e_5=e_6$, & $e_2e_4=e_5$, & $e_3e_4=e_6$;\\

${\mathbb A}_{00}$ &:& $e_1e_2=e_3$,& $e_1e_3=e_4$,& $e_3e_4=e_5$; \\
${\mathbb A}_{01}$ &:& $e_1e_2=e_3$,& $e_1e_3=e_4$, & $e_2e_4=e_5$, & $e_3e_4=e_6$;\\
${\mathbb A}_{02}$ &:& $e_1e_2=e_3$,& $e_1e_3=e_4$, & $e_1e_4=e_5$,& $e_2e_3=e_5$, & $e_3e_4=e_6$;\\
${\mathbb A}_{03}$ &:& $e_1e_2=e_3$,& $e_1e_3=e_4$, &  $e_2e_3=e_5$, & $e_3e_4=e_6$;\\
${\mathbb A}_{04}$ &:& $e_1e_2=e_3$,& $e_1e_3=e_4$, & $e_1e_4=e_5$,&  $e_3e_4=e_6$; \\
${\mathbb A}_{05}$ &:&  
$e_1e_2=e_3$, & $e_1e_3=e_4$ , &  $e_2e_5=e_6$, & $e_4e_5=e_6$, & $e_3e_4=e_6$;  \\
${\mathbb A}_{06}$ &:&  
$e_1e_2=e_3$, & $e_1e_3=e_4$ ,  & $e_4e_5=e_6$, & $e_3e_4=e_6$;  \\
${\mathbb A}_{07}$ &:&  
$e_1e_2=e_3$, & $e_1e_3=e_4$ , &  $e_2e_5=e_6$, & $e_3e_4=e_6$;  \\
${\mathbb A}_{08}$ &:&  
$e_1e_2=e_3$, & $e_1e_3=e_4$ , &  $e_1e_5=e_6$, &  $e_3e_4=e_6$;\\
${\mathbb A}_{09}$ &:&  
$e_1e_2=e_4$, & $e_1e_3=e_5$, & $e_2e_3=e_6$, & $e_4e_5=e_6$;\\
${\mathbb A}_{10}$ &:&  
$e_1e_2=e_4$, & $e_1e_3=e_5$, & $e_4e_5=e_6$;\\
	${\mathbb A}_{11}$ &:&  
	$e_1e_2=e_3$, & $e_1e_3=e_4$, & $e_2e_3=e_5$, & $e_4e_5=e_6$; \\
	${\mathbb A}_{12}$ &:&  
	$e_1e_2=e_3$, & $e_1e_3=e_4$, & $e_2e_3=e_5$, & $e_2e_5=e_6$, & $e_3e_4=e_6$; \\
	${\mathbb A}_{13}$ &:&  
	$e_1e_2=e_3$, & $e_1e_3=e_4$, & $e_1e_5=e_6$, & $e_2e_3=e_5$, & $e_3e_4=e_6$; \\
	${\mathbb A}_{14}$ &:&  
		$e_1e_2=e_3$, & $e_1e_4=e_5$, & $e_2e_3=e_5$, & $e_2e_4=e_6$, & $e_3e_5=e_6$; \\
	${\mathbb A}_{15}$ &:&  
		$e_1e_2=e_3$, & $e_1e_4=e_5$, & $e_2e_3=e_5$, & $e_3e_5=e_6$;\\
	${\mathbb A}_{16}$ &:&  
	$e_1e_2=e_3$, & $e_1e_3=e_6$, & $e_1e_5=e_6$, & $e_2e_3=e_6$, & $e_3e_4=e_5$; \\
	${\mathbb A}_{17}$ &:&  
	$e_1e_2=e_3$, & $e_1e_3=e_6$, & $e_1e_5=e_6$, & $e_2e_3=e_6$, & $e_3e_4=e_5$, & $e_4e_5=e_6$; \\
	${\mathbb A}_{18}$ &:&  
	$e_1e_2=e_3$, & $e_1e_3=e_6$, & $e_1e_5=e_6$, & $e_2e_4=e_6$, & $e_3e_4=e_5$; \\
	${\mathbb A}_{19}$ &:&  
	$e_1e_2=e_3$, & $e_1e_3=e_6$, & $e_1e_5=e_6$, & $e_3e_4=e_5$; \\
	${\mathbb A}_{20}$ &:&  
	$e_1e_2=e_3$, & $e_1e_3=e_6$, & $e_1e_5=e_6$, & $e_3e_4=e_5$, & $e_4e_5=e_6$; \\
	${\mathbb A}_{21}$ &:&  
	$e_1e_2=e_3$, & $e_1e_3=e_6$, & $e_3e_4=e_5$, & $e_4e_5=e_6$; \\
	${\mathbb A}_{22}$ &:&  
	$e_1e_2=e_3$, & $e_1e_4=e_6$, & $e_3e_4=e_5$, & $e_3e_5=e_6$; \\
	${\mathbb A}_{23}$ &:&  
	$e_1e_2=e_3$, & $e_1e_5=e_6$, & $e_2e_4=e_6$, & $e_3e_4=e_5$; \\
	${\mathbb A}_{24}$ &:&  
	$e_1e_2=e_3$, & $e_1e_5=e_6$, & $e_2e_4=e_6$, & $e_3e_4=e_5$, & $e_3e_5=e_6$; \\
	${\mathbb A}_{25}$ &:&  
	$e_1e_2=e_3$, & $e_1e_5=e_6$, & $e_3e_4=e_5$; \\
	${\mathbb A}_{26}$ &:&  
	$e_1e_2=e_3$, & $e_1e_5=e_6$, & $e_3e_4=e_5$, & $e_3e_5=e_6$; \\
	${\mathbb A}_{27}$ &:&  
	$e_1e_2=e_3$, & $e_1e_5=e_6$, & $e_3e_4=e_5$, & $e_4e_5=e_6$; \\
	${\mathbb A}_{28}$ &:&  
	$e_1e_2=e_3$, & $e_3e_4=e_5$, & $e_3e_5=e_6$; \\
	${\mathbb A}_{29}$ &:&  
	$e_1e_2=e_3$, & $e_3e_4=e_5$, & $e_4e_5=e_6$;\\
 ${\mathbb A}_{30}$ &:&  
 $e_1e_2=e_3$, & $e_1e_3=e_4$, & $e_1e_4=e_5$, & $e_1e_5 = e_6$, & $e_3e_4 = e_6$;\\
 ${\mathbb A}_{31}$ &:&  
 $e_1e_2=e_3$, & $e_1e_3=e_4$, & $e_1e_4=e_5$, & $e_2e_3 = e_6$, & $e_2e_5 = e_6$, & $e_3e_4 = e_6$;\\
 ${\mathbb A}_{32}(\af)$ &:&  
 $e_1e_2=e_3$, & $e_1e_3=e_4$, & $e_1e_4=e_5$, & $e_2e_3 = \af e_6$, & $e_2e_5 = e_6$, & $e_4e_5 = e_6$;\\
 ${\mathbb A}_{33}$ &:&  
 $e_1e_2=e_3$, & $e_1e_3=e_4$, & $e_1e_4=e_5$, & $e_2e_3 = e_6$, & $e_4e_5 = e_6$;\\
 ${\mathbb A}_{34}$ &:&  
 $e_1e_2=e_3$, & $e_1e_3=e_4$, & $e_1e_4=e_5$, & $e_2e_4 = e_6$, & $e_2e_5 = -e_6$, & $e_3e_4 = e_6$;\\
 ${\mathbb A}_{35}(\af)$ &:&  
 $e_1e_2=e_3$, & $e_1e_3=e_4$, & $e_1e_4=e_5$, & $e_2e_4 = \af e_6$, & $e_2e_5 = e_6$, & $e_3e_5 = e_6$;\\
 ${\mathbb A}_{36}$ &:&  
 $e_1e_2=e_3$, & $e_1e_3=e_4$, & $e_1e_4=e_5$, & $e_2e_4 = e_6$, & $e_3e_5 = e_6$;\\
 ${\mathbb A}_{37}(\af)$ &:&  
 $e_1e_2=e_3$, & $e_1e_3=e_4$, & $e_1e_4=e_5$, & $e_2e_5 = \af e_6$, & $e_3e_4 = e_6$;\\
 ${\mathbb A}_{38}$ &:&  
 $e_1e_2=e_3$, & $e_1e_3=e_4$, & $e_1e_4=e_5$, & $e_3e_5 = e_6$;\\
 ${\mathbb A}_{39}$ &:&  
 $e_1e_2=e_3$, & $e_1e_3=e_4$, & $e_1e_4=e_5$, & $e_4e_5 = e_6$;\\
${\mathbb A}_{40}$ &:&  
$e_1e_2=e_3$, & $e_1e_3=e_4$, & $e_1e_4=e_5$, & $e_1e_5=e_6$, & $e_2e_3=e_5$, & $e_2e_5=e_6$, & $e_3e_4=-e_6$;\\
${\mathbb A}_{41}$ &:&  
$e_1e_2=e_3$, & $e_1e_3=e_4$, & $e_1e_4=e_5$, & $e_1e_5=e_6$, & $e_2e_3=e_5$, & $e_3e_4=e_6$;\\
${\mathbb A}_{42}$ &:&  
$e_1e_2=e_3$, & $e_1e_3=e_4$, & $e_1e_4=e_5$, & $e_2e_3=e_5$, & $e_2e_4=e_6$, & $e_3e_5=e_6$;\\
${\mathbb A}_{43}(\af)$ &:&  
$e_1e_2=e_3$, & $e_1e_3=e_4$, & $e_1e_4=e_5$, & $e_2e_3=e_5$, & $e_2e_5=\af e_6$, & $e_3e_4=e_6$;\\
${\mathbb A}_{44}$ &:&  
$e_1e_2=e_3$, & $e_1e_3=e_4$, & $e_1e_4=e_5$, & $e_2e_3=e_5$, & $e_3e_5=e_6$;\\
${\mathbb A}_{45}$ &:&  
$e_1e_2=e_3$, & $e_1e_3=e_4$, & $e_1e_4=e_5$, & $e_2e_3=e_5$, & $e_4e_5=e_6$;\\
${\mathbb A}_{46}$ &:&  
$e_1e_2=e_3$, & $e_1e_3=e_4$, & $e_1e_4=e_6$, & $e_1e_5=e_6$, & $e_2e_3=e_6$, & $e_2e_4=e_5$;\\
${\mathbb A}_{47}(\af)$ &:&  
$e_1e_2=e_3$, & $e_1e_3=e_4$, & $e_1e_4=\af e_6$, & $e_1e_5=e_6$, & $e_2e_4=e_5$, & $e_2e_5=e_6$, & $e_3e_4=e_6$, & $e_3e_5=e_6$;\\
${\mathbb A}_{48}(\af)$ &:&  
$e_1e_2=e_3$, & $e_1e_3=e_4$, & $e_1e_4=\af e_6$, & $e_1e_5=e_6$, & $e_2e_4=e_5$, & $e_3e_4=e_6$, & $e_3e_5=e_6$;\\
${\mathbb A}_{49}$ &:&  
$e_1e_2=e_3$, & $e_1e_3=e_4$, & $e_1e_4=e_6$, & $e_1e_5=e_6$, & $e_2e_4=e_5$;\\
${\mathbb A}_{50}$ &:&  
$e_1e_2=e_3$, & $e_1e_3=e_4$, & $e_1e_4=e_6$, & $e_1e_5=e_6$, & $e_2e_4=e_5$, & $e_2e_5=e_6$;\\
${\mathbb A}_{51}$ &:&  
$e_1e_2=e_3$, & $e_1e_3=e_4$, & $e_1e_4=e_6$, & $e_2e_4=e_5$, & $e_2e_5=e_6$;\\
${\mathbb A}_{52}$ &:&  
$e_1e_2=e_3$, & $e_1e_3=e_4$, & $e_1e_4=e_6$, & $e_2e_4=e_5$, & $e_2e_5=e_6$, & $e_3e_5=e_6$;\\
${\mathbb A}_{53}$ &:&  
$e_1e_2=e_3$, & $e_1e_3=e_4$, & $e_1e_4=e_6$, & $e_2e_4=e_5$, & $e_3e_5=e_6$;\\
${\mathbb A}_{54}$ &:&  
$e_1e_2=e_3$, & $e_1e_3=e_4$, & $e_1e_5=e_6$, & $e_2e_3=e_6$, & $e_2e_4=e_5$;\\
${\mathbb A}_{55}$ &:&  
$e_1e_2=e_3$, & $e_1e_3=e_4$, & $e_1e_5=e_6$, & $e_2e_3=e_6$, & $e_2e_4=e_5$, & $e_4e_5=e_6$;\\
${\mathbb A}_{56}$ &:&  
$e_1e_2=e_3$, & $e_1e_3=e_4$, & $e_2e_3=e_6$, & $e_2e_4=e_5$, & $e_4e_5=e_6$;\\
${\mathbb A}_{57}(\af)$ &:&  
$e_1e_2=e_3$, & $e_1e_3=e_4$, & $e_1e_5=(\af+1)e_6$, & $e_2e_4=e_5$, & $e_3e_4=\af e_6$;\\
${\mathbb A}_{58}(\af)$ &:&  
$e_1e_2=e_3$, & $e_1e_3=e_4$, & $e_1e_5=(\af+1)e_6$, & $e_2e_4=e_5$, & $e_2e_5=e_6$, & $e_3e_4=\af e_6$;\\
${\mathbb A}_{59}$ &:&  
$e_1e_2=e_3$, & $e_1e_3=e_4$, & $e_1e_5=e_6$, & $e_2e_4=e_5$, & $e_2e_5=e_6$, & $e_3e_4=e_6$;\\
${\mathbb A}_{60}$ &:&  
$e_1e_2=e_3$, & $e_1e_3=e_4$, & $e_1e_5=e_6$, & $e_2e_4=e_5$, & $e_4e_5=e_6$;\\
${\mathbb A}_{61}$ &:&  
$e_1e_2=e_3$, & $e_1e_3=e_4$, & $e_2e_4=e_5$, & $e_2e_5=e_6$;\\
${\mathbb A}_{62}$ &:&  
$e_1e_2=e_3$, & $e_1e_3=e_4$, & $e_2e_4=e_5$, & $e_2e_5=e_6$, & $e_3e_5=e_6$;\\
${\mathbb A}_{63}$ &:&  
$e_1e_2=e_3$, & $e_1e_3=e_4$, & $e_2e_4=e_5$, & $e_3e_5=e_6$;\\
${\mathbb A}_{64}$ &:&  
$e_1e_2=e_3$, & $e_1e_3=e_4$, & $e_2e_4=e_5$, & $e_3e_5=e_6$, & $e_4e_5=e_6$;\\
${\mathbb A}_{65}$ &:&  
$e_1e_2=e_3$, & $e_1e_3=e_4$, & $e_2e_4=e_5$, & $e_4e_5=e_6$;\\

${\mathbb A}_{66}$ & : &  
$e_1e_2=e_3$, & $e_1e_3=e_4$,& $e_1e_4=e_6$, & $e_2e_3=e_6$, & $e_2e_5=e_6$, & $e_3e_4=e_5$;\\
${\mathbb A}_{67}$ & : &  
$e_1e_2=e_3$, & $e_1e_3=e_4$,& $e_1e_4=e_6$, & $e_2e_5=e_6$, & $e_3e_4=e_5$;\\
${\mathbb A}_{68}$ & : &  
$e_1e_2=e_3$, & $e_1e_3=e_4$,& $e_1e_4=e_6$, & $e_2e_5=e_6$, & $e_3e_4=e_5$, & $e_3e_5=e_6$;\\
${\mathbb A}_{69}$ & : &  
$e_1e_2=e_3$, & $e_1e_3=e_4$,& $e_1e_4=e_6$, & $e_3e_4=e_5$, & $e_3e_5=e_6$;\\
${\mathbb A}_{70}(\af)$ & : &  
$e_1e_2=e_3$, & $e_1e_3=e_4$, & $e_1e_5=e_6$, & $e_2e_3=\af e_6$, & $e_2e_4=e_6$, & $e_3e_4=e_5$;\\
${\mathbb A}_{71}$ & : &  
$e_1e_2=e_3$, & $e_1e_3=e_4$, & $e_1e_5=e_6$, & $e_2e_3= e_6$, & $e_3e_4=e_5$;\\
${\mathbb A}_{72}$ & : &  
$e_1e_2=e_3$, & $e_1e_3=e_4$, & $e_2e_3= e_6$, & $e_2e_5=e_6$, & $e_3e_4=e_5$;\\
${\mathbb A}_{73}$ & : &  
$e_1e_2=e_3$, & $e_1e_3=e_4$, & $e_2e_3= e_6$, & $e_2e_5=e_6$, & $e_3e_4=e_5$, & $e_4e_5=e_6$;\\
${\mathbb A}_{74}$ & : &  
$e_1e_2=e_3$, & $e_1e_3=e_4$, & $e_2e_3= e_6$, & $e_3e_4=e_5$, & $e_3e_5=e_6$, & $e_4e_5=e_6$;\\
${\mathbb A}_{75}$ & : &  
$e_1e_2=e_3$, & $e_1e_3=e_4$, & $e_2e_3= e_6$, & $e_3e_4=e_5$, & $e_4e_5=e_6$;\\
${\mathbb A}_{76}$ & : &  
$e_1e_2=e_3$, & $e_1e_3=e_4$, & $e_1e_5=e_6$, & $e_2e_4= e_6$, & $e_3e_4=e_5$, & $e_3e_5=e_6$;\\
${\mathbb A}_{77}$ & : &  
$e_1e_2=e_3$, & $e_1e_3=e_4$, & $e_2e_4= e_6$, & $e_3e_4=e_5$, & $e_3e_5=e_6$;\\
${\mathbb A}_{78}$ & : &  
$e_1e_2=e_3$, & $e_1e_3=e_4$, & $e_1e_5=e_6$, & $e_3e_4=e_5$;\\
${\mathbb A}_{79}$ & : &  
$e_1e_2=e_3$, & $e_1e_3=e_4$, & $e_1e_5=e_6$, & $e_3e_4=e_5$, & $e_3e_5=e_6$;\\
${\mathbb A}_{80}$ & : &  
$e_1e_2=e_3$, & $e_1e_3=e_4$, & $e_2e_5=e_6$, & $e_3e_4=e_5$;\\
${\mathbb A}_{81}$ & : &  
$e_1e_2=e_3$, & $e_1e_3=e_4$, & $e_2e_5=e_6$, & $e_3e_4=e_5$, & $e_3e_5=e_6$;\\
${\mathbb A}_{82}(\af)$ & : &  
$e_1e_2=e_3$, & $e_1e_3=e_4$, & $e_2e_5=\af e_6$, & $e_3e_4=e_5$, & $e_3e_5=e_6$, & $e_4e_5=e_6$;\\
${\mathbb A}_{83}$ & : &  
$e_1e_2=e_3$, & $e_1e_3=e_4$, & $e_2e_5=e_6$, & $e_3e_4=e_5$, & $e_4e_5=e_6$;\\
${\mathbb A}_{84}$ & : &  
$e_1e_2=e_3$, & $e_1e_3=e_4$, & $e_3e_4=e_5$, & $e_3e_5=e_6$;\\
${\mathbb A}_{85}$ & : &  
$e_1e_2=e_3$, & $e_1e_3=e_4$, & $e_3e_4=e_5$, & $e_4e_5=e_6$.
\end{longtable}

}

\

\ 

\

\end{document}